\documentclass[10pt]{article}

\usepackage{amsmath}
\usepackage{url}
\usepackage{amsthm}
\usepackage{authblk}
\usepackage{amssymb}
\usepackage{amsfonts}
\usepackage{graphicx}
\usepackage{verbatim}
\usepackage{mathrsfs}
\usepackage{subfigure}
\usepackage{fancyhdr}
\usepackage{latexsym}
\usepackage{graphicx}
\usepackage{dsfont}
\usepackage{color}
\usepackage{fancybox}
\usepackage{framed, color}
\definecolor{shadecolor}{rgb}{1, 0.8, 0.3}
\usepackage{algorithm}
\usepackage{algpseudocode}

\usepackage {blindtext}
\usepackage{geometry}
\usepackage{fancyhdr}
\usepackage{amsmath , amsthm , amssymb}
\usepackage{amsfonts}
\usepackage{graphicx}
\usepackage{hyperref}
\usepackage{enumitem}

\normalfont
\usepackage[T1]{fontenc}

\usepackage{relsize}

\usepackage[utf8]{inputenc}
\usepackage[english]{babel}
\usepackage{fancyhdr}
\usepackage{enumitem,graphicx,appendix}
\usepackage{amsmath,amssymb,amsthm}
\usepackage{xargs}
\usepackage{marginnote,color}

\SetLabelAlign{Center}{\hfil \mbox({#1}) \hfil}

\title{
Chaotic motions in the restricted four body problem \\ via 
Devaney's saddle-focus homoclinic tangle theorem
}

\author[1]{Shane Kepley \thanks{S.K.
partially supported by NSF grant DMS-1700154 and by  the Alfred P. Sloan Foundation grant G-2016-7320.
Email: {\tt skepley@my.fau.edu}} }

\author[2]{J.D. Mireles James \thanks{J.D.M.J 
partially supported by NSF grant DMS-1700154 and by  the Alfred P. Sloan Foundation grant G-2016-7320.
Email: {\tt jmirelesjames@fau.edu}}}

\affil[1,2]{Florida Atlantic University, Department of Mathematical Sciences}
\date{\today}

\newcommand{\rr}{\mathbb{R}}

\newcommand{\cc}{\mathbb{C}}

\newcommand{\intrr}{\mathbb{IR}}
\newcommand{\overbar}[1]{\mkern 1.5mu\overline{\mkern-1.5mu#1\mkern-1.5mu}\mkern 1.5mu}

\newcommand{\norm}[1]{\left|\left| #1 \right|\right|}

\theoremstyle{remark}
\newtheorem{remark}{Remark}

\newtheorem{theorem}{Theorem}[section]
\newtheorem{lemma}[theorem]{Lemma}
\newtheorem{prop}[theorem]{Proposition}
\newtheorem{cor}[theorem]{Corollary}

\begin{document}

\maketitle

\begin{abstract}
We prove the existence of chaotic motions in a planar
restricted four body problem, establishing that the 
system is not integrable.  The idea of the proof
is to verify the hypotheses of a 
topological forcing theorem.  The forcing theorem 
applies to two freedom Hamiltonian systems where the stable and unstable 
manifolds of a saddle-focus equilibrium intersect transversally 
in the energy level set of the equilibrium. 
We develop a mathematically rigorous computer assisted
argument which verifies the hypotheses of the forcing theorem, and we
implement our approach for the restricted four body problem.   
The method is constructive and works far from any perturbative 
regime and for non-equal masses. 
Being constructive, our argument results in useful 
byproducts like information about the 
locations of transverse connecting orbits, 
quantitative information about the invariant manifolds,
and upper/lower bounds on transport times.    
\end{abstract}

\newpage

\tableofcontents

\newpage 

\section{Introduction}
Gravitational $N$-body problems occupy a central place in 
classical mathematical physics.  Studying  
their long time behavior raises subtle questions
about the interplay 
between regular and irregular motions and the boundary 
between integrable and chaotic dynamics.  Over the last 
hundred years
concepts from the qualitative theory of dynamical 
systems -- concepts like stable/unstable manifolds, homoclinic and 
heteroclinic tangles, KAM theory, and whiskered invariant tori --
have come to play an increasingly important
role in the discussion.  
In the last fifty years the study of numerical methods
for computing invariant objects has matured into a 
thriving sub-discipline.  This growth is
driven at least in part by the needs of the worlds space programs.
More recent work on validated numerical methods
has begun to unify the computational and 
analytical perspectives, enriching both aspects of the subject.

In this paper 
we prove that a certain gravitational four body problem
is not integrable by establishing the existence of chaotic motions.  
As a byproduct of our proof we obtain 
detailed information about the embedding of the stable and unstable 
manifolds, the existence of a number of 
transverse homoclinic connecting orbits, and  
accurate information about the location and other properties of these orbits.  
The problem we study is referred to as the planar circular restricted 
four body problem (or just the CRFBP), and is discussed in Section \ref{sec:intro}.

Our proof of the existence of chaotic motions in the 
CRFBP proceeds by verifying the hypotheses of the following 
topological forcing theorem, due to Devaney 
\cite{MR0442990}.  The theorem exploits the fact that, while the stable and unstable 
manifolds attached to an equilibrium solution cannot intersect 
transversally in the phase space, 
it is possible for them to intersect transversally when restricted to an energy manifold
of a Hamiltonian system.

\begin{theorem}[Theorem A of \cite{MR0442990} -- Hamiltonian saddle-focus homoclinic tangle theorem] \label{thm:devaney}
{\em
Suppose that $f \colon \mathbb{R}^4 \to \mathbb{R}^4$ 
is a two degree of freedom Hamiltonian vector field.  
Assume that $\mathbf{x}_0 \in \mathbb{R}^4$ is a saddle-focus equilibrium 
solution for $f$, and 
that $\gamma$ is a transverse homoclinic orbit for $\mathbf{x}_0$.  Then for any local transverse 
section $\Sigma$ to $\gamma$, and for any positive integer $N$, there is a compact, 
invariant, hyperbolic set $\Lambda_N \subset \Sigma$ on 
which the Poincar\'{e} map is topologically conjugate to the Bernoulli 
shift on $N$ symbols. 
}
\end{theorem}

\noindent Our main result is the following theorem.

\begin{theorem}[A CRFBP satisfying the hypotheses of Theorem \ref{thm:devaney}] \label{thm:ourThm1}
{\em
The planar circular restricted four body problem with mass parameters $m_1 = 0.5$, $m_2 = 0.3$, and 
$m_3 = 0.2$ satisfies the hypotheses of Theorem \ref{thm:devaney}. }
\end{theorem}

\noindent The proof of the theorem is both constructive and non-perturbative, and 
the argument makes substantial use of the digital computer.  

By non-perturbative we refer to the fact that if $m_2, m_3 \ll 1$ then 
the CRFBP is a perturbation of the Kepler problem, while
if only $m_3 \ll 1$ then the CRFBP is  
a perturbation of the circular restricted three body problem. 
Previous studies -- of both the pen and paper and computer
assisted variety -- establish the existence of hyperbolic chaotic invariant sets in 
the three body problem. Such results provide information about the four 
body for small $m_3$ via perturbative considerations.  See the references discussed 
below for more complete discussion of perturbative arguments.  
In the non-perturbative regime on the other hand, there are to our knowledge
no previous mathematically rigorous proofs of chaos in the CRFBP. 
Moreover, while the results of the present work establish the existence of
chaotic motions only for the parameter values given in Theorem \ref{thm:ourThm1}, 
our methods work \textit{in principle} for any values of $m_1, m_2, m_3$ giving 
rise to a saddle focus equilibrium.  
 
Note also that  the much studied planar circular restricted three body problem
does not admit any saddle-focus equilibrium 
solutions: the equilibrium in that problem are typically of saddle $\times$ center 
or center $\times$ center stability type.
Then the dynamical mechanism for chaos considered 
in the present work is properly a four body phenomena.  
It seems that the Devaney theorem has not been exploited in 
any previous work on computer assisted proofs.

Both the strategy of the proof
and the organization of the paper are outlined in Section \ref{sec:ideaOfTheProof}, 
but the main tools are mathematically rigorous numerical 
methods for studying invariant manifolds of analytic differential equations.  
Its worth mentioning two novelties of the present work  
before beginning the more detailed
discussion of the problem and approach. 
\begin{itemize}
\item \textbf{Automatic transversality:} the computer assisted proof of 
Theorem \ref{thm:ourThm1} exploits Theorem \ref{prop:finiteDimNewton}, which is a result of 
Newton-Kantorovich type.  
A feature of the Newton-like argument is
that, whenever the hypotheses are verified, we automatically obtain 
non-degeneracy of the solution (in the sense that 
a certain Jacobian matrix is invertible).  Then in Section \ref{sec:ideaOfTheProof} we
prove that the transversality hypothesis of Devaney's theorem are obtained from 
this non-degeneracy \textit{for free} (or almost for free -- we only have to check
that the gradient of the energy at the homoclinic point is not zero).  
Earlier transversality results of this kind 
appear in \cite{MR3207723, MR3068557}, however these previous 
results do not apply in an energy section.  
\item \textbf{Rigorous lower bounds on transfer times:} the argument of the present 
work exploits some validated numerical methods for growing atlases for  
stable/unstable manifolds which were recently developed by the authors in \cite{manifoldPaper1}.
Because of the way we systematically grow the manifold atlas, we are also able to 
rule out connecting orbits.  
The ability to rule out orbits leads to lower bounds on transport times.
This aspect of our approach is 
discussed further in Section  \ref{sec:boundtransporttime}. 
\end{itemize}

\subsection{The equations of motion for the CRFBP}\label{sec:intro}

\begin{figure}[!h]
\centering
\includegraphics[width=3in,height=3in]{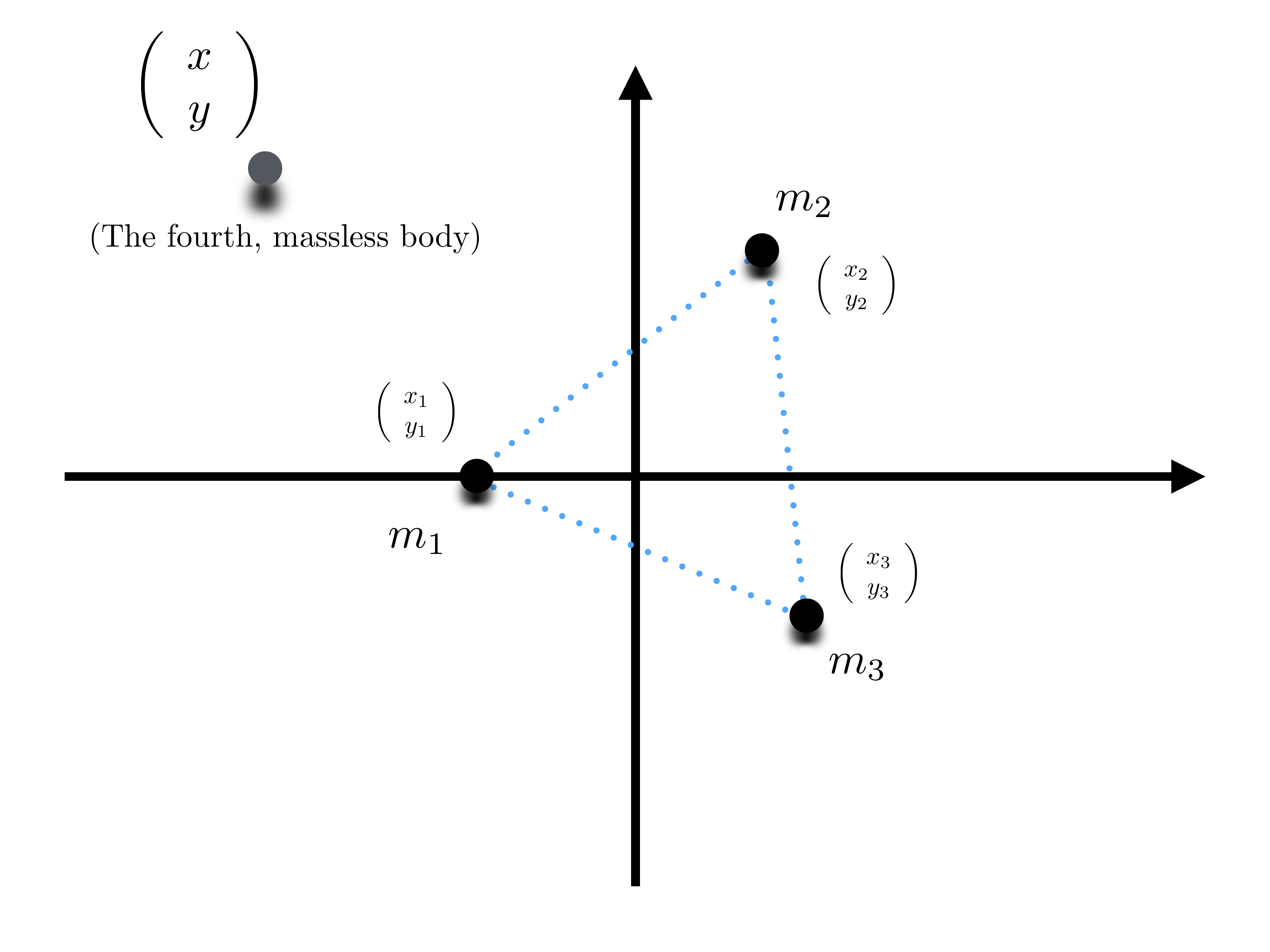}
\caption{Configuration space for the CRFBP: The three primaries 
with masses $m_1,m_2,$ and $m_3$ are arranged in the equilateral triangle configuration of Lagrange.
The equilateral configuration is a relative equilibrium solution of the 
three body problem.  After shifting the center of mass to 
the origin, removing the velocity of the center, 
 and changing to a co-rotating coordinate frame, 
 we study the dynamics of a fourth massless 
particle like an asteroid or space craft.  
The vector field governing the motion of the massless particle 
is given in 
Equation \eqref{eq:SCRFBP}.}\label{rotatingframe}
\end{figure}

Consider three particles with masses $0 < m_3 \leq m_2 \leq m_1 < 1$,
normalized so that 
\[
m_1 + m_2 + m_3 = 1.
\]
Assume the masses are located at the vertices of a planar equilateral triangle, 
rotating with constant angular velocity.  That is, we assume that the three massive 
bodies, which we call the ``primaries'', are in the triangular configuration of Lagrange.   
We choose a co-rotating coordinate frame which fixes
the barycenter of the triangle at the origin, 
the first primary on the negative $x$-axis, and causes the positive $x$-axis to cut through the opposite side of the triangle.  
Once in co-rotating coordinates we 
are interested in the dynamics of a fourth, massless particle
$p$ with coordinates $(x,y)$ moving in the 
gravitational field of the primaries.   
The situation is illustrated in Figure \ref{rotatingframe}.

We write $(x_1, y_1)$, $(x_2, y_2)$ and $(x_3, y_3)$
to denote the 
locations of the primary masses.  
Explicit formulas for the values of 
$(x_j, y_j)$, $j = 1,2,3$
as a function of the masses are are 
recorded in Appendix \ref{sec:CRFBP_basics}.
Define the potential function 
\begin{equation} \label{eq:CRFBP_potential}
\Omega(x,y) :=
\frac{1}{2} (x^2 + y^2) + \frac{m_1}{r_1(x,y)} + \frac{m_2}{r_2(x,y)} + \frac{m_3}{r_3(x,y)}, 
\end{equation}
where 
\begin{equation} \label{eq:def_r}
r_j(x,y) := \sqrt{(x-x_j)^2 + (y-y_j)^2},  \quad \quad \quad j = 1,2,3.
\end{equation}
%
%
%
Let $\mathbf{x} = (x, \dot x, y, \dot y) \in \mathbb{R}^4$ denote the
state of the system.   
The equations of motion in the rotating frame are 
\[
\mathbf{x}' = f(\mathbf{x}),
\]
where
\begin{equation}\label{eq:SCRFBP}
 f(x, \dot x, y, \dot y) := 
\left(
\begin{array}{c}
\dot x \\
2 \dot y + \Omega_x(x, y) \\
\dot y \\
-2 \dot x + \Omega_y(x, y) \\
\end{array}
\right).
\end{equation}
The explicit form of the functions $\Omega_{x,y}$, along with some
 other useful formulas for the CRFBP,  
 are given in  Appendix \ref{sec:CRFBP_basics}.

The system conserves the quantity   
\begin{eqnarray} \label{eq:CRFB_energy}
E(x, \dot x, y, \dot y) &= 
\frac{1}{2}\left( {\dot x}^2 + {\dot y}^2 \right) - \Omega(x,y),
 \label{eq:energy}
\end{eqnarray}
which is called the \textit{Jacobi integral}.   
We refer to $E$ as ``the energy'' of the CRFBP,
though strictly speaking the mechanical energy of the system is 
$- E$.  Note that $E$ is continuous (in fact real analytic) away from the 
primaries.  

Equation \eqref{eq:SCRFBP} defines a complex valued 
vector field $f$ on the set 
\begin{equation} \label{eq:defDomain_U}
U := \left\{ (x, \dot x, y, \dot y) \in \mathbb{C}^4 \, : \, 
(x,y) \neq (x_{j},y_{j}) \text{ for } j=1,2,3 \right\},
\end{equation}
just by letting the arguments take complex values.  In any 
contractible subdomain of $U$ it is possible to define 
a single valued branch of the square root, and 
on any such domain $f$ is analytic.

Before discussing our results in detail it is reasonable to say a few words 
about related work on the problem.
The literature on the restricted four body problem is rich, and a thorough scholarly 
review would take us far afield.
The interested reader will find many other interesting references by consulting the works cited  
below.

The first modern study of the CRFBP focused on the number, location, and stability of
the equilibrium configurations \cite{MR510556}.  Many other interesting 
questions about equilibrium configurations remain open, and there has been 
sustained interest this topic.  See for example the works of 
\cite{MR2232439, MR2845212, MR3176322, MR2784870, MR3463046}.
Further studies examine periodic solutions of the CRFBP
from both theoretical and numerical viewpoints 
\cite{MR3500916, MR3554377, MR3571218}.
The studies of \cite{MR2917610, maximePOman} considers also the stable/unstable manifolds 
attached to periodic orbits.

More global considerations such as connecting orbits and chaotic 
dynamics are studied from a numerical perspective in 
\cite{MR3304062, MR2596303, MR2013214, maximePOman}.
Theoretical/pen and paper works on heteroclinic/homoclinic orbits and chaotic 
motions are found in the works of 
\cite{MR3105958, MR3038224, MR3158025, MR3626383}, and use 
perturbative methods to establish the existence of complex 
dynamics.  The work of \cite{MR3239345} considers regularization 
of collisions with the primary bodies.  
A Hill's approximation is developed in 
\cite{MR3346723}.

Finally, it must be remarked that readers interested in methods of
computer assisted proof as tools for studying invariant manifolds and 
connecting orbits in  celestial mechanics problems should also consult the 
works of \cite{MR1947690, MR2112702, MR2785975, capinski_roldan, MR3622273, 
MR1961956, MR3443692, MR3032848, fourierAutomaticDiff} 
on the dynamics of the circular three body problem, 
as well as to the computer assisted work of 
\cite{MR2185163, MR2012847, MR2259202, MR2312391}
on gravitational $N$-body problems.  This list is of course
far from exhaustive, and the reader will find many additional 
references when consulting these works.

\subsection{The proof of Theorem \ref{thm:ourThm1}} \label{sec:ideaOfTheProof}
The strategy behind the proof of Theorem \ref{thm:ourThm1} is this:
first we formulate a zero finding problem whose non-degenerate roots correspond to transverse
homoclinic connections for an equilibrium solution of the CRFBP.
Next we compute a numerical solution of this zero finding problem.  
Finally we develop a-posteriori analysis and 
validated numerical methods which establish 
the existence of a true zero 
near the approximate numerical solution. 
In the present section we derive an appropriate zero finding problem.

\begin{figure}[t!]
\center{
  \includegraphics[width=0.5\linewidth]{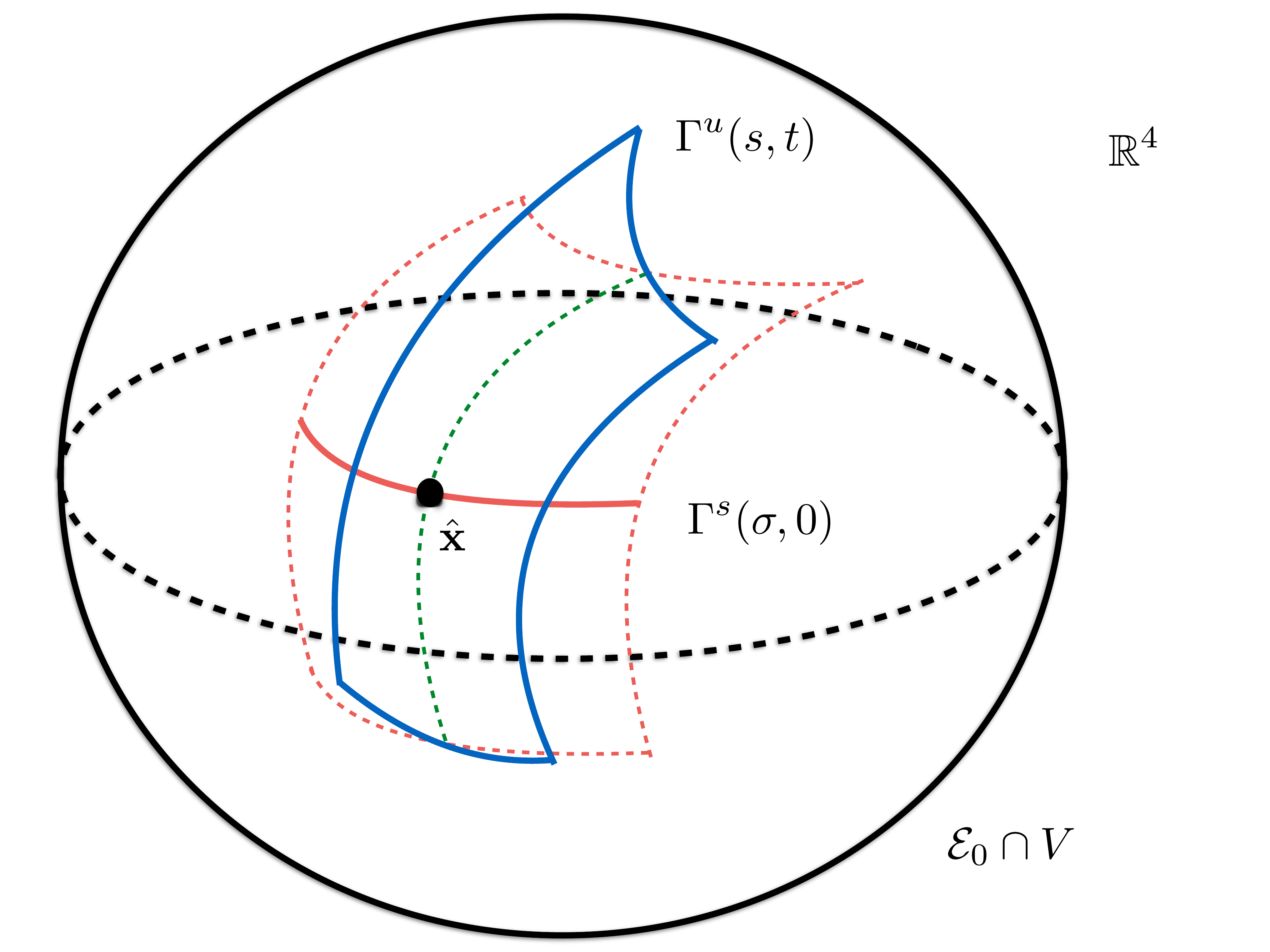}
}
\caption{
\textbf{Schematic illustration of the zero finding problem for the 
connecting orbit}: If $\mathcal{E}_0$ is the three dimensional energy section 
of the equilibrium $\mathbf{x}_0$ (not pictured) and $U$ is a ball in $\mathbb{R}^4$ 
then the figure illustrates a transverse intersection of the 
stable and unstable manifolds at $\hat{\mathbf{x}} \in U \cap \mathcal{E}_0$.
The parameterizations $\Gamma^{s,u} \colon [-1,1]^2 \to \mathbb{R}^4$ are
charts for some neighborhoods of the stable and unstable manifolds $W^{s,u}(\mathbf{x}_0)$.
These are represented respectively by the red and blue surface elements.
By restricting $\Gamma^s$ to its zero section -- shown as the solid red curve --
we isolate $\hat{\mathbf{x}}$. The desired homoclinic is the orbit through $\hat{\mathbf{x}}$ pictured as the dotted green line in the figure. The orbit lies in the transverse intersection of the stable/unstable manifolds, though
the transversality is only relative to the energy section.
} \label{fig:connection}
\end{figure}

Let $f \colon U \subset \mathbb{R}^4 \to \mathbb{R}^4$ denote the CRFBP vector field 
defined in Equation \eqref{eq:SCRFBP},  $\mathbf{x}_0 \in \mathbb{R}^4$ be an equilibrium solution, 
and let $\Phi: U \times \rr \to \rr^4$ denote the flow generated by $f$. Suppose that $\Gamma^u, \Gamma^s \colon [-1, 1]^2 \subset \mathbb{R}^2 \to \mathbb{R}^4$
are smooth maps with 
\[
\Gamma^u([-1,1]^2) \subset W^u(\mathbf{x}_0), 
\quad \quad \quad \mbox{and} \quad \quad \quad 
\Gamma^s([-1,1]^2) \subset W^s(\mathbf{x}_0).
\]
Components of the charts are denoted by $\Gamma^{s,u}_{j}$, for $j = 1,2,3,4$.
Assume that $\Gamma^{s,u}$ are \textit{well aligned} with the flow, 
in the sense that  
\begin{equation} \label{eq:wellAlligned}
\Phi(\Gamma^{s,u}(s, t_1), t_2/\tau_{s,u}) = \Gamma^{s,u}(s, t_1 + t_2),
\end{equation}
for all $t_1, t_2$ so that $t_1 + t_2 \in [-1,1]$. Here 
$\tau_{s,u} > 0$  are reparameterizations of time associated with 
the unstable/stable charts.
The following Lemma defines the appropriate zero finding problem.
The geometric meaning of the problem is illustrated in Figure \ref{fig:connection}.

\begin{lemma} \label{lem:aPosEnergyLemma1} {\em
Define the function $G \colon \times [-1,1]^3 \to \mathbb{R}^3$
by 
\[
G(s, t, \sigma) := 
\left(
\begin{array}{c}
\Gamma_1^u(s,t) - \Gamma_1^s(\sigma, 0) \\
\Gamma_2^u(s,t) - \Gamma_2^s(\sigma, 0) \\
\Gamma_3^u(s,t) - \Gamma_3^s(\sigma, 0) \\
\end{array}
\right),
\]
and suppose that $(\hat s, \hat t, \hat \sigma) \in [-1, 1]^3$ has
$G(\hat s, \hat t, \hat \sigma) = 0$.  
If $\Gamma_4^u(\hat s, \hat t)$ and $\Gamma_4^s(\hat \sigma, 0)$ have the same
sign, 
then $\hat{\mathbf{x}} := \Gamma^u(\hat s, \hat t)$ is
homoclinic to $\mathbf{x}_0$.}
\end{lemma}

\begin{proof}
First note that
\[
\Gamma^s(s, t) = \Gamma^u(\sigma, \tau), 
\]
implies that 
\[
\hat{\mathbf{x}} = \Gamma^s(s, t) \in W^s(\mathbf{x}_0) \cap W^u(\mathbf{x}_0), 
\]
and hence $\hat{\mathbf{x}}$ is homoclinic to $\mathbf{x}_0$, 
simply because $\Gamma^{s,u}$ are charts for the stable/unstable manifolds. 
Restricting to the $\tau = 0$ section for $\Gamma^u$ isolates an intersection point.

Now suppose that $(\hat s, \hat t, \hat \sigma) \in [-1, 1]^3$ and that 
$G(\hat s, \hat t, \hat \sigma) = 0$.
Then $\Gamma^u(\hat s, \hat t)$ and $\Gamma^s(\hat \sigma, 0)$
already agree in their first three components,
and we focus on the fourth component.
Define the numbers
\[
x^u = \Gamma_1^u(\hat s, \hat t), 
\quad \quad 
\dot x^u =  \Gamma_2^u(\hat s, \hat t), 
\quad \quad 
y^u =  \Gamma_3^u(\hat s, \hat t),
\]
and 
\[
x^s = \Gamma_1^s(\hat \sigma, 0), 
\quad \quad 
\dot x^s =  \Gamma_2^s(\hat \sigma, 0), 
\quad \quad 
y^s =  \Gamma_3^s(\hat \sigma, 0).
\]
Since $G(\hat s, \hat t, \hat \sigma) = 0$ we have that 
\begin{equation} \label{eq:theThreeEqualities}
x^u = x^s, \quad \quad \quad
\dot x^u = \dot x^s, \quad \quad \quad 
\mbox{and} \quad \quad \quad 
y^u = y^s.
\end{equation}
Let 
\[
C_0 = E(\mathbf{x}_0),
\]
denote the energy of the equilibrium.  Then 
\begin{equation} \label{eq:equalEnergies}
E(x^u, \dot x^u, y^u,  \Gamma_4^u(\hat s, \hat t)) = 
E(x^s, \dot x^s, y^s,  \Gamma_4^s(\hat \sigma, 0)) = C_0.
\end{equation}
This is because points in the stable/unstable manifolds have the same energy as $\mathbf{x}_0$,
a fact which follows from the continuity of the function $E$ at $\mathbf{x}_0$.

Recalling the formula for the energy functional given in Equation 
\eqref{eq:CRFB_energy}, Equation \eqref{eq:equalEnergies} becomes 
\[
\frac{1}{2}\left( \left({\dot x^u}\right)^2 + { \Gamma_4^u(\hat s, \hat t)}^2 \right) - \frac{1}{2} (\left({x^u}\right)^2 + \left({y^u}\right)^2)
- \frac{m_1}{r_1(x^u, y^u)} - \frac{m_2}{r_2(x^u, y^u)} - \frac{m_3}{r_3(x^u, y^u)}
= 
\]
\[
\frac{1}{2}\left( \left({\dot x^s}\right)^2 + {\Gamma_4^s(\hat \sigma, 0)}^2 \right) - \frac{1}{2} (\left({x^s}\right)^2 + \left({y^s}\right)^2)
- \frac{m_1}{r_1(x^s, y^s)} - \frac{m_2}{r_2(x^s, y^s)} - \frac{m_3}{r_3(x^s, y^s)},
\]
and employing the equalities of Equation \eqref{eq:theThreeEqualities}
reduces this to 
\[
 \Gamma_4^u(\hat s, \hat t)^2 = \Gamma_4^s(\hat \sigma, 0)^2,
\]
or 
\[
\Gamma_4^u(\hat s, \hat t) = \pm {\Gamma_4^s(\hat \sigma, 0)}.
\]
But $\Gamma_4^u(\hat s, \hat t)$ and $\Gamma_4^s(\hat \sigma, 0)$ 
have the same sign by hypothesis, and we conclude they are actually equal.
It follows that 
\[
\Gamma^u(\hat s, \hat t) = \Gamma^s(\hat \sigma, 0),
\]
in all four components
and hence the orbit of $\hat{\mathbf{x}}$ is homoclinic to $\mathbf{x}_0$
as desired. 
\end{proof}

The Devaney theorem requires more than the existence of a 
homoclinic solution.  We need to verify that 
the stable and unstable manifolds intersect transversally in the 
energy manifold.  
The following lemma provides a simple a-posteriori
condition verifying the transversality.  

\begin{lemma} \label{lem:aPosEnergyLemma2} {\em
Suppose that $G \colon [-1, 1]^3 \to \mathbb{R}^3$ is as in Lemma \ref{lem:aPosEnergyLemma1}
and that $(\hat s, \hat t, \hat \sigma) \in (-1, 1)^3$ has $G(\hat s, \hat t, \hat \sigma) = 0$.  
Assume in addition that $\Gamma_4^u(\hat s, \hat t)$ and $\Gamma_4^s(\hat \sigma, 0)$ have the same
sign, and let $\hat{\mathbf{x}} = \Gamma^u(\hat s, \hat t) = 
\Gamma^s(\hat \sigma, 0)$ denote the resulting 
homoclinic point, which exists by Lemma \ref{lem:aPosEnergyLemma1}.
Assume that $DG(\hat s, \hat t, \hat \sigma)$ is nonsingular,
and that 
\[
\nabla E(\hat{\mathbf{x}}) \neq 0.
\]
Then the energy level set is a smooth $3$-manifold near $\hat{\mathbf{x}}$, and
the stable/unstable manifolds of $\mathbf{x}_0$ intersect transversally 
at $\hat{\mathbf{x}}$ when restricted to the energy manifold.  }
\end{lemma}

\begin{proof}
Let
\[
C_0 = E(\mathbf{x}_0),
\]
once again denote the equilibrium energy and define
\[
\mathcal{E}_0 := \left\{ \mathbf{x} \in \mathbb{R}^4 \, : E(\mathbf{x}) = C_0 \right\},
\]
the energy level set of $C_0$.  
It is an elementary fact from differential calculus (essentially the implicit function theorem) that,
if $\mathbf{x} \in \mathcal{E}_0$ is a regular point for $E$ 
then there exists an open neighborhood $V \subset \mathbb{R}^4$ of $\mathbf{x}$ so that $\mathcal{E}_0 \cap V$ 
is an embedded $3$ dimensional disk
(see for example Corollary $8.9$ Chapter $8$ of \cite{MR2954043},
or Chapter 1, Section 4 of \cite{MR2680546}). 

Since  $E \colon U \subset \mathbb{R}^4 \to \mathbb{R}$ is real valued, 
$\mathbf{x}$ is a regular point if and only if 
\[
\nabla E(\mathbf{x}) \neq 0.
\] 
But $\nabla E(\hat{\mathbf{x}}) \neq 0$ by hypothesis, and by these remarks
we conclude that $\mathcal{E}_0$ is a smooth three dimensional 
manifold near the homoclinic intersection point $\hat{\mathbf{x}}$.  In particular, 
the tangent space of $\mathcal{E}_0$ is 
well defined at $\hat{\mathbf{x}}$, and has
\[
\mbox{dim} \left(T_{\hat{\mathbf{x}}} \mathcal{E}_0 \right) = 3.
\]

Next, recall that by the stable manifold theorem $W^{s,u}(\mathbf{x}_0) \subset \mathbb{R}^4$ are themselves smooth 2-manifolds
(as regular as $f$). Then each has two dimensional tangent
space at $\hat{\mathbf{x}}$.  Moreover,  $W^{s,u}(\mathbf{x}_0)  \subset \mathcal{E}_0$ again by the continuity of 
$E$ at $\mathbf{x}_0$.  
We claim that the tangent spaces of the stable/unstable manifolds at $\hat{\mathbf{x}}$ are linear subspaces of 
the tangent space of $\mathcal{E}_0$ at $\hat{\mathbf{x}}$.  
In other words: $W^{s,u}(\mathbf{x}_0)$ are embedded submanifolds near $\hat{\mathbf{x}}$.

To see this, define the embedded $2$-disks
\[
M_1 := \Gamma^u((-1,1)^2) \subset W^u(\mathbf{x_0}), 
\quad \quad \quad \mbox{and} \quad \quad \quad
M_2 := \Gamma^s((-1,1)^2) \subset W^s(\mathbf{x_0}).
\]
Note that $\hat{\mathbf{x}} \in M_1 \cap M_2$, 
by hypothesis.
We will now argue that the tangent space of $M_1$
at $\hat{\mathbf{x}}$ is contained in the tangent space of 
$\mathcal{E}_0$ at $\hat{\mathbf{x}}$.
The argument for $M_2$ is identical.  

Choose $\eta \in T_{\hat{\mathbf{x}}} M_1$.  Since $(\hat s, \hat t) \in (-1,1)^2$
there exists (by the definition of the tangent space/tangent vectors) an $\epsilon > 0$ and a curve
$\gamma \colon (-\epsilon, \epsilon) \to \mathbb{R}^2$ with 
\[
\gamma(0) = (\hat s, \hat t), \quad \quad \quad \mbox{and} \quad \quad \quad
\gamma(\alpha) \in (-1,1)^2 \quad \quad \mbox{for all } \alpha \in (-\epsilon, \epsilon),
\]
so that the curve $u \colon (-\epsilon, \epsilon) \to \mathbb{R}^4$ defined by 
\[
u(\alpha) := \Gamma^u(\gamma(\alpha)),
\]
has
\[
\frac{d}{d \alpha} u(0) = \eta.
\]

To obtain the desired containment we must show that
$\eta$ is in the tangent space of $\mathcal{E}_0$.  To see this 
consider the function $g \colon (-\epsilon, \epsilon) \to \mathbb{R}$ defined by 
\[
g(\alpha) = E(\Gamma^u(\gamma(\alpha))), 
\] 
and note that $g(\alpha) = C_0$ for all $\alpha \in (-\epsilon, \epsilon)$, simply 
because $\mbox{image}(\Gamma^u) \subset \mathcal{E}_0$ as remarked above.
Then 
\begin{align*}
0 &= \frac{d}{d \alpha} g(0) \\
&= \frac{d}{d \alpha} E(\Gamma^u(\gamma(0)))\\
&= \left< \nabla E(\Gamma^u(\gamma(0))), \frac{d}{d \alpha} \Gamma^u(\gamma(0)) \right> \\
&= \left< \nabla E(u(0)), u'(0) \right> \\
&= \left< \nabla E(\hat{\mathbf{x}}), \eta   \right>,
\end{align*}
so that 
\[
\eta \in \mbox{ker} \nabla E(\hat{\mathbf{x}}).
\]
But the tangent space of the level set of a smooth real valued function 
coincides with the kernel of its gradient, so that 
\[
\eta \in T_{\hat{\mathbf{x}}} \mathcal{E}_0,
\]
as desired.

Next, we establish the transversality.  
Define the vectors $\eta_1, \eta_2, \eta_3 \in \mathbb{R}^4$ by 
\[
\eta_1 = \partial_s \Gamma^u(\hat s, \hat t), 
\quad \quad \quad 
\eta_2 = \partial_t \Gamma^u(\hat s, \hat t),
\quad \quad \quad \mbox{and} \quad \quad \quad
\eta_3 = - \partial_{\sigma} \Gamma^s(\hat \sigma, 0).
\]
Since $\eta_1, \eta_2 \in T_{\hat{\mathbf{x}}} M_1$, and $\eta_3 \in T_{\hat{\mathbf{x}}} M_2$
it follows that  $\eta_1, \eta_2, \eta_3 \in T_{\hat{\mathbf{x}}}\mathcal{E}_0$
by the discussion above.
The hypothesis that $DG(\hat s, \hat t, \hat \sigma)$ is nonsingular gives that its columns span 
$\mathbb{R}^3$. But the columns of $DG(\hat s, \hat t, \hat \sigma)$ match the first three components of 
$\eta_1, \eta_2, \eta_3$, so that these three vectors   
are linearly independent in $T_{\hat{\mathbf{x}}} \mathcal{E}_0$. 
Now observe that any three linearly independent vectors in $T_{\hat{\mathbf{x}}} \mathcal{E}_0$ span, 
as the tangent space of $\mathcal{E}_0$ at $\hat{\mathbf{x}}$ is three dimensional.

Moreover, since $\eta_1, \eta_2$ are linearly independent they form a basis for 
$T_{\hat{\mathbf{x}}} M_1 = T_{\hat{\mathbf{x}}} W^u(\mathbf{x}_0)$. We claim that 
$\eta_2, \eta_3$ form a basis for $T_{\hat{\mathbf{x}}} M_2 = T_{\hat{\mathbf{x}}} W^s(\mathbf{x}_0)$ as well.   
To see this, recall that 
$\hat{\mathbf{x}} =  \Gamma^s(\hat \sigma, 0) =  \Gamma^s(\hat s,  \hat t)$
is the point of homoclinic intersection and 
note that 
\[
\partial_t \Gamma^s(\hat \sigma, 0) = \frac{1}{\tau_s} f(\Gamma^s(\hat \sigma, 0)), 
\]
as for  $t \in (-1,1)$ the curve $\Gamma^s(\hat \sigma, \tau_s t)$ is a solution
of the differential equation.
(Here we use that the charts $\Gamma^{s,u}$ are well aligned with the flow).  
Moreover 
\[
\partial_t \Gamma^u(\hat s, \hat t) = \frac{1}{\tau_u} f(\Gamma^u(\hat s, \hat t)),
\]
for the same reason.  But then 
\begin{align*}
\eta_2 &= \partial_t \Gamma^u(\hat s, \hat t) \\
&=  \frac{1}{\tau_u} f(\Gamma^u(\hat s, \hat t)) \\
&= \frac{1}{\tau_u} f(\hat{\mathbf{x}}) \\
&=  \frac{\tau_s}{\tau_u}    \frac{1}{\tau_s} f(\Gamma^s(\hat \sigma, 0)) \\
&= \frac{\tau_s}{\tau_u} \partial_t \Gamma^s(\hat \sigma, 0) \\
&= \frac{\tau_s}{\tau_u} \eta_1,
\end{align*}
so that
\begin{align*}
\mbox{span} \left(\eta_2,  \eta_3 \right) &=
 \mbox{span} \left(\partial_t \Gamma^u(\hat s, \hat t), -\partial_\sigma \Gamma^s(\hat \sigma, 0)\right) \\
&= \mbox{span} \left(\partial_\tau \Gamma^s(\hat \sigma, 0), -\partial_\sigma \Gamma^s(\hat \sigma, 0)\right) \\
& = T_{\hat{\mathbf{x}}} M_2  \\
&= T_{\hat{\mathbf{x}}}W^s(\mathbf{x}_0),
\end{align*}
as claimed. 
In summary, we have that 
$\eta_1, \eta_2, \eta_3$ simultaneously 
span the tangent spaces of the stable/unstable manifolds at $\hat{\mathbf{x}}$, 
and also span the tangent space of $\mathcal{E}_0$, which is to say that 
\[
T_{\hat{\mathbf{x}}} W^u(\mathbf{x}_0) \oplus T_{\hat{\mathbf{x}}} W^s (\mathbf{x}_0) = T_{\hat{\mathbf{x}}} \mathcal{E}_0.
\]
In other words $W^{s,u}$ intersect transversally at $\hat{\mathbf{x}}$
as desired.
\end{proof}

The remainder of the paper is organized as follows.  
In Section \ref{sec:a-pos} we state and prove the a-posteriori 
theorem which is the main tool for the computer assisted proof
of Theorem \ref{thm:ourThm1}.  We also illustrate the use of the theorem in a 
simple but important example.  Namely, we establish that 
the CRFBP has a non-trivial saddle-focus equilibrium for the mass parameters 
stated in Theorem \ref{thm:ourThm1}.

In Section \ref{sec:parmMethod} we review the parameterization method
for invariant manifolds.  We review its use as a tool for studying 
local stable/unstable manifolds attached to equilibrium solutions of 
differential equations, and also to study locally invariant manifold patches
which result from the flow advection of curves of initial conditions.
In both cases the desired invariant manifold parameterizations are 
recast as solutions of certain nonlinear partial differential equations
subject to some zero or first order constraints.  
By first parameterizing the local stable/unstable manifolds of the equilibrium 
solution and then advecting a mesh of curves parameterizing the boundary of 
the local manifolds, we are able to compute the chart maps discussed in 
Lemmas \ref{lem:aPosEnergyLemma1} and \ref{lem:aPosEnergyLemma2}.

In Section \ref{sec:formalComputations} we develop formal power series methods
for solving the partial differential equations which appear in 
the parameterization method.  First we make a change of variables which transforms
the CRFBP to a polynomial system of seven (rather than four) differential equations.
Working with polynomial vector fields simplifies the development of formal 
series solutions.   It also standardizes the computer assisted error analysis 
of the truncated series, as discussed in the companion paper \cite{thisPaperII}.
We discuss the algorithms for computing the series solutions to any 
desired finite order.  Then we give a brief discussion of the computer assisted proofs
of mathematically rigorous error bounds for the formal series expansions.  
The actual implementation of these computer assisted proofs
is in the companion paper \cite{thisPaperII}. It is reasonable to 
separate the discussion into two separate papers, as validating the truncation error 
estimates for the parameterization method has a distinctly infinite dimensional 
flavor.  

In Section \ref{sec:results} we show how the results of the earlier sections
are combined to complete the proof of Theorem \ref{thm:ourThm1}.
Section \ref{sec:conclusions} summarizes our conclusions and suggests
some directions for future research.  

The paper concludes with several appendices containing some 
more technical details needed in the implementation of the computer 
assisted arguments.
Appendix \ref{sec:norms} discusses explicit choices of
norms for the vector spaces, 
linear mappings (matrices), and bi-linear mappings appearing 
in the computer assisted proofs.  Appendix 
\ref{sec:boundsAppendix} develops some useful bounds
on first and second derivatives.  
Appendix \ref{sec:CRFBP_basics} catalogs a number of
formulas used in computer assisted proofs  for the CRFBP.
Finally, in Appendix \ref{sec:polyFacts} we deal with the fact that 
the formal power series calculations are carried out for a polynomial 
problem related to the CRFBP.  While it is intuitively clear 
how to recover the CRFBP parameterizations from those of the polynomial 
problem, the relationship between the two problems needs to be 
made precise in order that the results of the present work are 
mathematically rigorous.

%
%

\section{A finite dimensional a-posteriori existence theorem}\label{sec:a-pos}
The following is a modification of the classical Newton-Kantorovich 
Theorem,  with some features which facilitate 
computer assisted verification of the hypotheses.
 In particular the theorem incorporates an approximate derivative
and an approximate inverse, which are especially useful when working with 
functions known only up to some approximation error.
This kind of ``Newton-like'' argument appears frequently as a tool for 
computer assisted proof in analysis.
Our approach is in the functional analytic 
tradition of Lanford, Eckmann, Wittwer, and Koch
whose work on 
on renormalization theory and the proof of the Feigenbaum conjectures
\cite{MR648529, MR775044, MR883539, MR727816} was foundational.
For broader surveys of the literature on computer assisted proof
in analysis we refer the interested reader to the surveys   
\cite{MR759197, jpjbReview, MR2652784, Jay_Kons_Review}.

In the statement of the theorem,  $\| \cdot \|$ is any norm on $\mathbb{R}^n$,
$\| \cdot \|_{M}$ is the induced operator (matrix) norm, and $\| \cdot \|_Q$
is the induced norm on bi-linear mappings.
The explicit form of the norms used in numerical applications  
are discussed in detail in Appendix \ref{sec:norms}.
We include the proof of the theorem for the sake of completeness.

\begin{theorem} \label{prop:finiteDimNewton}
{\em
Let $U \subset \mathbb{R}^n$ be an open set, $F \colon U \to \mathbb{R}^n$
be twice continuously differentiable on $U$.  
Suppose that $\bar{x} \in U$, and that  
$A$, $A^\dagger$ are $n \times n$ matrices.   Let $r_* > 0$ be such that 
$\overline{B_{r_*}(\bar x)} \subset U$.   Suppose that $Y_0, Z_0, Z_1, Z_2$
are positive constants with 
\[
\left\| A F(\bar x) \right\| \leq Y_0,
\]
\[
\left\| \mbox{Id} - A A^\dagger \right\|_M \leq Z_0,
\]
\[
\left\| A\left(A^\dagger - DF(\bar x)\right) \right\|_M \leq Z_1, 
\]
and 
\[
\left\| A \right\|_M \sup_{x \in \overline{B_{r_*}(\bar x)}} \|D^2 F(x) \|_Q \leq Z_2.
\]
Define the polynomial 
\[
p(r) := Z_2 r^2 - (1-Z_0 - Z_1) r + Y_0.
\]
If $0 < r \leq r_*$ has 
\[
p(r) < 0,
\]
then there exists a 
unique $\hat x \in B_{r}(\bar x)$ so that 
\[
F(\hat x) = 0.
\]
Moreover $DF(\hat x)$ is invertible and 
\[
\| DF(\hat x)^{-1}\| \leq \frac{\| A \|}{1 - Z_2 r - Z_0- Z_1}.
\]
}
\end{theorem}

\bigskip

\begin{proof}
Assume that there is an $0 < r \leq r_*$ so that $p(r) < 0$.
Then 
\begin{equation}\label{eq:radPolyProof_ineq1}
Z_2 r^2 + (Z_1 + Z_0)r + Y_0 < r,
\end{equation}
from which we obtain
\[
(Z_2 r +Z_1 + Z_0) + \frac{Y_0}{r} < 1.
\]
Since all quantities are strictly positive we have that 
\[
\kappa := Z_2 r +Z_1 + Z_0 < 1.
\]
Furthermore we see that $Z_0 < 1$, again because the quantities above are all strictly positive.
Then
\[
\left\| \mbox{Id} - A A^\dagger \right\|_M \leq Z_0 < 1, 
\]
and it follows by the Neumann theorem that the matrix $A A^\dagger$ is invertible.  
Then $A$ and $A^\dagger$ are each invertible matrices, 
as an invertible matrix cannot factor through a singular matrix. 

Define the \textit{Newton-like} operator $T \colon U \to \mathbb{R}^n$ by 
\[
T(x) = x - AF(x).
\]
Observe that, since $A$ is an invertible matrix, fixed points of 
$T$ are in one to one correspondence with zeros of $F$.
We show that $T$ has a unique fixed point in $B_r(\bar x)$, using 
Banach's fixed point theorem.

First note that 
\[
DT(x) = \mbox{Id} - A DF(x),
\quad \quad \quad \quad x \in U.
\]  
Since $\overline{B_r(\bar x)} \subset U$ the formula holds throughout the closed set $\overline{B_r(\bar x)}$.
We estimate the derivative of $T$ as follows.  
Let $x \in \overline{B_r(\bar x)}$.  Then 
\begin{align*}
\| DT(x) \|_M &= \| \mbox{Id} - A DF(x) \|_M \\
&\leq \|\mbox{Id} - AA^\dagger \|_M + \| A A^\dagger - A DF(\bar x) \|_M
+ \| A DF(\bar x) - A DF(x) \|_M \\
&\leq  \|\mbox{Id} - AA^\dagger \|_M + \left\| A \left(A^\dagger -  DF(\bar x)\right) \right\|_M
+ \| A \|_M \left( \sup_{y \in \overline{B_r(\bar x)}} \|D^2 F(y) \|_Q \right)  \| x - \bar x \| \\
&\leq  Z_0+ Z_1
+ \| A \|_M \left( \sup_{y \in \overline{B_{r_*}(\bar x)}} \|D^2 F(y) \|_Q  \right) \, r \\
&\leq  Z_0+ Z_1
+ Z_2  r  \\
& \leq \kappa.
\end{align*}
Here we use the second derivative bound of Lemma \ref{lem:lipBoundDF}
from Appendix \ref{sec:norms} to pass from line two to three.
The bound is uniform in $x \in  \overline{B_r(\bar x)}$, so that 
\begin{equation}\label{eq:radPolyProof_derBound}
\sup_{x \in  \overline{B_r(\bar x)}} \| DT(x) \| \leq \kappa < 1.
\end{equation} 

To apply the fixed point theorem on the compact metric space $\overline{B_r(\bar x)}$
we first need to show that $T$ maps $ \overline{B_r(\bar x)}$ into itself.
To see this, let $x \in \overline{B_r(\bar x)}$ and observe that  
\begin{align*}
\| T(x) - \bar x \| &\leq  \|T(x) - T(\bar x) \| + \|T(\bar x) - \bar x \| \\
&\leq \sup_{y \in \overline{B_r(\bar x)}} \|DT(y)\|_M \| x - \bar x\| + \| A F(\bar x)\| \\
&\leq (Z_2 r + Z_0 + Z_1) r + Y_0 \\
&< r,
\end{align*}
by the inequality of Equation \eqref{eq:radPolyProof_ineq1},
so that 
\begin{equation} \label{eq:T_containment}
T\left(\overline{B_r(\bar x)}\right) \subset B_r(\bar x).
\end{equation}
Then $T$ maps a complete metric space into itself 
(in fact the mapping is strictly into the interior).

Now for any $x_1, x_2 \in \overline{B_r(\bar x)}$ we see that 
\begin{align*}
\|T(x_1) - T(x_2) \| & \leq \sup_{y \in \overline{B_r(\bar x)}} \|DT(y)\|_M \| x_1 - x_2 \| \\
&\leq \kappa \| x_1 - x_2 \|, \\ 
\end{align*}
with $\kappa < 1$.  Then $T$ is a strict contraction on $\overline{B_r(\bar x)}$, 
hence has a unique fixed point $\hat x \in  \overline{B_r(\bar x)}$ by the 
Banach fixed point theorem.  In fact, Equation \eqref{eq:T_containment}
gives that $\hat x \in B_r(\bar x)$.  Since fixed points of $T$ correspond to 
zeros of $F$, we have that $\hat x$ is the unique zero of $F$ in $B_r(\bar x)$.

Finally, we show that $DF(\hat x)$ is invertible.  
Define the matrix 
\[
B = -ADF(\hat x) + A DF(\bar x) - ADF(\bar x) + A A^\dagger - A A^\dagger + \mbox{Id}.
\]
Then 
\begin{align*}
ADF(\hat x) &= ADF(\hat x) - A DF(\bar x) + ADF(\bar x) - A A^\dagger + A A^\dagger - \mbox{Id} + \mbox{Id} \\
& = \mbox{Id} - B.
\end{align*}
But $ADF(\hat x) = \mbox{Id} - B$ is an invertible matrix by the Neumann theorem, as 
\begin{align*}
\| B \|_M &\leq  \|  ADF(\hat x) - A DF(\bar x)\|_M + \|ADF(\bar x) - A A^\dagger\|_M + \|A A^\dagger - \mbox{Id} \|_M \\
& \leq  \| A(DF(\hat x) - DF(\bar x))\|_M + \|A(DF(\bar x) -  A^\dagger) \|_M + \|A A^\dagger - \mbox{Id} \|_M \\
&\leq  \| A\|_M \sup_{y \in \overline{B_r(\bar x)}} \| DF(y)\|_Q \, \|\hat x - \bar x \| + Z_1 + Z_0 \\
&\leq   \| A\|_M \sup_{y \in \overline{B_{r_*}(\bar x)}} \| DF(y)\|_Q \, r + Z_1 + Z_0 \\
&\leq   Z_2 r + Z_1 + Z_0 \\
&= \kappa \\
&< 1,
\end{align*}
so that $A$ and $DF(\hat x)$ are each invertible singly,
again due to the fact that an invertible matrix cannot factor through a
singular matrix.  

The Neumann theorem provides the bound 
\[
\|(\mbox{Id} - B)^{-1} \|_M \leq \frac{1}{1 - \kappa},
\]
and since $A$ and $DF(\hat x)$ are invertible we have that
\begin{align*}
[A DF(\hat x)]^{-1} &= DF(\hat x)^{-1} A^{-1} \\
&=  (\mbox{Id} - B)^{-1}.
\end{align*}
Multiplying on the right by $A$ we have
\[
DF(\hat x)^{-1} = (\mbox{Id - B})^{-1} A,
\]
so that taking norms gives
\[
\| DF(\hat x)^{-1}\|_M \leq \frac{\|A\|_M}{1 - (Z_2 r + Z_0 + Z_1)},
\]
and the proof is complete.  
\end{proof}

\bigskip

\subsection{Example: computer assisted proof of 
a non-trivial equilibrium solution.} \label{sec:NKex1}
 The following proposition
provides an elementary application of Theorem \ref{prop:finiteDimNewton}. 
We give the details as a kind of review/tutorial on 
validated numerics for a-posteriori analysis of finite dimensional systems of nonlinear 
equations, and also because the equilibrium point is the jumping off point for all the  
subsequent analysis in the remainder of the paper.  The example 
is simple enough that it could be worked by hand.  Nevertheless, 
careful consideration of this problem highlights the issues which come
up later when we check the hypotheses of the Devaney theorem.

\begin{prop}[Existence of a non-trivial equilibrium in the CRFBP] \label{prop:capEquil} {\em
Consider the circular restricted four body problem 
whose vector field $f$ is defined in Equation \eqref{eq:SCRFBP}.  Choose 
mass values $m_1 = 0.5$, $m_2 = 0.3$, and $m_3 = 0.2$, and let
\[
\bar{\mathbf{x}} = 
\left(
\begin{array}{c}
0.927099246135636 \\
0 \\
0.217703423699760 \\
0
\end{array}
\right)
\]
The vector field $f$ has an isolated equilibrium solution $\mathbf{x}_0 \in \mathbb{R}^4$
with 
\[
\left\|\mathbf{x}_0 - \bar{\mathbf{x}} \right\| \leq 3 \times 10^{-15}.
\]
Here the norm is the max-norm defined in Appendix  \ref{sec:norms}.}
\end{prop}

\begin{proof}
Throughout the proof we use the IntLab package running under MatLab 
to obtain interval enclosures of numerical quantities.  
Taking $m_1 = 1/2$, $m_2 = 3/10$ and $m_3 = 1/5$ we use the formulae in Appendix
\ref{sec:CRFBP_basics} -- evaluated with interval arithmetic -- to compute explicit enclosures
of the locations of the primary bodies.
We find that 
\[
x_1 \in [  -0.43588989435407,  -0.43588989435406],
\quad 
y_1 = 0,
\]
\[
x_2 \in [   0.48177304112817,   0.48177304112819],
\quad 
y_2 \in [  -0.39735970711952,  -0.39735970711951],
\]
and
\[
x_3 \in [   0.36706517419289,   0.36706517419290],
\quad 
y_3 \in [   0.59603956067926,   0.59603956067928].
\]
These values allow us to define the functions $r_j$, $j = 1,2,3$
appearing in the definition of $f$ and its derivatives, and 
determine the domain $U \subset \mathbb{R}^4$ on which the 
problem is posed.  Note that $m_2, m_3$ cannot be represented 
exactly as floating point numbers base $2$, so that even these 
masses have to be treated as intervals in all computations.  

An equilibrium of the restricted four body problem 
must have that $\dot x = \dot y = 0$.  Re-examining the vector field $f$ in light of 
this observation reveals that the remaining variables $(x,y)$
are in equilibrium if and only if
\begin{align*}
\Omega_x(x,y) & = 0 \\
\Omega_y(x,y) & = 0.
\end{align*}
Define the mapping $g \colon \mathbb{R}^2 \to \mathbb{R}^2$
by 
\[
g(x,y) := \left(
\begin{array}{c}
\Omega_x(x,y) \\
\Omega_y(x,y) 
\end{array}
\right)
\]
\[
= 
\left(
\begin{array}{c}
   x - \frac{m_1(x - x_1)}{r_1(x,y)^3}
- \frac{m_2(x - x_2)}{r_2(x,y)^3} - \frac{m_3(x - x_3)}{r_3(x,y)^3} \\
y - \frac{m_1(y - y_1)}{r_1(x,y)^3}
- \frac{m_2(y - y_2)}{r_2(x,y)^3} - \frac{m_3(y - y_3)}{r_3(x,y)^3}
\end{array}
\right).
\]
Then $(x, 0 , y, 0)$ is an equilibrium for the CRFBP if and only if $(x,y)$ is a zero 
of for $g$.

Observe that 
\[
Dg(x,y) = 
\left(
\begin{array}{cc}
   \frac{\partial}{\partial x} \Omega_x(x,y) &  \frac{\partial}{\partial y} \Omega_x(x,y)      \\
 \frac{\partial}{\partial x} \Omega_y(x,y)  & \frac{\partial}{\partial y} \Omega_y(x,y)
\end{array}
\right) = 
\left(
\begin{array}{cc}
  g_{11} &g_{12}  \\
 g_{21} &g_{22}
\end{array}
\right),
\]
where the quantities $g_{ij}$, $1 \leq i,j \leq 2$ are as defined in Appendix \ref{sec:CRFBP_basics}.
Using the explicit formulas for $g$ and $Dg$, we run a numerical 
Newton method starting from the initial guess
$(0.92, 0.21)$ and obtain the approximate zero  
\[
\bar x = 0.927099246135636, \quad \quad 
\bar y = 0.217703423699760.
\]
Our goal is to prove that there is true solution nearby via Theorem \ref{prop:finiteDimNewton}.

In preparation for application of Theorem we define $2 \times 2$ matrices  
$A^\dagger$ and $A$.  For the former we choose the $2 \times 2$ matrix of floating point numbers
which results from evaluating the formula for $Dg(\bar x, \bar y)$
without using interval arithmetic.  The result is
\[
A^\dagger := 
\left(
\begin{array}{cc}
2.074531863336581 &  0.163688766491296 \\
   0.163688766491296  & 1.448616847931906
\end{array}
\right).
\]  
The matrix $A$ needs to be a $2 \times 2$ approximate inverse of $A^\dagger$.  
For this we compute a floating point approximation of the 
inverse of $A^\dagger$ using MatLab 
and obtain
\[
A := \left(
\begin{array}{cc}
  0.486372903543232  & -0.054958480394206 \\
  -0.054958480394206 &  0.696523782188810
\end{array}
\right).
\]
The formula for the matrix norm is given explicitly in Appendix  \ref{sec:norms},
and we evaluate this formula to obtain the interval enclosure
\[
\| A \|_M \leq 0.75148226258302.
\]
Now we need to define the positive constants $Y_0$, $Z_0$, $Z_1$, and $Z_2$ hypothesized in 
Theorem \ref{prop:finiteDimNewton}.

First we compute an interval enclosure of the defect 
\[
A g(\bar x, \bar y) \in 
 10^{-14}
 \left(
 \begin{array}{c}
\, [  -0.24035361087522,   0.18345710317770] \\
\, [  -0.25639867955817,   0.29923189175017]
\end{array}
\right),
\]
and using the explicit formula for the norm given in Equation \eqref{eq:normDef} of the Appendix
we have that 
\[
\| A g(\bar x, \bar y) \| \in  [0,   0.29923189175017 \times 10^{-14}].
\]
Then we take 
\[
Y_0 := 0.29923189175017 \times 10^{-14}.
\]

Next we check, again using interval arithmetic and the definitions of $A^\dagger$ and $A$ above, that 
\[
\mbox{Id} - A A^\dagger \in
\]
\[
10^{-15} \left(
\begin{array}{cc}
\, [  -0.22204460492504,   0.11102230246252] & [   0.00000000000000,   0.02775557561563] \\
\, [  -0.01387778780782,   0.01387778780782] & [  -0.22204460492504,   0.11102230246252] \\
\end{array}
\right).
\]
Using the formula for the induced matrix norm derived in Appendix  \ref{sec:norms} we now have that 
\[
\| \mbox{Id} - A A^\dagger  \|_M \in 
 [   0,   0.24980018054067 \times 10^{-15}]. 
\]
So, we define 
\[
Z_0 :=  0.24980018054067 \times 10^{-15}.
\]

For the $Z_1$ bound we first use that $Dg(\bar x, \bar y)$ is given explicitly 
by the formula above.  Once again using interval arithmetic we obtain
the enclosure 
\[
A(A^\dagger - Dg(\bar x, \bar y)) \in  
\]
\[
10^{-13} \left(
\begin{array}{cc}
\, [  -0.05697595725541,   0.05451094681948] & [  -0.08685491176201,   0.08979888842367] \\
\, [  -0.12111986933565,   0.12523041771790] & [  -0.08678718041998,   0.08338891260758] 
\end{array}
\right).
\]
The explicit formula for the matrix norm given in Appendix  \ref{sec:norms} 
allows us to compute the interval enclosure
\[
\| A(A^\dagger - Dg(\bar x, \bar y))\|_M \in [   0.00000000000000,   0.20861933032547 \times 10^{-13}], 
\]
so we define 
\[
Z_1 :=  0.20861933032547 \times 10^{-13}.
\]

Finally, we bound the second derivative of $g$ near the approximate
solution $(\bar x, \bar y)$ on a neighborhood of size 
$r_* = 10^{-6}$ (this choice is somewhat arbitrary).
Exploiting both the formulas for the second derivatives of $F$ given in 
Section \ref{sec:CRFBP_basics}, and the explicit formula for the second derivative bounds 
from Lemma \ref{lem:lift_of_a_solution}, we have that 
\begin{align*}
\sup_{\mathbf{x} \in \overline{B_{r_*}(\bar x, \bar y)}} \| D^2 F(\mathbf{x}) \|_Q \in 
[   131.144085137264,   140.058509630884]
\end{align*}
so that 
\[
\|A\|_M  \sup_{\mathbf{x} \in \overline{B_{r_*}(\bar x, \bar y)}} \| D^2 F(\mathbf{x}) \|_Q  \in 
[ 70.992409102564,   105.251485711422].
\]
Then we take
\[
Z_2 = 105.251485711422.
\]

Now define the polynomial 
\[
p(r) = Z_2 r^2 - (1 - Z_1 - Z_0) r + Y_0,
\]
with $Y_0, Z_0, Z_1$, and $Z_2$ as given above.  
Evaluating the quadratic formula using interval arithmetic we find 
that if 
\[
r \in \left[2.992618149394393 \times 10^{-15}, 0.01 \right],
\]
then 
\[
p(r) \leq 0.
\]
But the result is only meaningful for $r \leq r_*$. 
So for example, for every 
\[
r \in [3 \times 10^{-15}, 10^{-6}],
\]
we have that $r \leq r_*$ and 
\[
p(r) < 0.
\]
We conclude that there exists a true solution pair $(\hat x, \hat y)$ having that 
\[
\| (\bar x - \hat x, \bar y - \hat y) \| < 3 \times 10^{-15},
\]
and the proof is complete.
\end{proof}

\begin{prop}[Saddle-focus stability] \label{prop:saddleFocus}{\em
Consider the CRFBP with masses $m_1 = 1/2$, $m_2 = 3/10$ and $m_3 = 1/5$.
Let $\mathbf{x}_0 = (\hat x, 0, \hat y, 0) \in \mathbb{R}^4$ 
be the equilibrium solution whose existence is 
established in Proposition \ref{prop:capEquil}.  The matrix 
$Df(\mathbf{x}_0)$ has four non-zero eigenvalues 
of the form 
\[
\lambda = \pm \alpha \pm i \beta,
\]
where
\[
\alpha \in [   0.86237485318926,   0.86237485318937] ,
\]
and
\[
\beta \in [   0.99101767480653,   0.99101767480664] .
\]
In particular $\mathbf{x}_0$ is a saddle-focus.}
\end{prop}
\begin{proof}
To find the eigenvalues we use the explicit formula given by Equation \eqref{eq:CRFBP_eigs}
in the Appendix.  The formula is evaluated using interval arithmetic and the interval enclosures 
of $\hat x, \hat y$ from Proposition \ref{prop:capEquil}.
\end{proof}

Throughout the remainder of the paper we adopt the convention that 
\[
\lambda_{1,2} = -\alpha \pm i \beta,
\]
are the stable and 
\[
\lambda_{3,4} = \alpha \pm i \beta,
\]
the unstable eigenvalues.  

\begin{remark}[Eigenvalues/eigenvectors] 
We compute eigenvectors for $Df(\hat x, 0, \hat y, 0)$ using the formulas 
established in Lemma \ref{lem:eigVectLemma}. 
The reader should be aware that in other problems, where explicit formulas for the 
eigenvalues/eigenvectors are not available, one can use existing 
validated numerical methods to compute the eigendata.  Such algorithms 
are standard for example in the IntLab package.    
See \cite{MR2652784}
for more complete discussion.
\end{remark}

\section{The parameterization method} \label{sec:parmMethod}
The parameterization method is a functional analytic framework 
for studying invariant manifolds of discrete and continuous time dynamical systems.
The method applies to both finite and infinite dimensional dynamical systems.  The  
works of \cite{MR1976079, MR1976080, MR2177465, MR2240743,MR2289544, MR2299977}
originally developed the method for stable/unstable manifolds in Banach 
spaces and for whiskered tori and their attached invariant manifolds.  
By now there is a substantial literature surrounding
the parameterization method, and a detailed survey is beyond the scope of 
the present work.  The interested reader can consult 
the book of \cite{mamotreto} on the topic, and there find many 
examples, applications, and a thorough overview of the literature.
In this section we give only a brief overview.

\begin{figure}[t!]
\center{
  \includegraphics[width=0.75\linewidth]{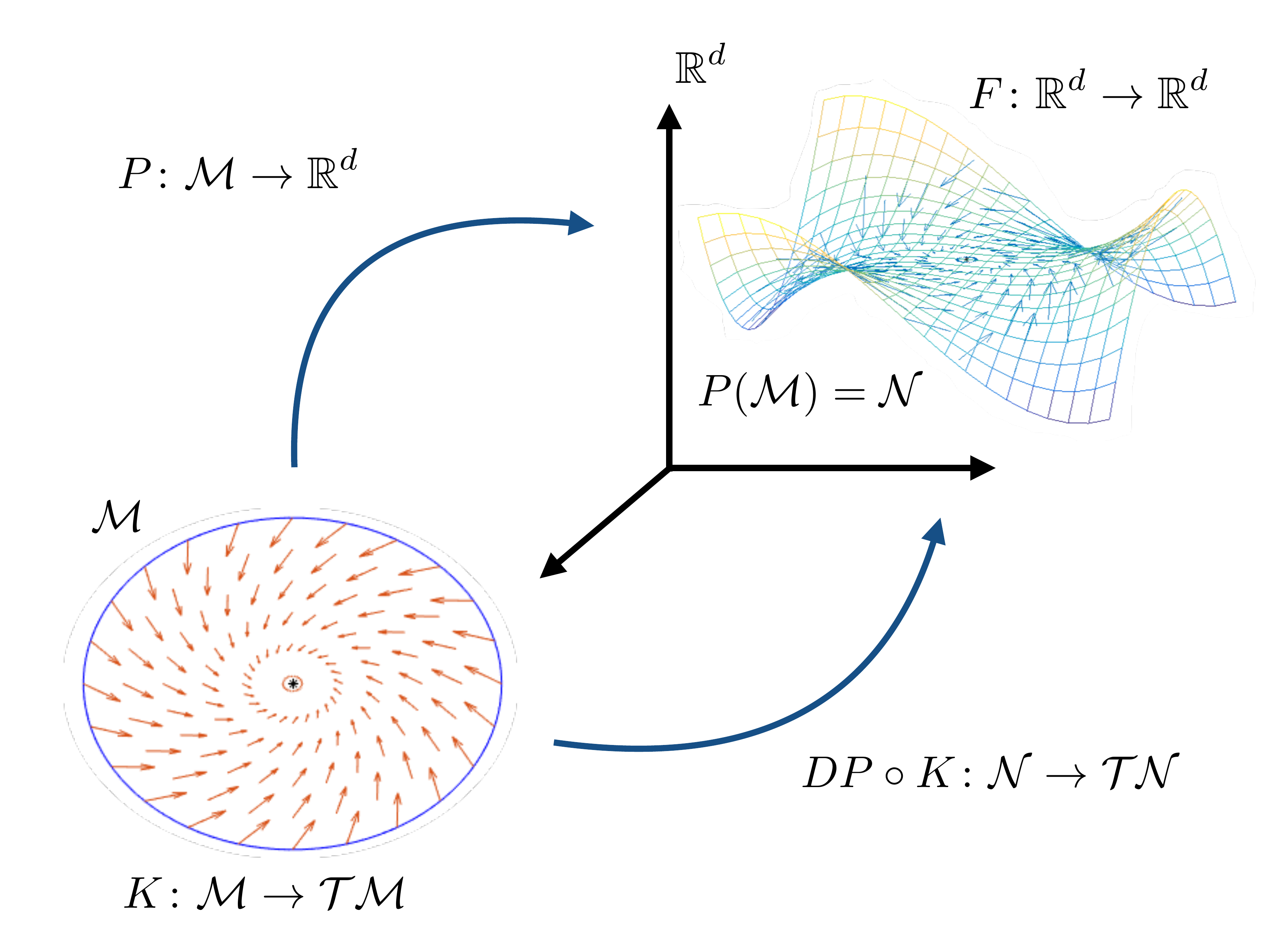}
}
\caption{
\textbf{Schematic illustration of the parameterization method}: here 
$\mathcal{M}$ is a smooth manifold and $P$ is an embedding,
so that $\mbox{image}(P) = \mathcal{N}$ is an embedded manifold in $\mathbb{R}^d$.
Any vector field $K$ on $\mathcal{M}$ is pushed forward by $DP$, so that 
$DP \circ K$ is a vector field on $\mathcal{N}$.  If the push forward vector field 
is equal to the vector field $F$ on $\mathcal{N}$, then 
$F$ is tangent to $\mathcal{N}$ and the two vector fields
generate the same dynamics.  In particular, $\mathcal{N}$ is an invariant manifold
whose dynamics are conjugate to those of $K$. 
} \label{fig:infConj}
\end{figure}

So, consider a smooth vector field
$F \colon \mathbb{R}^d \to \mathbb{R}^d$.  We are interested in 
the dynamics generated by $F$, and in particular we want to study smooth invariant manifolds
for the differential equation $x' = F(x)$ (or even manifolds with boundary 
which are only forward/backward invariant).
Let $\mathcal{M}$ be a smooth manifold and $P \colon \mathcal{M} \to \mathbb{R}^d$  
an embedding, so that $\mathcal{N} := P(\mathcal{M})$ is a smooth embedded 
manifold in $\mathbb{R}^d$.  
If $\mathcal{N}$ is a manifold with boundary we ask that $F$ is inflowing/outflowing
on the boundary.  Now, we ask the question: when is $\mathcal{N}$ an 
invariant manifold for $F$?

Let $T\mathcal{M}$, $T\mathcal{N}$ denote the tangent bundles. Suppose that $\sigma \in \mathcal{M}$ and $x \in \mathcal{N}$.  
We write $T_\sigma \mathcal{M}$ and $T_x \mathcal{N}$ to denote the  tangent spaces based at $\sigma$ and $x$. For any $\sigma \in \mathcal{M}$ the differential $DP(\sigma)$ is a linear map from $T_\sigma \mathcal{M}$ to $T_{P(\sigma)} \mathcal{N}$.
This fact lets us push forward vector fields from $\mathcal{M}$ to $\mathcal{N}$.

Then let $K \colon \mathcal{M} \to T\mathcal{M}$ be a vector field on $\mathcal{M}$.   
Define the vector field $DP \circ K \colon \mathcal{N} \to T\mathcal{N}$ on the image of $P$
as follows.  For each $\sigma \in \mathcal{M}$, there is an attached 
vector $K(\theta) \in T_\sigma \mathcal{M}$.  Applying the 
differential gives $DP(\sigma) K(\sigma) \in T_{P(\sigma)} \mathcal{N}$.
Letting $\sigma$ vary, we have a vector field defined over $\mathcal{N}$.
Another vector field is defined on $\mathcal{N}$
by simply restricting $F$, and we are interested in reconciling the 
difference between these two fields.

Indeed, suppose that $P \colon \mathcal{M} \to \mathbb{R}^d$ satisfies the 
\textit{invariance equation}
\begin{equation} \label{eq:invEqAbstract}
DP \circ K = F \circ P.
\end{equation}
Then the push forward of $K$ is actually equal to the restriction of $F$ on $P$.
It follows from the fact that the push forward of $K$
is tangent to $P$, that the restriction of $F$ to $P$ is everywhere tangent to
the image of $P$.  Hence, the image of $P$ is locally flow invariant.  
The situation is illustrated in Figure \ref{fig:infConj}.

To make this discussion more quantitative we introduce 
coordinates $\sigma = (\sigma_1, \ldots, \sigma_m)$
on $\mathcal{M}$, so that $K_1(\sigma)$, $\ldots$, $K_m(\sigma)$ 
are the components of the vector field $K$.
The coordinates for the push forward are
\[
(D P \circ K)(\sigma) = K_1(\sigma)\frac{\partial}{\partial \sigma_1} P(\sigma) + \ldots
+ K_m(\sigma) \frac{\partial}{\partial \sigma_m} P(\sigma),
\]   
while the restriction in coordinates is
\[
\left(F \left|_{\mathcal{N}} \right. \right)(\sigma)  = F(P(\sigma)).
\]
In coordinates this invariance equation \eqref{eq:invEqAbstract} becomes 
\begin{equation} \label{eq:invEqGeneral}
K_1(\sigma)\frac{\partial}{\partial \sigma_1} P(\sigma) + \ldots
+ K_m(\sigma) \frac{\partial}{\partial \sigma_m} P(\sigma) = F(P(\sigma)),
\end{equation}
for $\sigma \in \mathcal{M}$.  
Solutions of the invariance equation have several desirable properties.   
\begin{itemize}
\item $P$ maps orbits 
of $K$ on $\mathcal{M}$ to orbits of $F$ on $\mathcal{N}$.  Or to put it another way,
the infinitesimal conjugacy of Equation \eqref{eq:invEqGeneral} generates a 
flow conjugacy on the manifolds. This is made more precise in the examples 
in the next sections.  Because of this fact, the manifold $\mathcal{M}$ is referred to 
as the model space and the vector field $K$ as the model dynamics.  
\item There is no requirement that $P$ be the graph of a function, 
hence the parameterization can follow folds in the embedding.
\item  Equation \eqref{eq:invEqGeneral} is a first order nonlinear system of PDEs,
and provides a quantitative approach to the study of invariant manifolds.
Note for example that in contrast to more standard approaches to invariant 
manifolds -- like the graph transform method,
the sequence space approach of Irwin, and Lyapunov-Perron method -- 
Equation \eqref{eq:invEqGeneral} does not involve the flow generated by 
$F$.  Only the vector fields appear in the equation and $P$ is the only quantity not 
explicitly known.    
\item The fact that $P$ solves a PDE is useful for both implementing numerical methods
and for developing a-posteriori analysis.  
\end{itemize}

\subsection{Stable/Unstable Manifolds Attached to Equilibria} \label{sec:parameterization}
We now specialize the general parameterization method discussed 
above to the case of a two dimensional 
stable manifold attached to an equilibrium solution. 
Vector space norms used in the present 
work are discussed in detail in Appendix \ref{sec:norms}.
In this section  $z = x + i y$ denotes a complex number, and $|z| = \sqrt{x^2 + y^2}$
denotes the usual complex absolute value.  Let 
\[
D := \left\{ \sigma = (z_1, z_2) \in \mathbb{C}^2 \, \colon |z_{1}|, |z_2| < 1 \right\},
\]
denote the open unit polydisc centered at the origin in $\mathbb{C}^2$.

Given an open set $U \subset \mathbb{C}^d$ and an analytic vector field 
$F \colon U \to \mathbb{C}^d$, suppose that $\mathbf{x}_0 \in U$
is an equilibrium solution for $F$.   We are interested in saddle-focus equilibria, 
so (focusing for a moment on the stable case) let us assume that $DF(\mathbf{x}_0)$ has 
a complex pair of eigenvalues $\lambda_1, \lambda_2 \in \mathbb{C}$ with 
$\lambda_2 = \overline{\lambda_1}$ and $\mbox{real}(\lambda_{1,2}) < 0$.
Assume in addition that none of the other eigenvalues for $DF(\mathbf{x}_0)$ are stable,
so that $W^s(\mathbf{x}_0)$ is two dimensional.  Let $\mathbf{\xi}_1, \mathbf{\xi}_2 \in \mathbb{C}^d$
be eigenvectors associated with $\lambda_1, \lambda_2$.

Following the discussion of the previous section, we choose $\mathcal{M} = D$ as
the model space for the stable manifold 
and define the model vector field $K \colon D \to \mathbb{C}^2$ by  
\[
K(\mathbf{z}) := \Lambda \mathbf{z}, \quad \quad \quad \mathbf{z} \in D,
\]
where $\mathbf{z} = (z_1, z_2)$, and 
\[
\Lambda = \left(
\begin{array}{cc}
\lambda_1 & 0 \\
0 & \lambda_2 
\end{array}
\right).
\]
Plugging these choices into Equation \eqref{eq:invEqGeneral}
we see that a function $P \colon D \to \mathbb{C}^d$ parameterizing the 
stable manifold should satisfy 
\begin{equation}\label{eq:invEqnFlows1}
 \lambda_1 z_1 \frac{\partial}{\partial z_1} P(z_1, z_2) + 
 \lambda_2 z_2 \frac{\partial}{\partial z_2} P(z_1, z_2) = F(P(z_1, z_2)),
\end{equation}
for all $(z_1, z_2) \in D$.
We impose the first order constraints
\begin{equation} \label{eq:firstOrder1}
P(0, 0) = \mathbf{x}_0, 
\quad \quad \quad \quad
\frac{\partial}{\partial z_1} P(0, 0) = s \mathbf{\xi}_1, \quad \quad \quad \quad 
 \mbox{and} \quad \quad \quad \quad 
\frac{\partial}{\partial z_2} P(0, 0) = s \mathbf{\xi}_2,
\end{equation}
with $|s| > 0$ an arbitrary nonzero scaling.  
The following Lemma, 
whose proof is included for the sake of completeness,
establishes the conjugacy alluded to in the previous section. 
The meaning of the lemma is illustrated in Figure \ref{fig:conjugacyFP}.

\begin{figure}[t!]
\center{
  \includegraphics[width=0.6\linewidth]{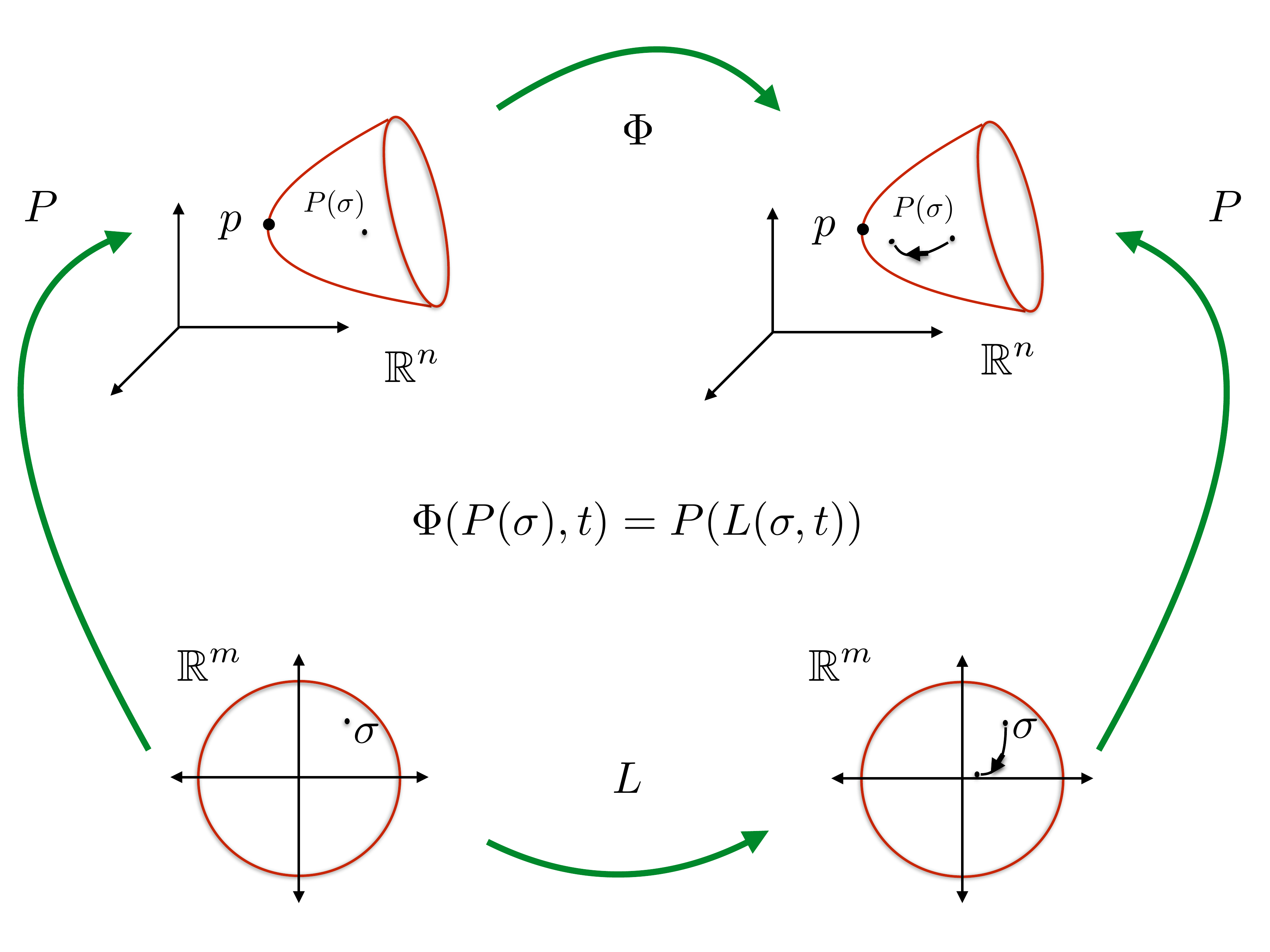}
}
\caption{Flow conjugacy: if $P$ solves Equation \eqref{eq:invEqnFlows1} then 
orbits of the linearized system are mapped to orbits of $f$, and the diagram 
above commutes.} \label{fig:conjugacyFP}
\end{figure}

\begin{lemma}[Flow conjugacy for stable manifolds] \label{lem:flowConj} {\em
Suppose that $P \colon D \to \mathbb{C}^d$ is a smooth solution of 
the invariance equation  \eqref{eq:invEqnFlows1}.  Then 
\begin{equation} \label{eq:flowInvSM}
\Phi(P(z_1, z_2), t) = P(e^{\lambda_1 t} z_1, e^{\lambda_2 t} z_2), 
\end{equation}
for all $(z_1, z_2) \in D$ and all $t \geq 0$.  }
\end{lemma}
\begin{proof}
Choose $(z_1, z_2) \in D$, and define the curve $\gamma \colon [0, \infty) \to \mathbb{C}^d$
by 
\[
\gamma(t) = P(e^{\lambda_1 t} z_1, e^{\lambda_2 t} z_2).
\]
Note that $\gamma$ is well defined, as  
\begin{equation} \label{eq:goodDom}
(e^{\lambda_1 t} z_1, e^{\lambda_2 t} z_2) \in D,
\quad \quad \mbox{ for all } t \in [0, \infty).
\end{equation}
This last fact exploits the hypothesis that $\mbox{real}(\lambda_{1,2}) < 0$, 
that $\lambda_1, \lambda_2$ are complex conjugates, and that $D$ is a polydisc 
It is also clear that $\gamma(0) = P(z_1, z_2)$.
 
The inclusion of Equation \eqref{eq:goodDom}, 
combined with the hypothesis that $P$ satisfies Equation \eqref{eq:invEqnFlows1} 
on $D$ for all $t > 0$, give that 
\[
F\left(P(e^{\lambda_1 t} z_1, e^{\lambda_2 t} z_2)\right)
= 
\]
\[
\lambda_1 \left(e^{\lambda_1 t} z_1\right) \frac{\partial}{\partial z_1} P(e^{\lambda_1 t} z_1, e^{\lambda_2 t} z_2)
+ \lambda_2 \left( e^{\lambda_2 t} z_2 \right)  \frac{\partial}{\partial z_2} P(e^{\lambda_1 t} z_1, e^{\lambda_2 t} z_2). \label{eq:parmProof1}
\]
But then 
\begin{align*}
\frac{d}{dt} \gamma(t)  &= \frac{d}{dt} P(e^{\lambda_1 t} z_1, e^{\lambda_2 t} z_2) \\
&= D P(e^{\lambda_1 t} z_1, e^{\lambda_2 t} z_2) \left(
\begin{array}{c}
\lambda_1 e^{\lambda_1 t} z_1 \\
\lambda_2 e^{\lambda_2 t} z_2
\end{array} 
\right) \\
&= \lambda_1 \left(e^{\lambda_1 t} z_1\right) \frac{\partial}{\partial z_1} P(e^{\lambda_1 t} z_1, e^{\lambda_2 t} z_2)
+ \lambda_2 \left( e^{\lambda_2 t} z_2 \right)  \frac{\partial}{\partial z_2} P(e^{\lambda_1 t} z_1, e^{\lambda_2 t} z_2) \\
&=  F(P(e^{\lambda_1 t} z_1, e^{\lambda_2 t} z_2)) \\
&= F(\gamma(t)), 
\end{align*}
where we use Equation \eqref{eq:parmProof1} to pass from the third to the fourth line.  
So $\gamma$ solves the initial value problem $\gamma' = F(\gamma)$ with $\gamma(0) = P(z_1, z_2)$, 
which is to say that 
\[
\Phi(P(z_1, z_2), t) = \gamma(t) = P(e^{\lambda_1 t} z_1, e^{\lambda_2 t} z_2),
\]
for all $t \geq 0$.  Since $(z_1, z_2) \in D$ was arbitrary we have the result.  
\end{proof}

Examining the conjugacy of Lemma \ref{lem:flowConj} 
explains why the image of $P$ is a local stable manifold, as
for any $(z_1, z_2) \in D$ we have that 
\begin{align}
\lim_{t \to \infty} \Phi(P(z_1, z_2), t) &= \lim_{t \to \infty} P(e^{\lambda_1 t} z_1, e^{\lambda_2 t} z_2) \nonumber \\
& = P(0,0)  \nonumber \\
& = \mathbf{x}_0,  \label{eq:parmLimit}
\end{align}
by the continuity of $P$,
where the passage to the last line again evokes the first order constraints of Equation \eqref{eq:firstOrder1}.
In other words, for all $(z_1,z_2) \in D$, we have
\[
P(z_1, z_2) \in W^s(\mathbf{x}_0).
\]
All the comments in the this section apply to unstable manifolds by reversing time.

\begin{remark}[Flow conjugacy] 
Define the liner flow $L \colon \mathbb{R} \times D \to D$ by  
\begin{equation} \label{eq:defLinFlow_L}
\mathcal{L}(t, z_1, z_2) = 
\left(
\begin{array}{c}
e^{\lambda_1 t} z_1 \\
e^{\lambda_2 t} z_2
\end{array}
\right).
\end{equation}
Then Equation \eqref{eq:flowInvSM} is viewed as a conjugacy between 
flows.  That is, the equation says that 
\[
\left(\Phi \circ P \right) (t, z_1, z_2)= \left(P \circ \mathcal{L} \right) (t, z_1, z_2),
\]
for $(z_1, z_2) \in D$ and all $t \geq 0$.
\end{remark}

\begin{remark}[Real analytic vector fields and the real image of the parameterization] \label{rem:realImage}
We look for solutions $P$ of Equation \eqref{eq:invEqnFlows1} with the property that 
\[
P(z, \bar z) \in \mathbb{R}^d.
\]
That is we seek parameterizations which, when restricted to complex conjugate
variables, are real valued.  Note that  an analytic function $P$ has that
$P(z, \bar z) \in \mathbb{R}^d$ for all $z \in \mathbb{C}$ 
if and only if the Taylor coefficients $a_{mn} \in \mathbb{C}^d$  
satisfy the symmetry condition
\[
a_{nm} = \overline{a}_{mn},
\] 
for all $(m,n) \in \mathbb{N}^2$.
If $\lambda_1, \lambda_2$ are complex conjugates
and if $F$ is real analytic, then we can arrange that the power series coefficient of $P$ 
have this property. 
Since the solution $P$ is unique up to the choice of the eigenvectors, 
only the eigenvectors can effect this outcome.  It turns out that choosing 
$s \mathbf{\xi}_2 = s \overline{\mathbf{\xi}_1}$ gives the desired result.
See Remark \ref{rem:realConj} below for more detailed discussion
\end{remark}

\begin{remark}[Real parameterization and real flow conjugacy] \label{rem:realConj}
Suppose, as discussed in Remark \ref{rem:realImage}, that $P \colon D \to \mathbb{C}^d$ 
is a solution of Equation \eqref{eq:invEqnFlows1} with the 
property that $P(z, \bar z) \in \mathbb{R}^d$ for all $z \in D$.
Let 
\[
B = \left\{(\sigma_1, \sigma_2) \in \mathbb{R}^2 : \sqrt{\sigma_1^2 + \sigma_2^2} < 1 \right\}.
\]
For $\sigma = (\sigma_1, \sigma_2) \in B$ define the complex conjugate variables
\[
z = \sigma_1 + i \sigma_2, 
\quad \quad \quad \mbox{and} \quad \quad \quad 
\bar z = \sigma_1 - i \sigma_2,
\]
and the mapping $Q \colon B \to \mathbb{R}^d$ by 
\[
Q(\sigma_1, \sigma_2) := P(\sigma_1 + i \sigma_2, \sigma_1 - i \sigma_2).
\]
Note that $Q$ is well defined and real valued, 
as $(\sigma_1 + i \sigma_2, \sigma_1 - i \sigma_2) \in D$ for all 
$(\sigma_1, \sigma_2) \in B$.

Now define the linear flow $L \colon \mathbb{R}^2 \times \mathbb{R} \to \mathbb{R}^2$ by
\begin{equation} \label{eq:def_linearFlow}
L(\sigma_1, \sigma_2, t) := e^{\alpha t} \left(
\begin{array}{cc}
\cos(\beta t) &  - \sin(\beta t) \\
\sin(\beta t) &  \cos(\beta t) 
\end{array}
\right)
\left[
\begin{array}{c}
\sigma_1 \\
\sigma_2
\end{array}
\right],
\end{equation}
and note that $L$
is the real flow which results from restricting 
Equation \eqref{eq:defLinFlow_L} to complex conjugate variables.

Then, for any $(\sigma_1, \sigma_2) \in B$ and any $t \geq 0$ we have 
\begin{align*}
\Phi\left(Q(\sigma_1, \sigma_2), t \right) &= \Phi \left( P(\sigma_1 + i \sigma_2, \sigma_1 - i \sigma_2), t  \right) \\
&= P(e^{\lambda_1 t} (\sigma_1 + i \sigma_2), e^{\lambda_2 t} (\sigma_1 - i \sigma_2)) \\
& = Q \left (\mbox{real}\left(e^{\lambda_1 t} (\sigma_1 + i \sigma_2) \right), 
                  \mbox{imag}\left( e^{\lambda_1 t} (\sigma_1 + i \sigma_2) \right) \right),
\end{align*}
and
\begin{equation} \label{eq:realConj}
\Phi\left(Q(\sigma_1, \sigma_2), t \right) = Q(L(\sigma_1, \sigma_2, t)), 
\end{equation}
for all  $(\sigma_1, \sigma_2) \in B$ is the resulting real conjugacy. 
\end{remark}

\begin{remark}[Outflowing/inflowing invariant manifolds with boundary] \label{rem:inflowOutflow}
Consider the curve 
\[
\mathbf{c}(s) = \left(
\begin{array}{c}
c_1(s) \\
c_2(s)
\end{array}
\right)
= 
R \left(
\begin{array}{c}
\cos(s) \\
\sin(s)
\end{array}
\right),
\]
with $0 < R < 1$, so that $\mbox{image}(\mathbf{c}) \subset B$.
We are interested in the restriction of the real linear vector field
to the curve $\mathbf{c}(s)$.
The matrix for the linear vector field is
\begin{equation} \label{eq:def_M}
M = \left(
\begin{array}{cc}
\alpha & - \beta \\
\beta &  \alpha 
\end{array}
\right),
\end{equation}
so that 
\begin{align*}
\left< \mathbf{c}(s), M \mathbf{c}(s) \right> 
& = \alpha R^2.
\end{align*} 
where $\left<\cdot,\cdot\right>$ is the Euclidean inner product. It follows that 
the angle $\theta$ between the vectors $\mathbf{c}(s)$ and $M \mathbf{c}(s)$ is given by 
\begin{align*}
\cos(\theta) 
& = \frac{\alpha}{ \sqrt{\alpha^2 + \beta^2}},
\end{align*}
and since $\mathbf{c}(s)$ is an outward normal 
the linear vector field is inflowing with respect to $\mathbf{c}(s)$ 
if $\alpha < 0$ and outflowing of $\alpha > 0$.  
We also note that if we replace the circular curve $\mathbf{c}(s)$ with a fine enough piecewise 
linear mesh of secant lines, then the same result holds.   
\end{remark}

\begin{remark}[Existence and uniqueness of solutions of Equation \eqref{eq:invEqnFlows1}]
Existence, uniqueness, and regularity questions concerning the parameterization method
for stable/unstable manifolds of equilibrium/fixed points are considered at length 
in \cite{MR1976079, MR1976080, MR2177465, MR2289544, MR2122688, MR2966749, parmPDE}.
For example, it can be shown that solutions (if they exist) are unique up to the choice of the eigenvector scalings, 
and that (if it exists) the parameterization $P$ is as smooth as $F$ -- so analytic in the present case.
Existence issues are more subtle, and involve certain \textit{non-resonance}
conditions between eigenvalues of like stability.  See \cite{MR1976079, MR2177465, mamotreto, myAMSnotes} 
for precise definition of the resonance conditions and fuller discussion.

At present we only remark that in the case of a two dimensional saddle-focus all subtleties concerning existence vanish.
This is because a single pair of complex conjugate eigenvalues cannot be resonant in the 
relevant sense.  For the manifolds studied in the present work we 
have the following: let $\xi_1, \xi_2$ denote some choice of eigenvectors, and $s \neq 0$.
Then there exists an $\epsilon > 0$ 
so that for all $|s| < \epsilon$ there exists a unique analytic solution $P \colon D \to \mathbb{C}^d$ of 
equation \eqref{eq:invEqnFlows1}, satisfying the linear constraints given in Equation \eqref{eq:firstOrder1}.
Note that in Equation \eqref{eq:firstOrder1}, $s \neq 0$ is the scaling of the eigenvectors.
In the analytic case the proof follows by the implicit function theorem, see \cite{MR2177465}.
\end{remark}

\begin{figure}[h!]
\center{
  \includegraphics[width=0.8\linewidth]{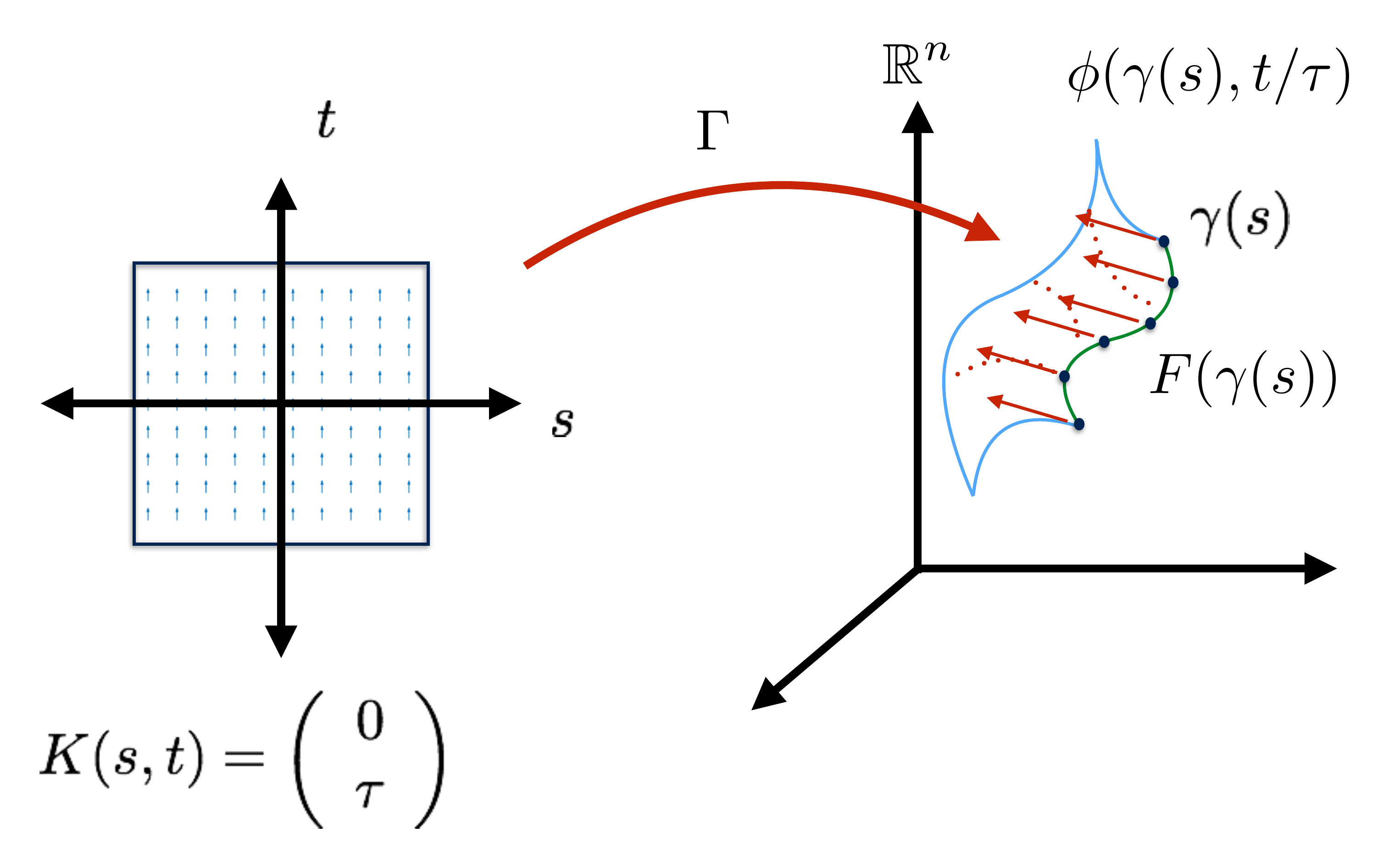}
}
\caption{\textbf{Parameterization of an advected curve:} here 
$\mathcal{M} = [-1,1]^2$ is our model manifold and the constant 
vector field $K(s,t) = (0, \tau)$ is the model dynamics.
Solving the invariance equation $DP \circ K = F(P)$ 
leads to a parameterization of the locally invariant manifold
given by advecting $\gamma$. The dynamics are conjugate to the 
flow box dynamics in $\mathcal{M}$.  That is  
$P(s,t) = \Phi(\gamma(s), t/\tau)$.} \label{fig:advectionSchematic}
\end{figure}

\subsection{Advection of transverse arcs/submanifolds}  \label{sec:growingTheManifold}
The parameterization method also describes the
locally invariant manifold generated by advecting a curve 
which is transverse to the vector field.  
More precisely, let $F \colon \mathbb{R}^d \to \mathbb{R}^d$ be a real analytic vector field and 
suppose that $\gamma \colon [-1, 1] \to \mathbb{R}^d$ is a real analytic curve.  
Assume further that $\gamma$ is transverse to the vector field $F$.
That is, suppose that
\[
\left< \gamma'(s), F \left(\gamma(s)\right) \right> \neq 0,
\]
for all $s \in [-1,1]$. (Actually it is reasonable to allow the inner product to vanish 
at at $\pm 1$ only if $F(\gamma(\pm 1)) = 0$).  Under this transversality condition,
the advected curve is a locally invariant manifold, which we wish to parameterize.

Choose $\mathcal{M} = [-1,1]^2$ and $\tau > 0$.  Define the vector field 
$K \colon [-1,1]^2 \to \mathbb{R}^2$ by 
\[
K(s,t) = \left(
\begin{array}{c}
0 \\
\tau
\end{array}
\right).
\]
That is , we take our model space to by the unit square and 
our model dynamics given by the two dimensional ``flow-box'' vector field.
We seek a parameterization $\Gamma \colon [-1,1]^2 \to \mathbb{R}^d$ 
satisfying the invariance equation \eqref{eq:invEqGeneral} 
with these choices of $\mathcal{M}$ and $K$.
This gives 
\begin{equation} \label{eq:invEqAdvArc}
\tau \frac{\partial}{\partial t} \Gamma(s,t) = F(\Gamma(s,t)), 
\end{equation}
subject to the initial condition $\Gamma(s, 0) = \gamma(s)$.

This equation simply re-states that $\Gamma$ solves the initial 
value problem with initial conditions on  $\gamma$.  
But considering the conjugacy with the flow-box vector 
field is useful, as it guarantees that the parameterization 
is well aligned with the flow in the sense discussed in 
Section \ref{sec:ideaOfTheProof}.

More precisely, by imitating the proof of Lemma \ref{lem:flowConj}, 
one checks that $\Gamma$ satisfies the flow conjugacy 
\[
\Gamma(s, t) = \Phi(\gamma(s), t/\tau), 
\]
for $(s,t) \in (-1,1)^2$.
Now choose any $t_1, t_2 \in (-1,1)$ having that $t_1 + t_2 \in (-1,1)$.
Then
\begin{align}
\Phi(\Gamma(s, t_1),   t_2/ \tau) &= \Phi(\Phi(\gamma(s), t_1/ \tau),  L t_2) \nonumber \\
&= \Phi(\gamma(s), (t_1 + t_2)/ \tau) \nonumber \\
&= \Gamma(s, t_1 + t_2) \nonumber,
\end{align}
and $\Gamma$ satisfies Equation \eqref{eq:wellAlligned}.

\begin{figure}[h!]
\center{
  \includegraphics[width=0.8\linewidth]{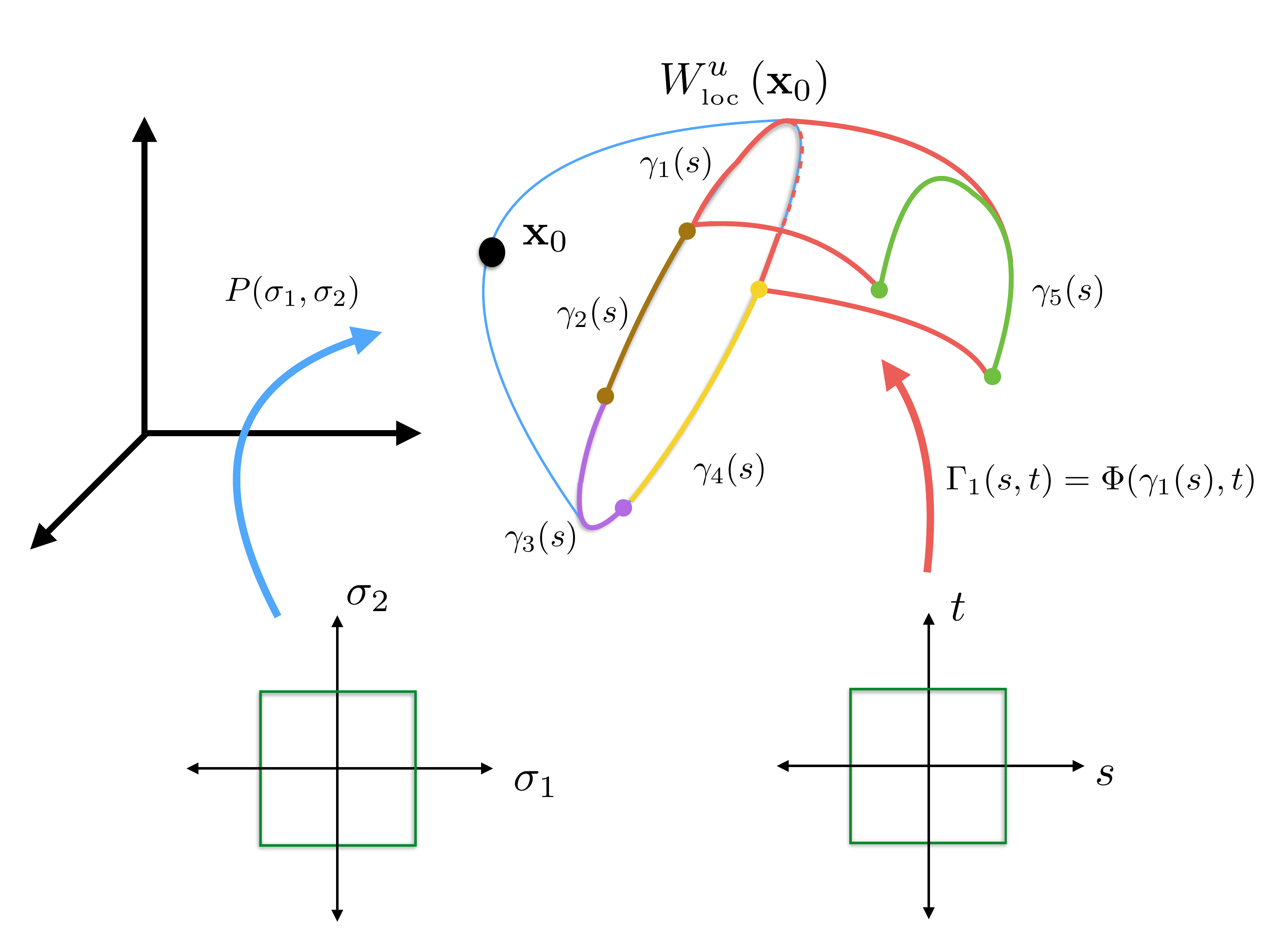}
}
\caption{The figure provides a schematic rendering of the two kinds of 
charts used on our method.  Here $P$ is the local patch containing the 
fixed point.  This chart is computed and analyzed using the Parameterization
method discussed in Section \ref{sec:parameterization}.  
The boundary of the image of $P$ is meshed into a number of 
boundary arcs $\gamma_j(s)$.  A large local manifold is ``grown'' by 
advecting these boundary arcs.  This results in the patches $\Gamma_j(s,t)$
which describe the manifold far from the equilibrium point.} \label{fig:introSchematic}
\end{figure}

\subsection{Growing the atlas by advecting boundary arcs}

Consider an equilibrium point $\mathbf{x}_0 \in \mathbb{R}^d$ 
for the smooth vector field $F$.  Let 
$B \subset \mathbb{R}^2$ denote the open unit disk in the plane, and
suppose that  
$P \colon B  \to \mathbb{R}^d$ is a real valued parameterization of a
local (un)stable manifold attached to $\mathbf{x}_0$.  
In fact, suppose that $P$ is as discussed in Section \ref{sec:parameterization},
so that the dynamics on the stable(unstable) manifold are 
conjugate to the dynamics on $B$ generated by the vector field 
$M$ defined in Remark \ref{rem:inflowOutflow}.
In particular, suppose that the vector field is inflowing(outflowing)
on the boundary of $B$, as discussed in Remark \ref{rem:inflowOutflow}.

Let $\mathbf{c}_j \colon [-1, 1] \to B$, $1 \leq j \leq M$ denote a system of 
real analytic arcs, continuous on the closed interval, and that $\mathbf{c}_1(-1) = \mathbf{c}_M(1)$.
That is, suppose that $C := \cup_{j = 1}^M c_j([-1,1])$ is a closed loop.  Assume moreover that 
$C$ has no self intersections, and has winding number $1$ with respect to the origin in $B$.
Finally,  suppose that the vector field generated by $M$ 
is nowhere tangent to $C$.  That is we assume that the vector field is 
inflowing/outflowing with respect to $C$, so that the interior of the curve
$C$ defines a fundamental domain for the stable/unstable manifold.

Then 
\[
\gamma = \bigcup_{j = 1}^M P(c_j(s)),
\]
is an outflowing/inflowing boundary for a local unstable/stable manifold of $\mathbf{x}_0$,
so that each of the sub arcs $P \circ c_j$ are everywhere transverse to the 
vector field.   
Advecting each of the curves $\gamma_j(s) = P(c_j(s))$
we obtain manifold patches 
\[
\Gamma_j(s,t) = \Phi(\gamma_j(s), t).
\]
Taking 
\[
P(B) \cup \bigcup_{j=1}^M \Gamma_j([-1,1]^2), 
\]
we obtain an analytic continuation of the local unstable manifold.  
The idea is illustrated in Figure \ref{fig:introSchematic}.

Of course now a new boundary system of arcs is giving by ``collapsing'' the 
time variable in $\Gamma$, that is we define
\[
\gamma_j^1(s) = \Gamma_j(s, 1),
\]
for $j = 1, \ldots, M$.  This new system of arcs is again the boundary of a 
local unstable manifold for $\mathbf{x_0}$, and the procedure is started again from this
new initial data.  Proceeding in this way we systematically ``grow'' a larger
and larger local portion of the unstable manifold of $\mathbf{x}_0 $.  

The procedure can be repeated indefinitely, as long as the 
arcs remain in $U$,  because  the existence and uniqueness 
of solution curves rules out the development of a tangency 
between the vector field and the curves $\gamma$.

\begin{remark}[Remeshing]
During the procedure just described, it sometimes happens that the arc
$\Gamma_j(s, 1)$ is much longer, or has much higher curvature, than the 
initial arc $\gamma_j(s)$.  In this case it is desirable to cut $\gamma_j$ up 
into an appropriate number of smaller arcs and try again.  This remeshing 
 is the most delicate part of the procedure, but it is also much closer to 
 numerical analysis.  We use the adaptive scheme developed in
  \cite{manifoldPaper1}, and refer the interested reader to that reference for more
detailed discussion of these considerations.  
\end{remark}

\section{Formal series expansions} \label{sec:formalComputations}
In this section we derive the polynomial vector field resulting from automatic differentiation
of the CRFBP, and develop formal series expansions for the stable/unstable manifolds
as well as the advected boundary arcs.  

\subsection{A polynomial problem related to the CRFBP: Automatic Differentiation}
\label{sec:poly}The idea of the automatic differentiation is to introduce new variables 
\[
u_5 = \frac{1}{r_1}, \quad
u_6 = \frac{1}{r_2}, \quad
\text{and} \quad
u_7 = \frac{1}{r_3},
\]
where the $r_j$, for $j = 1,2,3$ are as defined in Equation \eqref{eq:def_r}.  
The new variables capture the non-polynomial nonlinearity of the CRFBP, and 
are incorporated into the vector field by observing that
\begin{align*}
u_5' &= \frac{-1}{r_1^2} r_1' \\
&= \frac{-1}{r_1^2}\frac{d}{dt} \sqrt{(x(t)-x_1)^2 + (y(t) - y_1)^2}
\\
&=  \frac{-1}{r_1^2}
\frac{\frac{d}{dt}(x(t)-x_1)^2 + (y(t) - y_1)^2}{2\sqrt{(x(t)-x_1)^2 + (y(t) - y_1)^2}}
\\
& = - (u_1 - x_1)u_2 u_5^3 - (u_3 - y_1)u_4 u_5^3.
\end{align*}
Similar calculations show that differentiating $r_6$ and $r_7$ leads to the
equations
\begin{align*}
u_6' &= - (u_1 -x_2) u_2 u_6^3  - (u_3-y_2)  u_4 u_6^3,
\\
u_7' &= - (u_1-x_3) u_2 u_7^3  - (u_3-y_3) u_4 u_7^3.
\end{align*}
Expressed in these new variables the partials of $\Omega$ are
\begin{align*}
\Omega_x &=
u_1 - m_1(u_1 - x_1) u_5^3 - m_2(u_1 - x_2) u_6^3 - m_3(u_1 - x_3) u_7^3,
\\
\Omega_y &= 
 u_3 - m_1(u_3 - y_1)u_5^3 - m_2(u_3 - y_2) u_6^3 - m_3(u_3 - y_3) u_7^3,
\end{align*}
and letting  $\mathbf{u} = (u_1, u_2, u_3, u_4, u_5, u_6, u_7) \in \mathbb{C}^7$ denote
the new vector of variables, we study the new vector field 
\[
\mathbf{u}' = F(\mathbf{u}),
\]
where $F \colon \mathbb{C}^7 \to \mathbb{C}^7$ is the 
fifth order polynomial given by 
\[
F(u_1, u_2, u_3, u_4, u_5, u_6, u_7) = 
\]
\begin{equation} \label{eq:bigPoly}
\left(
\begin{array}{c}
u_2 
\\
2u_4 + u_1 - m_1 u_1u_5^3 + m_1 x_1 u_5^3 - m_2 u_1 u_6^3 + m_2 x_2  u_6^3 - m_3 u_1 u_7^3  + m_3 x_3 u_7^3 
\\
u_4
\\
-2 u_2 + u_3 - m_1 u_3  u_5^3  + m_1 y_1  u_5^3 - m_2  u_3 u_6^3  + m_2 y_2  u_6^3 - m_3 u_3 u_7^3 + m_3 y_3 u_7^3
\\
 - u_1 u_2 u_5^3 + x_1 u_2 u_5^3 -  u_3 u_4 u_5^3 + y_1 u_4 u_5^3
\\
 - u_1 u_2 u_6^3  +x_2  u_2 u_6^3  -  u_3   u_4 u_6^3   + y_2  u_4 u_6^3
\\
 - u_1 u_2 u_7^3  + x_3 u_2 u_7^3  - u_3  u_4 u_7^3  + y_3 u_4 u_7^3  
\end{array}
\right).
\end{equation}
The dynamics of $f$ and $F$ are of course not equivalent.  For example $f$ 
has singularities while $F$ is entire.  Nevertheless, in  
Appendix \ref{sec:polyFacts} we show that by 
restricting to an appropriate four dimensional 
sub-manifold of $\mathbb{R}^7$, we recover the dynamics of $f$ 
from those of $F$.  First we focus on formals series methods
for the polynomial field $F$.


\subsection{Stable/unstable manifold of a saddle-focus}
Let $\mathbf{u}_0 \in \mathbb{R}^7$ be an equilibrium solution for $F$,
$\lambda_1, \lambda_2 \in \mathbb{C}$ be eigenvalues for $DF(\mathbf{u}_0)$, and 
$\mathbf{v}_1, \mathbf{v}_2 \in \mathbb{C}^7$ be associated 
eigenvectors.
We look for a  formal power series 
\[
P(z_1, z_2) = \sum_{m=0}^\infty \sum_{n = 0}^\infty a_{mn} z_1^m z_2^n, 
\]
with $a_{mn} \in \mathbb{C}^7$ for all $m, n \in \mathbb{N}$, 
such that $P$ is a solution of Equation Equation \eqref{eq:invEqnFlows1}.  
The first order constraints given in Equation \eqref{eq:firstOrder1} are satisfied by taking
\[
a_{00} = \mathbf{u}_0, \quad \quad \quad 
a_{10} = \mathbf{v}_1, \quad \quad \quad 
\mbox{and} \quad \quad \quad 
a_{01} = \mathbf{v}_2.
\]
The higher order Taylor coefficients of $P$ are found by  
expanding both sides of the Equation \eqref{eq:invEqnFlows1}
as a power series.   

Expanding the left hand side of Equation \eqref{eq:invEqnFlows1} 
as a power series gives  
\begin{equation} \label{eq:formalSerLHS}
\lambda_1 z_1 \frac{\partial}{\partial z_1} P(z_1, z_2) + \lambda_2 z_2 \frac{\partial}{\partial z_2} P(z_1, z_2) = 
\sum_{m=0}^\infty \sum_{n=0}^\infty (m \lambda_1 + n \lambda_2) a_{mn} z_1^m z_2^n.
\end{equation}
For the right hand side of  Equation \eqref{eq:invEqnFlows1}, we study the composition $F(P(z_1, z_2))$
on the level of power series. To this end let 
\[
a_{mn} = \left(
p_{mn},  
q_{mn},  
r_{mn},  
s_{mn},  
u_{mn}, 
v_{mn},  
w_{mn}
\right),
\]  
denote the components of the Taylor coefficients, and $b_{mn}$ denote the 
power series coefficients of the composition
\begin{equation} \label{eq:formalSerRHS}
F(P(z_1, z_2)) = \sum_{m=0}^\infty \sum_{n=0}^\infty b_{mn} z_1^m z_2^n,
\end{equation}
Exploiting the formula for the Cauchy product of power series gives
\[
b_{mn} =
\]
\scalebox{.60}{
$ 
\left(
\begin{array}{c}
q_{m, n}       \\
2 s_{mn} + p_{mn} - m_1 (p * u * u * u)_{mn} + m_1 x_1 (u*u*u)_{mn}  
  - m_2 (p * v*v*v)_{mn}   + m_2 x_2  (v*v*v)_{mn} - m_3 (p *w*w*w)_{mn}  + m_3 x_3 (w*w*w)_{mn}  \\
s_{mn}     \\
-2 q_{mn} + r_{mn} - m_1 (r * u*u*u)_{mn} + m_1 y_1 (u*u*u)_{mn}    
 - m_2 (r*v*v*v)_{mn}  + m_2 y_2 (v*v*v)_{mn} - m_3 (r*w*w*w)_{mn} + m_3 y_3 (w*w*w)_{mn}  \\
 - (p*q*u*u*u)_{mn}  +  x_1 (q *u*u*u)_{mn}   - (r*s*u*u*u)_{mn} + y_1 (s*u*u*u)_{mn}   \\
  - (p*q*v*v*v)_{mn}  + x_2 (q *v*v*v)_{mn} - (r*s*v*v*v)_{mn} + y_2 (s*v*v*v)_{mn}       \\
 - (p*q*w*w*w)_{mn}  + x_3 (q *w*w*w)_{mn} - (r*s*w*w*w)_{mn} + y_3 (s*w*w*w)_{mn}  
\end{array}
\right).
$
}

\medskip

\noindent Equating the right hand sides of Equations \eqref{eq:formalSerLHS} and \eqref{eq:formalSerRHS} and 
matching like powers of $z_1, z_2$, 
we obtain the countably many vector valued equations
\begin{equation} \label{eq:preHomological}
(m \lambda_1 + n \lambda_2) a_{mn} = b_{mn}.
\end{equation}
This simple looking expression is of little value until we 
understand the dependence of the right hand side on the 
unknown Taylor coefficients $\left\{ a_{mn} \right\}$.  

%
%

The nonlinear terms in the expression for $b_{mn}$ are all of the form 
$(a * a * a)_{mn}$ -- a cubic Cauchy product coefficient, 
$(a * b * b * b)_{mn}$ -- a quartic Cauchy product coefficient, or
$(a * b * c * c * c)_{mn}$ -- a quintic Cauchy product coefficient.    
For example the cubic terms have the general form
\begin{align*}
(a*a*a)_{mn}  &= \sum_{j_1 = 0}^m \sum_{j_2 = 0}^{j_1} 
                            \sum_{k_1 = 0}^n \sum_{k_2 = 0}^{k_1} 
                            a_{m-j_1, n-k_1} a_{j_1 - j_2, k_1-k_2} a_{j_2, k_2}.
\end{align*}
Note that there are three ways that an $a_{mn}$ term can appear in the 
summand.  
It could be that $j_1 = k_1 = 0$, that $j_1 = m$ and $k_1 = n$, or that 
$j_2 = m$ and $k_3 = n$.  In each of these three cases the summand reduces to 
$a_{00}^2 a_{mn}$.  Since this can happen exactly three ways, the entire sum contains a 
term of the form $3 a_{00}^2 a_{mn}$.  Moreover this is the only term in the sum
depending on $a_{mn}$.  

Motivated by these comments we define a new triple product 
$(a \hat * a \hat * a)$ whose $mn$ term is given by 
\[
(a \hat * a \hat * a)_{mn} := (a*a*a)_{mn}  - 3 a_{00}^2  a_{mn}.
\]
Extending these considerations to the higher order products
leads to a new quartic product
$(a \hat* b \hat* b \hat* b)$ defined by 
\[
(a \hat* b \hat* b \hat* b)_{mn} := (a * b * b * b)_{mn} - 3 a_{00} b_{00}^2 b_{mn} - b_{00}^3 a_{mn}, 
\]
and a new quintic product $(a \hat* b \hat* c \hat* c \hat* c)$ defined by 
\[
(a \hat* b \hat* c \hat* c \hat* c)_{mn} :=  (a * b * c * c * c)_{mn} 
- b_{00} c_{00}^3 a_{mn} - a_{00} c_{00}^3 b_{mn}  - 3 a_{00} b_{00} c_{00}^2 c_{mn}.
\]
Expanding all Cauchy products using 
the hat products just defined leads to    
\[
b_{mn} = 
\]
\scalebox{.55}{
$\left(
\begin{array}{c}
q_{m n}       \\
2 s_{mn} + p_{mn} - m_1(u_{00}^3 p_{mn} + 3 p_{00} u_{00}^2 u_{mn} ) 
m_1 x_1 3 u_{00}^2 u_{mn} - m_2 (v_{00}^3 p_{mn} + 3 p_{00}v_{00}^2 v_{mn}) 
+ m_2 x_2 3 v_{00}^2 v_{mn} - m_3 (w_{00}^3 p_{mn} + 3 p_{00} w_{00}^2 w_{mn})
+ m_3 x_3 3 w_{00}^3 w_{mn}
+ Q_{mn}   \\
 \\
s_{mn}     \\
-2 q_{mn} + r_{mn} - m_1 (u_{00}^3 r_{mn} + 3r_{00}u_{00}^2 u_{mn}) + 3 m_1 y_1 u_{00}^2 u_{mn} 
- m_2(v_{00}^3 r_{mn} + 3 r_{00} v_{00}^2 v_{mn}) + 3 m_2 y_2 v_{00}^2 v_{mn} - 
m_3 (w_{00}^3 r_{mn} + 3r_{00} w_{00}^2 w_{mn}) + 3 m_3 y_3 w_{00}^2 w_{mn}
+ S_{mn}  \\
- (q_{00}u_{00}^3 p_{mn} + p_{00} u_{00}^3 q_{mn} + 3 p_{00} q_{00} u_{00}^2 u_{mn})
+ x_1 (u_{00}^3 q_{mn} + 3 q_{00} u_{00}^2 u_{mn})
- (s_{00}u_{00}^3 r_{mn} + r_{00} u_{00}^3 s_{mn} + 3 r_{00} s_{00} u_{00}^2 u_{mn})
+ y_1 (u_{00}^3 s_{mn} + 3 s_{00} u_{00}^2 u_{mn})
+U_{mn} \\
  - (q_{00}v_{00}^3 p_{mn} + p_{00} v_{00}^3 q_{mn} + 3 p_{00} q_{00} v_{00}^2 v_{mn})
+ x_2 (v_{00}^3 q_{mn} + 3 q_{00} v_{00}^2 v_{mn})
- (s_{00}v_{00}^3 r_{mn} + r_{00} v_{00}^3 s_{mn} + 3 r_{00} s_{00} v_{00}^2 v_{mn})
+ y_2 (v_{00}^3 s_{mn} + 3 s_{00} v_{00}^2 v_{mn})
+  V_{mn}   \\
 - (q_{00}w_{00}^3 p_{mn} + p_{00} w_{00}^3 q_{mn} + 3 p_{00} q_{00} w_{00}^2 w_{mn})
+ x_3 (w_{00}^3 q_{mn} + 3 q_{00} w_{00}^2 w_{mn})
- (s_{00}w_{00}^3 r_{mn} + r_{00} w_{00}^3 s_{mn} + 3 r_{00} s_{00} w_{00}^2 w_{mn})
+ y_3 (w_{00}^3 s_{mn} + 3 s_{00} w_{00}^2 w_{mn})
+W_{mn}
\end{array}
\right),
$
}

\medskip

\noindent where

\scalebox{.7}{
$Q_{mn} :=  -m_1(p \hat * u \hat * u \hat * u)_{mn} + m_1 x_1 (u \hat* u \hat *u)_{mn}  
  - m_2 (p \hat* v\hat*v\hat*v)_{mn}   + m_2 x_2  (v\hat*v\hat*v)_{mn} 
  - m_3 (p \hat*w\hat*w\hat*w)_{mn}  + m_3 x_3 (w\hat*w\hat*w)_{mn},$
}

\scalebox{.7}{$
S_{mn} :=  -m_1 (r \hat* u\hat*u\hat*u)_{mn} + m_1 y_1 (u\hat*u\hat*u)_{mn}    
 - m_2 (r\hat*v\hat*v\hat*v)_{mn}  + m_2 y_2 (v\hat*v\hat*v)_{mn} - m_3 (r\hat*w\hat*w\hat*w)_{mn} 
 + m_3 y_3 (w\hat*w\hat*w)_{mn}$,
}

\scalebox{0.7}{
$
U_{mn} := 
 - (p\hat*q\hat*u\hat*u\hat*u)_{mn}  +  x_1 (q \hat*u\hat*u\hat*u)_{mn}   - (r\hat*s\hat*u\hat*u\hat*u)_{mn} + y_1 (s\hat*u\hat*u\hat*u)_{mn},
$}

\scalebox{0.7}{
$
V_{mn} := 
 - (p\hat*q\hat*v\hat*v\hat*v)_{mn}  +  x_2 (q \hat*v\hat*v\hat*v)_{mn}   - (r\hat*s\hat*v\hat*v\hat*v)_{mn} + y_2 (s\hat*v\hat*v\hat*v)_{mn},
$}

\noindent and

\scalebox{0.7}{
$
W_{mn} := 
 - (p\hat*q\hat*w\hat*w\hat*w)_{mn}  +  x_3 (q \hat*w\hat*w\hat*w)_{mn}   - (r\hat*s\hat*w\hat*w\hat*w)_{mn} + y_3 (s\hat*w\hat*w\hat*w)_{mn}.
$}

\medskip


Note that $Q_{mn}, S_{mn}, U_{mn}, V_{mn}$, and $W_{mn}$ are composed of hat products
and hence depend on lower order terms.  Moreover the terms highest order terms  
$p_{mn}$, $q_{mn}$, $r_{mn}$, $s_{mn}$, $u_{mn}$, $v_{mn}$, and $w_{mn}$ appear only linearly.
We write the linear terms in matrix-vector form and compare the resulting 
matrix with the formula for the Jacobian $DF(a_{00})$ 
given in Equation \eqref{eq:polyJacobi}
and see that 
\[
b_{mn } = DF(a_{00}) a_{mn} + c_{mn}, 
\]
where 
\begin{equation}\label{eq:defOfcmn}
c_{mn} :  =  \left(
0, 
Q_{mn}, 
0, 
S_{mn}, 
U_{mn}, 
V_{mn}, 
W_{mn}
\right).
\end{equation}
Finally, recalling that $a_{00} = \mathbf{u}_0$, we rewrite Equation \eqref{eq:preHomological} 
as
\begin{equation} \label{eq:homEq}
\left[DF(\mathbf{u}_0) - (m \lambda_1 + n \lambda_2) \mbox{Id} \right] a_{mn} = - c_{mn}.
\end{equation}
Equation \eqref{eq:homEq} is called the homological equation for 
$P$, and provides one linear equation for each Taylor coefficient $a_{mn}$, with
right hand side depending only on 
lower order coefficients.  

\begin{remark}[The formal series is well defined]
Since 
\[
\lambda_1 = -\alpha + i \beta,  
\]
and $\lambda_2 = \overline{\lambda_1}$, we have that 
$m \lambda_1 + n \lambda_2$ is never an eigenvalue of $DF(a_{00})$, and 
the homological equations are uniquely solvable to all orders $m + n \geq 2$.  
This gives a direct proof, in the explicit case of the automatically differentiated CRFBP, 
of the more general Existence result alluded to in
Remark \ref{eq:invEqnFlows1}.  That is, the homological 
equations are uniquely solvable, hence we can compute the 
coefficients of $P$ to any desired finite order.   

This discussion is summarized in Algorithm \ref{alg:equilibrium}.
Supposing that $\{\bar{a}_{mn}\}_{0 \leq m+n \leq N} \in \mathbb{C}^7$ is the
output the algorithm,  then 
\[
P^N(z_1, z_2) = \sum_{0 \leq m+n \leq N} \bar{a}_{mn} z_1^m z_2^n,
\] 
is our approximation of the stable manifold.  Running the algorithm on 
interval enclosures of the input data, and using a validated 
interval linear system solver results in 
interval coefficients $\bar{a}^N = \{\bar{a}_{mn}\}_{0 \leq m+n \leq N}$
enclosing the actual Taylor coefficients of the true solution $P$.
\end{remark}

\begin{remark}[Scaling the eigenvectors] \label{rem:eigScaling}
Since the homological equations uniquely determine all the coefficients of order $m + n \geq 2$, 
we also have direct proof, again in the case of the automatically differentiated CRFBP,
of the uniqueness result cited in Remark \ref{eq:invEqnFlows1}.  
That is, the coefficients 
of $P$ are uniquely determined once the scalings of the eigenvectors are fixed.  
In fact more is true.  The coefficients of different parameterizations corresponding to 
different eigenvector scalings are related in a quantifiable way.  

Suppose that $P, \tilde P \colon D \to \mathbb{C}^d$ are parameterizations satisfying Equation
\eqref{eq:invEqnFlows1}, but suppose that $P$ satisfies the first order constraints 
\[
P(0,0) = \mathbf{x}_0, 
\quad \quad \quad 
\frac{\partial}{\partial z_1} P(0,0) = \mathbf{\xi}_1, 
\quad \quad \quad 
\frac{\partial}{\partial z_2} P(0,0) = \mathbf{\xi}_2, 
\] 
while 
\[
\tilde{P}(0,0) = \mathbf{x}_0, 
\quad \quad \quad 
\frac{\partial}{\partial z_1} \tilde{P}(0,0) = s \mathbf{\xi}_1, 
\quad \quad \quad 
\frac{\partial}{\partial z_2} \tilde{P}(0,0) = s \mathbf{\xi}_2, 
\] 
for some $s \neq 0$.  That is: $P$ and $\tilde P$ are parameterizations corresponding 
to different choices of eigenvector scalings.
Let $\{a_{mn}\}_{(m,n) \in \mathbb{N}^2}$ and $\{\tilde{a}_{mn}\}_{(m,n) \in \mathbb{N}^2}$
denote, respectively the power series coefficients of $P$ and $\tilde P$.  
Then for each $m + n \geq 2$ the coefficients have
\[
\tilde{a}_{mn} = s^{m+n} a_{mn}.
\]
A proof of this claim is found in Section $3.3$ of \cite{myAMSnotes}.

The formula says we obtain the coefficients of any parameterization 
once we know the coefficients of one.
This means we can select the desired exponential growth rate of the parameterization 
coefficients by adjusting the scaling of the eigenvectors.  Larger choices of $|s|$ correspond to 
larger embeddings of $D$ in phase space and hence parameterization of larger patches
of the local stable manifold.  Numerical optimization schemes for the parameterization method
based on these observations have been developed by \cite{maximeJPMe}.
\end{remark}

\begin{remark}[Complex conjugate coefficients] \label{rem:complexConjCoeff}
Returning to Equation \eqref{eq:homEq}, note that
\[
\overline{DF(p) - (m \lambda_1 + n \lambda_2) \mbox{Id}} = DF(p) - (m \lambda_2 + n \lambda_1) \mbox{Id},
\]
as $DF(p)$ is a real matrix and $\lambda_1, \lambda_2$ are complex conjugates.  
Moreover, we have that the right hand side $c_{mn}$ is a real polynomial function of lower order coefficients.  
Choosing complex conjugate eigenvectors for the first order constraints gives
\[
\overline{a_{10}} = \overline{\mathbf{v}_1} =\mathbf{v}_2 = a_{01},  
\]
and an induction argument with the above as the base case shows that   
\[
\overline{a_{mn}} = a_{nm},
\]
for all $m + n \geq 2$.  
\end{remark}

\begin{remark}[Implementation of the hat product] \label{rem:implementHat}
When the hat products are implemented numerically it is not necessary to 
compute the terms of the Cauchy product and the subtract out the 
terms depending on $(m,n)$ as in the definition.  Instead these terms are
simply ignored when computing the sums, that is they are not included 
in the first place.  
\end{remark}

\begin{algorithm}[t!]
\caption{Stable manifold approximation -- equilibrium $p$}\label{alg:equilibrium}
\begin{algorithmic}[1]
\Function{Param\_Equilibrium}{$\mathbf{u}_0, \mathbf{v}_1,\mathbf{v}_2, \lambda_1, \lambda_2, N$} 
\State $a_{00} \gets \mathbf{u}_0$
\State $a_{10} \gets \mathbf{v}_1$
\State $a_{01} \gets \mathbf{v}_2$
\ForAll{$n=2$ to $N$}
\ForAll{$k=0$ to $n$}
\State  $c_{n-k, k} =  \mbox{computeHomEqRHS\_CRFBP}(\{a_{jk}\}_{0 \leq j+k \leq n}, n-k, k)$
 \State $a_{n-k k} \gets  \mbox{LinSysSolve}(DF(p) - ((n-k)\lambda_1 + k \lambda_2)\mbox{Id}, -c_{n-k k})$  
\EndFor
\EndFor
\State \Return $\displaystyle\{a_{mn}\}_{0\leq m +n \leq N}$
\EndFunction 
\State \textbf{Remark $1$:} note that the computation of the $c_{mn}$
is specific to the CRFBP, but  the rest of the algorithm is general.
The function \verb|computeHomEqRHS_CRFBP| implements the  
evaluation of Equation \eqref{eq:defOfcmn}, which is itself comprised of
Cauchy ``hat'' products.  The hat products are simply finite sums. 
\State \textbf{Remark $2$:} if $\mathbf{u}_0, \mathbf{v}_1,\mathbf{v}_2, \lambda_1, \lambda_2$
are interval enclosures of the first order data, 
and if  \verb|computeHomEqRHS_CRFBP| is implemented in 
interval arithmetic, and if the linear system
solver solver returns a validated interval enclosure of the solution, then 
\verb|Param_Equilibrium| returns interval enclosures of the $\{a_{mn}\}$.  
Linear system solvers with mathematically 
rigorous interval error bounds are discussed in 
\cite{MR2652784} and implemented in IntLab\cite{Ru99a}. 
\end{algorithmic}
\end{algorithm}

\subsection{Formal integration of a material line}
Let $\gamma \colon [-1,1] \to \mathbb{R}^d$ be a real analytic curve with 
Taylor series expansion 
\[
\gamma(s) = \sum_{n=0}^\infty \gamma_n s^n.
\]
We look for a power series solution $\Gamma \colon [-1,1]^2 \to \mathbb{R}^d$ of 
Equation \eqref{eq:invEqAdvArc}, and write
\[
\Gamma(s,t) = \sum_{n=0}^\infty \Gamma_n(s) t^n,
\] 
where
\[
\Gamma_n(s) = \sum_{m=0}^\infty \Gamma_{mn} s^t.
\]
Let 
\[
\Gamma_n(s) = 
\left(
\begin{array}{c}
p_n(s) \\
q_n(s)  \\
r_n(s)  \\
s_n(s)  \\
u_n(s)  \\
v_n(s)  \\
w_n(s) 
\end{array}
\right),
\quad \quad \quad 
\mbox{and}
\quad \quad \quad 
\Gamma_{mn} = 
\left(
\begin{array}{c}
p_{mn} \\
q_{mn}  \\
r_{mn}  \\
s_{mn} \\
u_{mn} \\
v_{mn} \\
w_{mn} 
\end{array}
\right),
\]
denote the component functions and coefficients.  

Expanding the left hand side of 
Equation \eqref{eq:invEqAdvArc} as a power series leads to 
\[
\tau \frac{\partial}{\partial t} \Gamma(s, t) = \sum_{n=0}^\infty
(n+1) \Gamma_{n+1}(s) t^n.
\]
On the other hand, we write 
\[
F(\Gamma(s,t)) = \sum_{n=0}^\infty b_n(s) t^n, 
\]
where 
\[
b_n(s) = \sum_{m=0}^\infty b_{mn} s^m t^n.
\]
Employing the Cauchy product
leads to 
\[
b_n(s) = 
\]
\scalebox{.60}{
$
\left(
\begin{array}{c}
q_n       \\
2 s_n + p_n - m_1 (p * u * u * u)_n + m_1 x_1 (u*u*u)_n  
  - m_2 (p * v*v*v)_n   + m_2 x_2  (v*v*v)_n - m_3 (p *w*w*w)_n  + m_3 x_3 (w*w*w)_n  \\
s_n     \\
-2 q_n + r_n - m_1 (r * u*u*u)_n + m_1 y_1 (u*u*u)_n    
 - m_2 (r*v*v*v)_n  + m_2 y_2 (v*v*v)_n - m_3 (r*w*w*w)_n + m_3 y_3 (w*w*w)_n  \\
 - (p*q*u*u*u)_n  +  x_1 (q *u*u*u)_n   - (r*s*u*u*u)_n + y_1 (s*u*u*u)_n   \\
  - (p*q*v*v*v)_n  + x_2 (q *v*v*v)_n - (r*s*v*v*v)_n + y_2 (s*v*v*v)_n       \\
 - (p*q*w*w*w)_n  + x_3 (q *w*w*w)_n - (r*s*w*w*w)_n + y_3 (s*w*w*w)_n  
\end{array}
\right).
$
}

\smallskip

\noindent Matching like powers and solving for $\Gamma_{n+1}(s)$ leads to 
the recursion relation
\[
\Gamma_{n+1}(s) = \frac{1}{\tau (n+1)} b_n(s),
\]
where $b_n(s)$ is a function of only lower order terms.  By iterating this 
recursion, initializing with $\Gamma_0(s) = \gamma(s)$, we compute 
the Taylor expansion of $\Gamma(s,t)$ to any desired order.

In practice we are only able, for each value of $n$, to 
compute the expansion for $\Gamma_{n+1}(s)$ to 
finite order in $m$ .  That is, for fixed $n$ we solve the recursion relations 
\begin{equation} \label{eq:gammaRecurEq}
\Gamma_{m(n+1)} = \frac{1}{\tau (n+1)} b_{mn}, 
\end{equation}
where 
\[
b_{mn} = 
\]
\scalebox{.60}{
$ 
\left(
\begin{array}{c}
q_{mn}       \\
2 s_{mn} + p_{mn} - m_1 (p * u * u * u)_{mn} + m_1 x_1 (u*u*u)_{mn}  
  - m_2 (p * v*v*v)_{mn}   + m_2 x_2  (v*v*v)_{mn} - m_3 (p *w*w*w)_{mn}  + m_3 x_3 (w*w*w)_{mn}  \\
s_{mn}     \\
-2 q_{mn} + r_{mn} - m_1 (r * u*u*u)_{mn} + m_1 y_1 (u*u*u)_{mn}    
 - m_2 (r*v*v*v)_{mn}  + m_2 y_2 (v*v*v)_{mn} - m_3 (r*w*w*w)_{mn} + m_3 y_3 (w*w*w)_{mn}  \\
 - (p*q*u*u*u)_{mn}  +  x_1 (q *u*u*u)_{mn}   - (r*s*u*u*u)_{mn} + y_1 (s*u*u*u)_{mn}   \\
  - (p*q*v*v*v)_{mn}  + x_2 (q *v*v*v)_{mn} - (r*s*v*v*v)_{mn} + y_2 (s*v*v*v)_{mn}       \\
 - (p*q*w*w*w)_{mn}  + x_3 (q *w*w*w)_{mn} - (r*s*w*w*w)_{mn} + y_3 (s*w*w*w)_{mn}  
\end{array}
\right).
$
}

\begin{algorithm}[t!]
\caption{Approximate integration of a material line}\label{alg:flowMaterialLine}
\begin{algorithmic}[1]
\Function{Flow\_Line}{$\left\{\gamma_{mn}\right\}_{m=0}^{M}, N$} 
\State \ForAll{$m = 0$ to $M$}
\State $\Gamma_{m0} = \gamma_m$
\EndFor
\State \ForAll{$n=1$ to $N$}
\ForAll{$m=0$ to $M$}
\State  $b_{m, n} =  \mbox{computeRecursionRHS\_CRFBP}(\{\Gamma_{jk}\}_{j=0, k=0}^{mn}, m, n)$
 \State $\Gamma_{m n+1} \gets \frac{1}{\tau(n+1)} b_{mn}$
\EndFor
\EndFor
\State \Return $\displaystyle \{\Gamma_{mn}\}_{m=0, n=0}^{M, N}$
\EndFunction 
\State \textbf{Remark $1$:}
The first loop initializes $\Gamma(s,0)$ using the given expansion for $\gamma$.
The next loop compute the Taylor coefficients from order $n = 1$ to order $n = N$, 
using the recursion relation of Equation \eqref{eq:gammaRecurEq}.
The inner loop computes for each $n$ 
the coefficient of $\Gamma_{mn}$ to order $m = M$.
\State \textbf{Remark $2$:}
The function \verb|computeRecursionRHS_CRFBP| 
just implements the explicit formula given below 
Equation \eqref{eq:gammaRecurEq}. When this function is 
implemented in interval arithmetic then the output
results in rigorous interval encloses the correct Taylor coefficients. 
\end{algorithmic}
\end{algorithm}

Algorithm \ref{alg:flowMaterialLine} summarizes the preceding discussion.
Suppose that $\left\{\Gamma_{mn} \right\}_{m=0 n=0}^{MN}$ is the 
output of the algorithm.  Then 
\[
\Gamma^{MN}(s,t) = \sum_{m=0}^M \sum_{n=0}^N \Gamma_{mn} s^m t^n,  
\]
is our polynomial approximation of the advected curve.

\begin{remark}[Rescaling to control the coefficient growth]
From a practical point of view, the most difficult part of growing the local 
stable/unstable manifolds
is managing the rescaling and recentering of the manifold patches in an 
efficient and automated way.  These technical details are managed using the 
techniques developed in \cite{manifoldPaper1}.  The reference just cited 
is devoted to the treatment of these issues.
\end{remark}

\subsection{Remarks on the tail validations for the formal series}
Suppose that $f \colon D \to \mathbb{C}$ is an analytic function
and $f^N$ is a polynomial approximation of $f$. For example
in the discussion above $f$ could be a component of the 
local stable/unstable manifold parameterization or a component of 
a chart $\Gamma$ for an advected boundary arc. 
The goal of any validated numerical method 
is to obtain, with computer assistance, a mathematically 
rigorous error bound $r > 0$ of the form 
\[
\sup_{(z_1, z_2) \in D} \left| f^N(z_1, z_2) - f(z_1, z_2) \right| \leq r.
\]
Note that if $f$ is analytic on $D$ then 
\[
h(z_1, z_2) =  f(z_1, z_2) - f^N(z_1, z_2),
\]
defines an analytic function on $D$, and $r$ provides a bound on the 
supremum norm of $h$.  

It follows that 
\[
f(z_1, z_2) =  f^N(z_1, z_2) + h(z_1, z_2),
\]
so that 
\[
\sup_{(z_1, z_2) \in D} \left| f(z_1, z_2) \right| \leq \sup_{(z_1, z_2) \in D} \left| f^N(z_1, z_2) \right| + r,
\]
where the supremum of the polynomial $f^N$ is easily bound numerically.
Considering derivatives, we have for example that 
\[
\frac{\partial}{ \partial z_1 } f(z_1, z_2) =  \frac{\partial}{ \partial z_1 }  f^N(z_1, z_2) 
+ \frac{\partial}{ \partial z_1 }  h(z_1, z_2),
\]
where the partial derivative of a polynomial is once again easy to compute numerically.
Derivatives of the truncation error $h$ are obtained on any smaller domain
than $D$ by using classical estimates
or complex analysis.  See for example the Cauchy Bounds of Lemma 2.9 
of \cite{MR3068557}.  Of course these remarks generalize to other/higher order
derivatives of $f$ in the obvious way.  

Implementation of the computer assisted error analysis
used for the CRFBP is discussed in the companion paper 
\cite{thisPaperII}.  Roughly speaking, the idea of the 
error analysis is to exploit that the unknown function 
solves a functional equation (invariance equation for the parameterization 
method in the case of $P$ and the differential equation in the case of
$\Gamma$).  The functional equation is used in conjunction with the 
known polynomial approximation $P^N$ or $\Gamma^N$ to derive
a fixed point problem for the truncation error $h(z_1, z_2)$ on an 
appropriate function space.  In fact, the 
computer assisted argument is actually formulated on an 
infinite sequence space of power series coefficients.
This is an often used  convenience going back to the work of
\cite{MR648529, MR727816, MR883539}.

The approach of \cite{thisPaperII} is similar to 
that of \cite{MR3281845}, and also to the techniques worked out for 
some simple example problems in \cite{myAMSnotes}.   
Fuller discussion of validated numerical 
integrators -- that is, computer assisted proofs for initial value problems -- 
is found in the works of  
\cite{cnLohner, MR1701385, MR1870856, 
MR1652147,MR2312532,
MR1930946, MR3022075, MR3281845}
and in the references found therein.  
See also the book of \cite{MR2807595}.
In Section \ref{sec:results} we only remark on the results of the 
validated numerical computations carried out using the 
methods of \cite{thisPaperII} and suppress the technical details
of how the bounds are actually obtained.  

\section{Results} \label{sec:results}

\subsection{Computer assisted proof of a transverse homoclinic} 
\label{sec:finalresults}
In this section, we conclude the proof of Theorem \ref{thm:ourThm1}. Specifically, we verify 
the hypotheses of Theorem \ref{thm:devaney} by explicitly computing a validated intersection of the stable 
and unstable manifolds and then proving transversality for this intersection. 
Throughout this section we use the following notation. Recall that $\norm{z}_{\infty}$ denotes the max 
norm on $\cc^d$ as defined in Appendix \ref{sec:norms}. For any $r \in \rr^+$, let 
\[
B_r(z_0) = \{z \in \mathbb{C}^d : \ \norm{z - z_0}_{\infty} < r \}
\]
denote the open ball of radius $r$, and let $C_r^{d}$ denote the space of bounded 
analytic functions, $g: B_r(0) \to \cc^d$, which we equip with the norm
\[
\norm{g}_r = \sup\limits_{z \in B_r(0)} \norm{g(z)}_{\infty}.
\]

\subsubsection{Local manifolds}
\label{sec:localmanifolds}
The local stable and unstable manifolds are each computed to order 7 
using the methods discussed in Section \ref{sec:parameterization}. This yields a 
parameterization for $W^u_{\text{loc}}(\mathbf{x_0})$ of the form
\[
Q(z_1,z_2) = \sum_{0 \leq m+n \leq 7}  q_{m,n}z_1^{m}z_2^n + h_{Q}(z_1,z_2)
\]
where each $q_{m,n} \in \intrr ^7$ is an interval vector, and $h_{Q} \in C_1^{2}$ with a rigorous error bound
\[
\norm{h_{Q}}_1 \leq r_Q = .5048 \times 10^{-16}. 
\]
Any arc segment transverse to the linear flow can be parameterized, and then rigorously lifted through the local 
parameterization to obtain a parameterized arc segment, $\gamma^u: [-1,1] \to \rr^7$, such that 
\[
\gamma^u([-1,1]) \subseteq W^u_{\text{loc}}(\mathbf{x_0}),
\]
which is transverse to $\Phi$. For this validation, we lifted a piecewise linear closed curve transverse to the linear flow to obtain a collection of charts
\[
\gamma_*^u(s) = \sum_{n = 0}^{15}  a_ns^n + h^u_{*}(s),
\]
where $a_n \in \intrr^7$, $h^u_{*} \in C^1_1$, and a rigorous bound, $\norm{h^u_*}_1 \leq r^u = .2936 \times 10^{-13}$, which holds for all 20 unstable charts. 

A similar parameterization is carried out for the stable manifold to obtain a collection of charts mapping into $W^s_{\text{loc}}(\mathbf{x_0})$ of the form
\[
\gamma_*^s(s) = \sum_{n = 0}^{15}  b_ns^n + h_*^s(s)
\]
with $\norm{h_*^s}_1 \leq r^s = .3050 \times 10^{-13}$ for all 20 stable charts.


%
\begin{figure}[t!]
\center{
	\includegraphics[width = .49\textwidth,keepaspectratio=true]{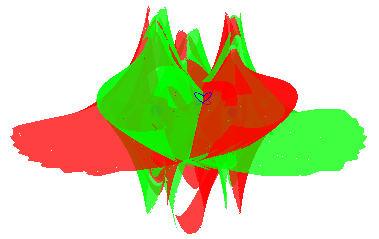}
	\includegraphics[width = .49\textwidth,height = 2.25in]{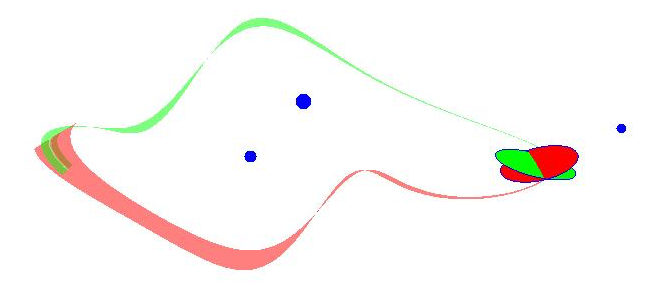}
	}
	\caption{Stable/unstable manifold atlases for the CRFBP. 
        Projection on the $(x,y,\dot{y})$ coordinates.
	(Left) thousands of polynomial chart maps
	obtained by advecting the boundaries of the local parameterizations.  Stable shown green and unstable 
	red.  The local patches are the disks with black boundaries in the center of the frames. (Right) a candidate 
	intersection extracted form the atlas by pairwise checking the patches for potential intersections.  The  
	parent charts are shown all the way back to the local manifold boundaries. Note that the Figure illustrates
	three dimensional projections of four dimensional objects, and the intersection of the manifolds in a 
	straight line through the center of the local manifolds is a projection error.  The actual local manifolds
	intersection only at the saddle-focus equilibrium point.}\label{fig:CRFBP_connection}
\end{figure}

%
%

\subsubsection{Growing an atlas} \label{sec:atlas}
We obtain the global stable/unstable manifolds by advecting the initial charts for the local stable/unstable manifolds 
using the formal series calculations discussed in Section \ref{sec:formalComputations} ,
the computer assisted validation techniques of \cite{thisPaperII}, and the automatic remeshing and 
rescaling algorithms developed in \cite{manifoldPaper1}.
For example, advecting the boundaries of the unstable/stable manifolds for 
five units of forward/backward leads to the manifold atlases illustrated in the left frame of 
Figure \ref{fig:CRFBP_connection}.
The stable atlas is comprised of $6,546$ and the unstable manifold comprised of $6,753$ 
polynomial chart maps.  Each chart is computed to order $15$ in the spatial direction and 
order $50$ in time.  The validated manifold patches have error bounds ranging from about
$10^{-11}$ close to the parameterized local manifolds, to $10^{-4}$ near the end of the calculation.

Next, we search the results for homoclinic connection candidates, which are then 
post-processed using the computer assisted methods of proof discussed in 
Section \ref{prop:finiteDimNewton}.  A candidate connection is 
illustrated in the right frame of Figure \ref{fig:CRFBP_connection}.  The proof is discussed in 
more detail in the next section.  For now we only remark that if the proof of the candidate 
connection fails then we can recover the ``parent'' charts of the candidates all the way back 
to the boundary of the invariant manifold and we can recompute to forward advection 
using increased accuracy.  This is much cheaper than recomputing the entire atlas with 
increased accuracy.

\subsubsection{Existence of a transverse intersection}
After identifying a pair of charts with potential transverse intersection as discussed above, 
we use Lemma \ref{lem:aPosEnergyLemma1} to establish the existence of a true intersection 
point.  Recall that the Lemma restates the existence of an intersection of the stable/unstable 
manifolds in terms of the solutions of a certain zero finding problem, defined in terms of the candidate 
charts.  The next step is to run a non-rigorous numerical Newton method 
to refine the approximate intersections.  This numerical calculation is 
carried out using the polynomial (truncated) part of the charts. 

Specifically, given a pair of charts, $\Gamma^u,\Gamma^s$, which are candidates for an intersection of the stable/unstable manifolds, we run Newton's method (non-rigorously) to obtain parameters, $(\overbar{s}, \overbar{t}, \overbar{\sigma}) \in \rr^3$, such that  $\Gamma^s(\overbar{s},\overbar{t}) \approx \Gamma^u(\overbar{\sigma},0)$. Now, we compute the a-posteriori estimates discussed in Theorem \ref{prop:finiteDimNewton}. If the estimates satisfy the theorem, then the a-posteriori validation succeeds and we check the condition described in Lemma \ref{lem:aPosEnergyLemma1} to conclude the existence of a true homoclinic in an explicit neighborhood of $\Gamma^u(\overbar{s},\overbar{t})$. In fact the argument is very similar to the example proof discussed in Section 
\ref{sec:NKex1}. Finally, we verify the hypothesis of Lemma \ref{lem:aPosEnergyLemma2} and conclude that the homoclinic is transverse. 
 
Below, we explicitly describe the a-posteriori estimates obtained for a pair of charts lying in the intersection shown on the right of Figure \ref{fig:CRFBP_connection}. Recall that the stable chart may be decomposed as $\Gamma^s = P^N + P^{\infty}$ where $P^N$ is a polynomial with $(M,N) = (15,50)$ coefficients and $P^{\infty} \in C^2_1$. A similar decomposition for the unstable chart is given by $\Gamma^u = Q^N + Q^{\infty}$. We define
\[
F(s,t,\sigma) = P(s,t) - Q(\sigma,0) = F^N(s,t,\sigma) + F^{\infty}(s,t,\sigma),
\] 
and applying Newton iteration to $F$ we find $(\overbar{s}, \overbar{t}, \overbar{\sigma}) = (-.1421, -.0682, .0946) = \overbar{x}$ satisfying $F^N(\overbar{x})\approx 0$. In other words, we take our approximate zero for $F$ to be an approximate zero for the polynomial part of $F$. Now, we define $A^{\dagger}$ to be the matrix obtained by 
evaluating the formula $DF^N(\overbar{x})$ using double precision floating point arithmetic
(no interval enclosures), and let $A$ be any numerical inverse of $A^\dagger$.  Set (somewhat arbitrarily)
$r_* = 7 \times 10^{-3}$. 

Next, we have a lemma which allows us to control derivatives of bounded analytic functions on $B_1(0) \subset \cc^3$, by restricting to a smaller polydisc. 
\begin{lemma}
\label{lem:derivativebounds}
Suppose $g \in C^3_1$ is a bounded analytic function defined on $B_1(0)$. Then for any $\nu > 0$, we have the bounds
\begin{align*}
\norm{Dg}_{e^{-\nu}} & \leq \frac{6 \pi}{\nu} \norm{g}_{1} \\
\norm{D^2g}_{e^{-\nu}} & \leq \frac{36 \pi^2}{\nu^2} \norm{g}_{1} 
\end{align*}
\end{lemma}
\noindent A proof of this lemma for arbitrary dimension can be found in \cite{MR3068557}. With this in hand, set 
\[
\nu = -\ln( \norm{\overbar{x}}_{\infty} + r_*) = 1.9028 
\]
and note that $B_{r_*}(\overbar{x}) \subseteq B_{e^{-\nu}}(0)$. Now, we are prepared to compute the a-posteriori estimates required for Theorem \ref{prop:finiteDimNewton}. All computations below are carried out using interval arithmetic with floating point numbers regarded as degenerate intervals of the form $[x,x] \in \intrr$. 

\subsubsection*{$Y_0$}
We note that we have an error bound for $F^{\infty}$ given by $\norm{F^{\infty}}_1 \leq \norm{P^{\infty}}_1 + \norm{Q^{\infty}}_1$ which leads to the enclosure
\[
F(\overbar{x}) \in [F^N(\overbar{x}) - \norm{F^{\infty}}_1, F^N(\overbar{x}) + \norm{F^{\infty}}_1].
\]
From this estimate, we compute the enclosure $\norm{A \cdot F(\overbar{x})}_{\infty} \in [0,.0054]$ and we take $Y_0 = .0054$. 

\subsubsection*{$Z_0$}
Let $I$ denote the $3 \times 3$ identify matrix, then we explicitly compute
\[
\norm{I - AA^{\dagger}}_{\infty} \in [0,.1349 \times 10^{-14}].
\]
Hence, we take $Z_0 = .1349 \times 10^{-14}$. 

\subsubsection*{$Z_1$}
We decompose $DF(\overbar{x}) = DF^N(\overbar{x}) + DF^{\infty}(\overbar{x})$ and apply Lemma \ref{lem:derivativebounds} to obtain
\[
\norm{DF^{\infty}}_{e^{-\nu}} \leq \frac{6 \pi}{\nu}\norm{F^{\infty}}_{1} \leq  4.3739 \times 10^{-4}.
\]
It follows that 
\[
\sup\limits_{y \in B_{r_*}(\overbar{x})} \norm{DF^{\infty}(y)}_\infty \leq \norm{DF^{\infty}}_{e^{-\nu}} \leq 4.2815 \times 10^{-4}.
\]
Combining this result with an interval computation on the finite part we obtain 
\[
DF(\overbar{x}) \in 
[DF^N(\overbar{x}) - 4.2815 \times 10^{-4}, DF^N(\overbar{x}) + 4.2815 \times 10^{-4}], 
\]
and we use this interval enclosure  to compute
\[
\norm{A(A^{\dagger} - DF(\overbar{x})}_{\infty} \in [0,.1538].
\]
Hence, we set $Z_1 = .1538$. 	

\subsubsection*{$Z_2$}
Similarly, we decompose $DF(\overbar{x}) = DF^N(\overbar{x}) + DF^{\infty}(\overbar{x})$ and apply the second part of Lemma \ref{lem:derivativebounds} to obtain the enclosure
\[
\norm{D^2F^{\infty}}_{e^{-\nu}} \leq \frac{36 \pi^2}{\nu^2}\norm{F^{\infty}}_{1} \leq  .0043,
\]
which yields the bound. 
\[
\sup\limits_{y \in B_{r_*}(\overbar{x})} \norm{D^2F^{\infty}(y)}_{\infty} \leq \norm{D^2F^{\infty}}_{e^{-\nu}} \leq .0043.
\]
Turning to the finite part, we apply the formula given in Appendix \ref{sec:boundsAppendix} which yields
\begin{align*}
\sup_{y \in B_{r_*}(\overbar{x})} \norm{D^2F^N(y)}_\infty & =
 \sup_{y \in B_{r_*}(\overbar{x})} \max_{1 \leq i \leq 3} \sum_{j = 1}^{3} \sum_{k = 1}^3 \left |\partial_j \partial_k F_i^N(y) \right| \\
& \leq \max_{1 \leq i \leq 3} \sum_{j = 1}^{3} \sum_{k = 1}^3 \sup_{y \in B_{r_*}} \left |\partial_j \partial_k F_i^N(y) \right| \\
& = \max_{1 \leq i \leq 3} \sum_{j = 1}^{3} \sum_{k = 1}^3 \norm{\partial_j \partial_k F_i^N}_{B_{r_*}(\overbar{x})}.
\end{align*}
For each $i = 1,2,3$, the terms in the sum are computed with interval 
arithmetic and taking the max yields
\[
\sup_{y \in B_{r_*}(\overbar{x})}\norm{D^2F^N(y)}_Q \leq .0151.
\]
Taking these bounds together, we obtain the estimate 
\begin{align*}
\sup_{y \in B_{r_*}(\overbar{x})} \norm{D^2F(y)}_{\infty} & \leq \sup_{y \in B_{r_*}(\overbar{x})} \norm{D^2F^N(y)}_{\infty} + \sup_{y \in B_{r_*}(\overbar{x})} \norm{D^2F^{\infty}(y)}_{\infty} \\
& \leq .0043 + .0151 = .0194.
\end{align*}
Finally, a rigorous computation yields the bound $\norm{A}_\infty \leq 84.6195$ and we obtain the enclosure
\[
\norm{A}\cdot \sup_{y \in B_{r_*}} \norm{D^2F(y)} \in [2.3284,2.3414],
\]
and we set $Z_2 = 2.3414$. 

Finally, we verify that Theorem \ref{prop:finiteDimNewton} holds by computing the radii polynomial,
\[
p(r) = Z_2r^2 - (1-Z_0 - Z_1)r + Y_0 = 2.3414r^2 - .8462r + .0053
\]
which has roots given by $(r_-,r_+) = (.0064,.3550)$. Noting that $r_- < r_*$, we conclude that $p(r) < 0$ for all $r \in [.0064,.007]$. Hence, setting $r = r_- = .0064$, it follows from Theorem  \ref{prop:finiteDimNewton} that there exists a unique 
$\hat{x} = (\hat{s}, \hat{t}, \hat{\sigma}) \in B_r(\overbar{x})$ such that $F(\hat{x}) = 0$.

Next, we verify Lemma \ref{lem:aPosEnergyLemma1} holds. We consider the rectangle, $B_r(\overbar{s},\overbar{t}) \in \intrr^2$, and the interval $B_r(\overbar{\sigma}) \in \intrr$ so that via interval arithmetic we compute the enclosures
\begin{align*}
	\Gamma_4^u(B_r(\overbar{s},\overbar{t})) & \in [-.1629,   -.1619]  \\
\Gamma_4^s(B_r(\overbar{\sigma}),0) & \in [-.1633,   -0.1615] 
\end{align*}
and we conclude that $\Gamma_4^u(\hat{s}, \hat{t})$ has the same sign as $\Gamma_4^s(\hat{\sigma}, 0)$. Thus, by Lemma \ref{lem:aPosEnergyLemma1}, the orbit of $\Gamma^u(\hat s, \hat t)$ is a homoclinic for the saddle focus $\mathbf{x}_0$.


Finally, the transversality is verified using Lemma \ref{lem:aPosEnergyLemma2}.  
First note that invertibility of $DF(\hat{x})$ is a conclusion of Theorem \ref{prop:finiteDimNewton}.
Then the hypotheses of Lemma \ref{lem:aPosEnergyLemma2} are satisfied as soon as
$\nabla E(\Gamma^u(\hat s, \hat t)) \neq 0$. Note that it suffices to prove this for any coordinate. In particular, in 
the second component rigorously compute
\[
\pi_2 \circ \nabla E = \partial_2 E =  \Gamma_2^u(B_r(\overbar{x})) \in  [ -0.0747, -.0618 ],
\]
and $0 \notin [ -0.0747, -.0618 ]$. Thus, we conclude that $\Gamma_2^u(\hat s, \hat t) \neq 0$ and therefore, $\nabla E(\Gamma^u(\hat s, \hat t)) \neq 0$, proving that this homoclinic is transverse. Combining these results, we have proved Theorem \ref{thm:ourThm1} , 
namely that the orbit of $\Gamma^u(\hat s, \hat t)$ is homoclinic to $\mathbf{x}_0$ and transverse in the energy section. 
In other words, the hypothesis for Theorem \ref{thm:devaney} 
are satisfied and we conclude that the CRFBP is chaotic.


\subsection{Bounding transport times}
\label{sec:boundtransporttime}
The global parameterization of the (un)stable manifolds for the saddle-focus at $\mathbf{x_0}$ is the backbone of our proof of chaos. Other methods for validating homoclinic orbits have been taken up which would allow a proof of a transverse homoclinic with significantly less effort. However, our method based on computing global atlas' for the stable and unstable manifolds was chosen in large part because this method can also rule out connections. In other words, our interest in this paper is not only to prove that the CRFBP is chaotic, but also, to show that our method of verifying Devaney's theorem also allows one to rigorously bound the transport times for {\em all} homoclinic connections. This is due to the fact that we find and validate such connections by computing global paramaterizations of the (un)stable manifolds. It follows that the {\em lack} of an intersection between these manifolds after some integration time can serve equally well as a rigorous proof that any connection which does exist must take longer. In this section we make this more precise. 

Suppose $\Gamma^u,\Gamma^s$ are any pair of analytic charts for the unstable/stable manifolds of $\mathbf{x_0}$ respectively. Following our scheme outlined in Section \ref{sec:growingTheManifold}, we obtain $\Gamma^s$ in the form
\[
\Gamma^s(s,t) = \sum_{j = 0}^{M} \sum_{k = 0}^{N} b_{j,k} s^kt^j + h^s(s,t)
\]
where each $a_{j,k} \in \intrr$ and $h^u$ is analytic and satisfying $\norm{h^u}_1 \leq r^u$ for $r^u \in (0,\infty)$. Similarly, we obtain
\[
\Gamma^u(s,t) = \sum_{j = 0}^{M} \sum_{k = 0}^{N} a_{j,k} s^kt^j + h^u(s,t)
\]
where each $b_{j,k} \in \intrr$ and $h^s$ is analytic and satisfying $\norm{h^s}_1 \leq r^s$ for $r^s \in (0,\infty)$. Now, we define 
$F: [-1,1]^4 \to \intrr$ given by 
\[
F(s_1,t_1,s_2,t_2) = \Gamma^s(s_1,t_1) - \Gamma^u(s_2,t_2).
\]
Evidently, if $F$ has no roots on $[-1,1]$, then $\Gamma^s$ and $\Gamma^u$ are non-intersecting charts. The rigorous verification of this follows by evaluating the finite part of $F$ with interval arithmetic and padding by the interval $[-(r^u + r^s), (r^u + r^s)]$. We carried out this computation pairwise for stable and unstable charts contained in the atlases shown on the left side of Figure \ref{fig:CRFBP_connection} until a connection could not be ruled out. In fact, the first instance that this procedure fails occurs for a pair of charts for which a homoclinic connection exists. By rigorously ruling out all pairwise charts with shorter (combined) integration time, we obtain the following result.
\begin{prop}
Let $W^s_{\text{loc}}(\mathbf{x_0})$, $W^u_{\text{loc}}(\mathbf{x_0})$ be the validated local stable/unstable manifolds computed to order 8 as described in Section \ref{sec:localmanifolds} and suppose $\gamma$ is a homoclinic connection for $\mathbf{x_0}$ which exits $W^u_{\text{loc}}(\mathbf{x_0})$ at time $0$ and intersects $W^s_{\text{loc}}(\mathbf{x_0})$ after $t_f < 4.1437$ time units. Then, for some $0 \leq t \leq t_f$, we have $\sqrt{\gamma_2(t)^2 + \gamma_4(t)^2} > 1.7$. 

Equivalently, for every homoclinic to $\mathbf{x_0}$, if the {\em speed} along the orbit has the bound $\sqrt{\dot x^2 + \dot y^2} < 1.7$ for all time, then connection time is bounded below by $4.1437$ time units. 
\end{prop}

\begin{figure}[t!]
	\label{fig:CRFBP_numConnections}
	\includegraphics[width = 0.9\textwidth,keepaspectratio=true]{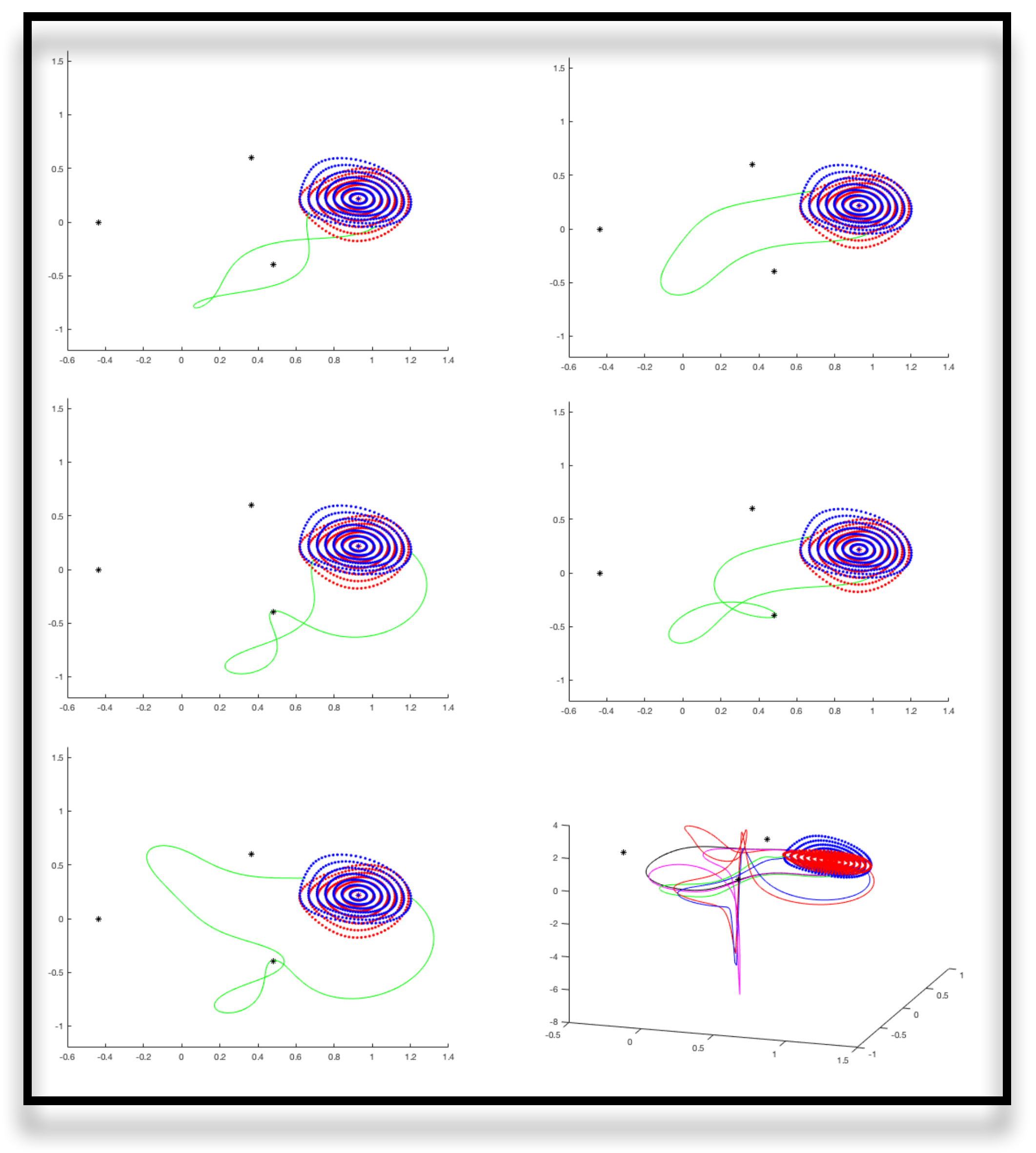}
	\caption{\textbf{Five numerical connecting orbits near the smallest primary:} the 
	top two frames illustrate connections which occur for moderately small velocities 
	(max velocity on the order of 1.5), but
	the velocities of the remaining orbits have comparatively large spikes
	(max closer to 8).  Note that the orbits in the middle two, and the bottom
	left frames pass near a primary.  The bottom right frame show all five orbits
	in the same figure, with the $y$ component of velocity as the third dimension.
        This last frame clearly illustrates the velocity spikes in the orbits.}
\end{figure}

\section{Conclusions} \label{sec:conclusions}
The methods of the present work yield computer assisted proof of energy section-transverse intersections
of the stable unstable manifolds of a saddle-focus equilibrium for the planar CRFBP
as desired, hence verrifying the hypotheses of the Devaney theorem and proving the existence of
chaotic dynamics in the problem.  

From here there are at least two very interesting 
possibilities for future extension of this work.
The first would be to explore the existence of Devaney homoclinic tangles at other parameter values.
Indeed, the parameter values of the present paper were chosen somewhat arbitrarily,
subject only to the constraint that they should avoid symmetries (non-equal masses)
and be non-perturbative (the masses are approximately the same size).  In truth, the precise 
mass values considered in the present work were chosen primarily because 
these mass values were studied numerically in
\cite{MR510556}, where they are used to illustrate
typical qualitative properties of the 
CRTBP in the asymmetric case.  

While the transversality results of the present work imply that the homoclinic 
connections persist for small enough changes in the masses of the primaries,
numerical simulations suggest that the homoclinic connections are typical.
Recent work in \cite{maximeBridge, meRon}
develop methods for mathematically 
rigorous computer assisted existence proofs of continuation
and bifurcation of connecting orbits.  These techniques could be used to study global branches
of homoclinic connections for the CRTBP over a large range of parameters.  
Indeed a very ambitious project along these lines would be to adapt the techniques
just mentioned to prove the existence of chaotic motions for all $m_1 + m_2 + m_3 = 1$
such that the CRFBP admits a saddle focus equilibrium.

Another interesting project will be to make a more exhaustive study of 
connecting orbits at a single set of parameter values, perhaps at the same 
parameter values used in the present work.
Numerical experiments conducted by the authors suggest that the homoclinic connecting orbits
in the problem are plentiful.  Yet many connecting orbits pass near one 
or more of the primary bodies, and these near collisions lead to large deviations in 
velocities.  Several such orbits passing near the smallest primary 
are illustrated in Figure \ref{fig:CRFBP_numConnections}.  
Large velocities make the validated Taylor computations used in the present work 
difficult, often eroding the validated bounds below useful limits or pushing 
the computational time beyond acceptable bounds.

One way to improve numerical outcomes near the collisions is to introduce regularized
coordinates near the primary or primaries, as a way of compactifying the phase space 
of the system.  This is a now standard technique for studying ejection/collision dynamics 
in celestial mechanics problems ever since the work of McGehee \cite{MR0495348, MR562695}.
Changing to regularized coordinates in the vicinity of the collision would lead to a problem 
with much better numerical properties, and allow use to prove the existence of 
connections passing near a collision.  

In fact, the introduction of regularized coordinates suggests 
even more interesting possibilities.  One effect of the compactifying transformation is 
that the finite time collision becomes a fixed point of parabolic stability type in the new 
coordinates.  See the works \cite{MR1085296, MR1342132, MR3111784}
for explicit results in the restricted three body problem.  
Numerical experiments  suggest that there are finite time collisions
between orbits on the stable/unstable manifold in the CRFBP, and 
these would become heteroclinic connections between the saddle-focus
and the regularized parabolic fixed point representing the collision.

Recently there has been a lot of interest in parameterization 
type method for invariant manifolds attached to objects of 
parabolic linear stability.   We refer the interested reader to the work of
\cite{MR2276478, MR2030148, MR3642259}.  If validated numerical 
methods were developed to study the error bounds 
associated with these parameterization methods, then it 
would be possible to use the techniques of the present work to 
prove the existence of transverse heteroclinic connections between 
the collision and the saddle-focus and back (heteroclinic cycles).  
These transverse cycles then lead to symbolic dynamics between the 
collision dynamics and the saddle-focus, providing a great deal of new
global dynamical information about the problem.  
This will make the topic of an upcoming study by the authors.

\section*{Acknowledgments}
The authors wish to thank Robert Devaney, Jan Bouwe van den Berg, 
J.P. Lessard, and Jaime Burgos for helpful conversations.
Both authors were partially supported by NSF grant DMS-1700154 
and by  the Alfred P. Sloan Foundation grant G-2016-7320 
during the work on this research.

\appendix

\section{Norms in $\mathbb{R}^d \backslash \mathbb{C}^d$} \label{sec:norms}
Let $V = \mathbb{R}^d$ or $V= \mathbb{C}^d$,
and write $\mathbf{v} = (v_1, \ldots, v_d)$ to denote an element of $V$.
In all numerical computations and the computer
assisted proofs in the present work endow $V$ with the max-norm 
\begin{equation} \label{eq:normDef}
\| v \|_\infty := \mbox{max}_{1 \leq j \leq d} \,  | v_j |,
\end{equation}
where $| \cdot |$ is the real or complex absolute value as appropriate.  
In fact, when we write $\| \cdot \|$ we always 
mean the max-norm, unless explicitly stated otherwise.  

Let $A$ be a $d \times d$ matrix and and write $\{a_{ij}\}_{1 \leq i,j \leq d}$
to denote the entries of $A$.  For $v \in V$ recall that the matrix-vector product
$Av \in V$  has components given by 
\[
(Av)_i = \sum_{j=1}^d a_{ij} v_j,
\]
for $1 \leq i \leq d$.  Define the \textit{matrix norm} 
\[
\| A \|_M := \max_{1 \leq i \leq d} \sum_{j = 1}^d |a_{ij}|.  
\]
The matrix norm above is in actually the norm on the space of linear operators 
from $V$ to $V$ induced by the max-norm (a fact which plays no further role 
in the present work).  
Observe that for any $v \in V$ we have the useful bound 
\begin{align}
\| A v \| &= \max_{1 \leq i \leq d}  \left| (Av)_i  \right|  \nonumber  \\
&=  \max_{1 \leq i \leq d}  \left|    \sum_{j=1}^d a_{ij} v_j   \right|   \nonumber \\
&\leq  \max_{1 \leq i \leq d}  \sum_{j=1}^d | a_{ij} | \,  |v_j |  \nonumber \\
&\leq    \max_{1 \leq i \leq d}  \sum_{j=1}^d | a_{ij}| \, \| v \|  \nonumber \\
&= \| A\|_M \| v \|.  \label{eq:matNormEst}
\end{align}

In a similar fashion, consider a $d \times d \times d$ \textit{matroid} $B$.
We write $\{b_{ijk}\}_{1 \leq i,j,k \leq d}$ do denote the entries or components of $B$.  
The matroid $B$ defines a $V$-valued bi-linear mapping on 
$V \times V$ with action given by the formula
\[
B(u,v)_i = \sum_{j = 1}^d \sum_{k = 1}^d b_{ijk} u_i v_j, 
\quad \quad \quad \quad u, v \in B.
\]
Here $B(u,v)_i$, for $1 \leq i \leq d$, are the components of $B(u,v) \in V$.
Define the \textit{matroid norm}
\[
\| B \|_Q := \max_{1 \leq i \leq d}  \sum_{j = 1}^d \sum_{k = 1}^d | b_{ijk} |,
\]
and observe that for any $u, v \in V$ we have the bound 
\begin{align}
\| B(u,v) \| &= \max_{1 \leq i \leq d} \left| B(u,v)_i \right| \nonumber \\
&= \max_{1 \leq i \leq d} \left|   \sum_{j = 1}^d \sum_{k = 1}^d b_{ijk} u_i v_j   \right| \nonumber \\
&\leq \max_{1 \leq i \leq d}    \sum_{j = 1}^d \sum_{k = 1}^d  | b_{ijk}| \, |u_i| \, |v_j|   \nonumber \\
&\leq \max_{1 \leq i \leq d}    \sum_{j = 1}^d \sum_{k = 1}^d  | b_{ijk}| \, \| u \| \, \| v \|   \nonumber \\
&= \| B \|_Q  \, \| u \| \, \| v \|.  \label{eq:quadNormEst} 
\end{align}
The subscript $Q$ on the matroid norm is suggestive of the fact that $B(u,u)$ is a 
``quadratic mapping,'' in the sense that 
\[
\| B(u,u) \| \leq C \| u \|^2, \quad \quad \quad \quad \mbox{for all } u \in V.
\]
Indeed, $C = \| B \|_Q$.

\section{Bounds on first and second derivatives} \label{sec:boundsAppendix}
Suppose that $U \subset V$ is an open set, 
$f \colon U \to V$ is a smooth map, and $\mathbf{v} \in V$.
We write $f = (f_1, \ldots, f_d)$ to denote the component maps.  
The first derivative of $f$ at $\mathbf{v}$ is the linear operator 
represented by the $d \times d$ \textit{Jacobian matrix}
$A = Df(\mathbf{v})$ with entries given by 
\[
a_{ij} =  \partial_j f_i(\mathbf{v}),     \quad \quad \quad \quad 1 \leq i, j \leq d, 
\]
and we have the explicit expression 
\begin{equation} \label{eq:jacobianBound}
\| Df(\mathbf{v})\|_M = \max_{1 \leq i \leq d} \sum_{j=1}^d \left|
\partial_j f_i(\mathbf{v}) 
\right|, 
\end{equation}
for the matrix norm of the derivative.

Similarly, the second derivative of $f$ at $\mathbf{v}$ is given by 
(what we might call the \textit{Hessian matroid})
$B = D^2f(\mathbf{v})$ with entries given by 
\[
b_{ijk} = \partial^2_{jk} f_i(\mathbf{v}), \quad \quad \quad \quad 1 \leq i, j, k \leq d, 
\]
and we have the expression  
\begin{equation} \label{eq:hessianEst}
\| D^2 f(\mathbf{v}) \|_Q = \max_{1 \leq i \leq d}  \sum_{j = 1}^d \sum_{k = 1}^d 
\left|
\partial^2_{jk} f_i(\mathbf{u})
\right|,
\end{equation}
for the matroid norm of the second derivative.
As an application of these ideas we have the following estimate. 
\begin{lemma} \label{lem:lipBoundDF}
{\em Let $r_* > 0$,
$\overline{\mathbf{u}} \in U \subset V$, and suppose that 
$\overline{B_{r_*}(\overline{\mathbf{u}})} \subset U$. 
Let $\mathbf{u}, \mathbf{v} \in \overline{B_{r_*}(\overline{\mathbf{u}})}$.
Then 
\begin{equation} \label{eq:2ndDerBound}
\| Df(\mathbf{u}) - Df(\mathbf{v}) \|_M  \leq \sup_{\mathbf{w} \in \overline{B_{r_*}(\overline{\mathbf{u}})}}  
\left\| D^2 f(\mathbf{w}) \right\|_Q \|u - v \|.
\end{equation}
}
\end{lemma}

\begin{proof}
Note that 
$f$ and all its partial derivatives are defined and continuous on 
an open set containing $\overline{B_{r_*}(\overline{\mathbf{u}})}$,
so that it makes sense to talk about partial derivatives defined in the 
closed ball.  Observe that the matrix $Df(\mathbf{u}) - Df(\mathbf{v})$ has 
entries
\[
\left[Df(\mathbf{u}) - Df(\mathbf{v}) \right]_{ij} = \partial_j f_i(\mathbf{u}) - 
\partial_j f_i(\mathbf{v}) =  \partial_j \left( f_i(\mathbf{u}) - 
f_i(\mathbf{v}) \right), 
\]
and that by the mean value theorem there is for each $1 \leq i, j \leq d$ 
a $\tilde{\mathbf{v}}_{ij} \in \overline{B_{r_*}(\overline{\mathbf{u}})}$ with 
\begin{itemize}
\item $\tilde{\mathbf{v}}_{ij}$ lies on the line segment connecting $\mathbf{u}$ to $\mathbf{v}$, 
\item and, 
\begin{align*}
 \frac{\partial}{\partial v_j} \left( f_i(\mathbf{u}) - 
f_i(\mathbf{v}) \right) &= \left< \nabla \partial_j f_i(\mathbf{v}_{ij}), \mathbf{u} - \mathbf{v} \right> \\
&= \sum_{k= 1}^d \partial_k \partial_{j} f_i(\mathbf{v}_{ij}) (u_k - v_k) \\
&= \sum_{k= 1}^d \partial^2_{jk} f_i(\mathbf{v}_{ij}) (u_k - v_k),
\end{align*}
by the equality of mixed partials.  
\end{itemize}
Then 
\begin{align*}
\| Df(\mathbf{u}) - Df(\mathbf{v}) \|_M &= \max_{1 \leq i \leq d} \sum_{j=1}^d 
\left|
\left[Df(\mathbf{u}) - Df(\mathbf{v}) \right]_{ij}
\right|  \\
&= \max_{1 \leq i \leq d} \sum_{j=1}^d 
\left|
 \sum_{k= 1}^d \partial^2_{jk} f_i(\mathbf{v}_{ij}) (u_k - v_k)
\right| 
\end{align*}
\begin{align*}
&\leq  \max_{1 \leq i \leq d} \sum_{j=1}^d 
 \sum_{k= 1}^d \left| \partial^2_{jk} f_i(\mathbf{v}_{ij}) \right| \| \mathbf{u} - \mathbf{v} \| 
\end{align*}
\begin{align*}
 & \leq \left( \max_{1 \leq i \leq d} \sum_{j=1}^d 
 \sum_{k= 1}^d \sup_{\mathbf{w}_{ij} \in \overline{B_{r_*}(\overline{u})}} 
 \left| \partial^2_{jk} f_i(\mathbf{w}_{ij}) \right|  \right)  \| \mathbf{u} - \mathbf{v} \| 
\end{align*}
\begin{align*}
&\leq \sup_{\mathbf{w} \in \overline{B_{r_*}(\overline{u})}}  \left( \max_{1 \leq i \leq d} \sum_{j=1}^d 
 \sum_{k= 1}^d 
 \left| \partial^2_{jk} f_i(\mathbf{w}) \right|  \right)  \| \mathbf{u} - \mathbf{v} \| \\
&= \sup_{\mathbf{w} \in \overline{B_{r_*}(\overline{u})}} \| D^2 f(\mathbf{w}) \|_Q \, 
 \| \mathbf{u} - \mathbf{v} \|,
\end{align*}
as desired.

\end{proof}

\section{Some explicit formulas for the CRFBP}  \label{sec:CRFBP_basics}
Let
\[
K = m_2(m_3 - m_2) + m_1(m_2 + 2 m_3).
\]
The precise positions of the primaries are given by 
the locations of the primaries are denoted by   
\[
p_1 = (x_1, y_1), \quad \quad 
p_2 = (x_2, y_2), \quad \quad
\text{and} \quad \quad
p_3 = (x_3, y_3),
\]
where we have
\begin{align*}
x_1 &=   \frac{-|K| \sqrt{m_2^2 + m_2 m_3 + m_3^2}}{K},  \\
y_1 &=   0  
\end{align*}
\begin{align*}
x_2 &=  \frac{|K|\left[(m_2 - m_3) m_3 + m_1 (2 m_2 + m_3)  \right]}{
2 K \sqrt{m_2^2 + m_2 m_3 + m_3^2} }  \\
y_2 & =  \frac{-\sqrt{3} m_3}{2 m_2^{3/2}} \sqrt{\frac{m_2^3}{m_2^2 + m_2 m_3 + m_3^2}}
\end{align*}
\begin{align*}
x_3 &=  \frac{|K|}{2 \sqrt{m_2^2 + m_2 m_3 + m_3^2}}  \\
y_3 &=  \frac{\sqrt{3}}{2 \sqrt{m_2}} \sqrt{\frac{m_2^3}{m_2^2 + m_2 m_3 + m_3^2}}.
\end{align*}

Let $\Omega \colon U \to \mathbb{C}$ be the effective potential function defined 
by Equation \eqref{eq:CRFBP_potential}.
Then
\[
\frac{\partial}{\partial x} \Omega = 
 x - \frac{m_1(x - x_1)}{r_1(x,y)^3}
- \frac{m_2(x - x_2)}{r_2(x,y)^3} - \frac{m_3(x - x_3)}{r_3(x,y)^3},
\]
and
\[
\frac{\partial}{\partial y} \Omega = 
 y - \frac{m_1(y - y_1)}{r_1(x,y)^3}
- \frac{m_2(y - y_2)}{r_2(x,y)^3} - \frac{m_3(y - y_3)}{r_3(x,y)^3},
\]
are the terms appearing explicitly in the vector field $f$ defined in
Equation \eqref{eq:SCRFBP}.

To study the derivative of the vector field $f$ we need the 
partial derivatives of $\Omega_{x,y}$, and these are given by the quantities
\[
g_{11}(x,y) :=  \frac{\partial}{\partial x} \Omega_x(x,y) = 
\]
\[
1 + \frac{m_1(2(x - x_1)^2 - (y - y_1)^2)}{\sqrt{(x - x_1)^2 + (y - y_1)^2}} 
+  \frac{m_2(2(x - x_2)^2 - (y - y_2)^2)}{\sqrt{x - x_2)^2 + (y - y_2)^2}}
+ \frac{m_3(2(x - x_3)^2 - (y - y_3)^2)}{\sqrt{(x - x_3)^2 + (y - y_3)^2}},
\]
\[
g_{12}(x,y) = g_{21}(x,y) := \frac{\partial}{\partial y} \Omega_x(x,y) =  \frac{\partial}{\partial x} \Omega_y(x,y) = 
\]
\[
\frac{3 m_1(x - x_1)(y - y_1)}{\sqrt{(x - x_1)^2 + (y - y_1)^2}} +
\frac{3 m_2(x - x_2)(y - y_2)}{\sqrt{(x - x_2)^2 + (y - y_2)^2}}  +
\frac{3 m_3(x - x_3)(y - y_3)}{\sqrt{(x - x_3)^2 + (y - y_3)^2}}, 
\]
and
\[
g_{22}(x,y) :=  \frac{\partial}{\partial y} \Omega_y(x,y)  = 
\]
\[
1 + \frac{m_1(2(y - y_1)^2 - (x - x_1)^2)}{\sqrt{(x - x_1)^2 + (y - y_1)^2}} 
+  \frac{m_2(2(y - y_2)^2 - (x - x_2)^2)}{\sqrt{(x - x_2)^2 + (y - y_2)^2}}
+ \frac{m_3(2(y - y_3)^2 - (x - x_3)^2)}{\sqrt{(x - x_3)^2 + (y - y_3)^2}}.
\]
When there is no possibility of confusion we sometimes 
suppress the $(x,y)$ dependence and simply write $g_{ij}$.  

Computer assisted proofs based on Theorem (REF) require bounds on the second derivative
of the vector field $f$, and of course this requires us to consider third derivatives of $\Omega$.
We record that the second partials of $\Omega_{x,y}$ are given by the expressions  
\begin{align*}
\frac{\partial^2}{\partial x^2} \Omega_x(x,y) = 
\sum_{j = 1}^3 \left(
\frac{4 m_j (x - x_j)}{r_j(x,y)^5} - 
\frac{5 m_j (x-x_j)(2 (x-x_j)^2 - (y-y_j)^2)}{r_j(x,y)^7}
\right),
\end{align*}

\begin{align*}
\frac{\partial^2}{\partial y \partial x} \Omega_x(x,y) = 
- \sum_{j = 1}^3 
\left(\frac{2 m_j (y-y_j)}{r_j(x,y)^5} + 
\frac{5 m_j (y-y_j)(2 (x-x_j)^2 - (y-y_j)^2)}{r_j(x,y)^7}  \right),
\end{align*}

\begin{align*}
\frac{\partial^2}{\partial x \partial y} \Omega_x(x,y) =  \frac{\partial^2}{\partial y \partial x} \Omega_x(x,y),
\end{align*}

\begin{align*}
\frac{\partial^2}{\partial x^2} \Omega_y(x,y) = 
\frac{\partial^2}{\partial x \partial y} \Omega_x(x,y),
\end{align*}

\begin{align*}
\frac{\partial^2}{\partial x \partial y} \Omega_y(x,y) = 
 \sum_{j = 1}^3 \left(
\frac{5 m_j (x-x_j)(2(y-y_j)^2 - (x-x_j))}{r_j(x,y)^7} - 
\frac{2 m_j (x-x_j)}{r_j(x,y)^5} \right),  
\end{align*}

\begin{align*}
\frac{\partial^2}{\partial y^2 } \Omega_x(x,y) = \frac{\partial^2}{\partial y \partial x} \Omega_x(x,y),
\end{align*}

\begin{align*}
\frac{\partial^2}{\partial y \partial x} \Omega_y(x,y) = 
\frac{\partial^2}{\partial y^2} \Omega_x(x,y),
\end{align*}

and
\begin{align*}
\frac{\partial^2}{\partial y^2 } \Omega_y(x,y) = 
 \sum_{j = 1}^3  \left(
\frac{4 m_j (y-y_j)}{r_j(x,y)^5} - 
\frac{5 m_j (y-y_j)(2(y-y_j)^2 - (x-x_j)^2)}{r_j(x,y)^7} \right).  
\end{align*}

By inspecting Equation \eqref{eq:SCRFBP} we see that an 
equilibrium solution of $f$ must have that $\dot x = \dot y  = 0$. 
Then Jacobian matrix of $f$ is  
\[
Df(x, 0, y, 0) = 
\left(
\begin{array}{cccc}
   0   &   1     &  0       &  0     \\
   g_{11}   &   0     &  g_{12}       &   2    \\
   0   &  0      &  0       &   1    \\
  g_{21}    &   -2     &  g_{22}       &  0     \\
\end{array}
\right),
\]
so the characteristic polynomial is 
\begin{equation} \label{eq:charEq_CRFBP}
\mbox{det}\left(Df(x, 0, y, 0) - \lambda \mbox{Id} \right) = 
\lambda^4 + \lambda^2 (4 - g_{11} - g_{22}) + (g_{11} g_{22} - g_{12}^2). \nonumber 
\end{equation}
Solving the equation above, we see that the eigenvalues are 
\begin{equation} \label{eq:CRFBP_eigs}
\lambda_{1,2,3,4} = \pm  \sqrt{\frac{-(4 - g_{11} - g_{22}) \pm \sqrt{(4 - g_{11} - g_{22})^2 - 4 (g_{11} g_{22} - g_{12}^2)}}{2}}.
\end{equation}
The following Lemma provides the associated eigenvectors.

\begin{lemma} \label{lem:eigVectLemma}  {\em
Let $\lambda \in \mathbb{C}$ be a nonzero eigenvalue for $Df(\mathbf{x}_0)$, 
and define the vectors 
\[
v_1 = \left(
\begin{array}{c}
1 \\
\lambda \\
0 \\
0
\end{array}
\right), 
\quad \quad \quad \mbox{and} \quad \quad \quad 
v_2 = \left(
\begin{array}{c}
0 \\
0 \\
1 \\
\lambda
\end{array}
\right).
\] 
Let $s \in \mathbb{C}$, with $s \neq 0$, and define 
\[
r = - s \frac{g_{12} + 2 \lambda}{g_{11} - \lambda^2}.
\]
Then  
\[
\xi = r v_1 + s v_2, 
\]
is an eigenvector for $Df(\mathbf{x}_0)$ associated with the eigenvalue $\lambda$.}
\end{lemma}
\begin{proof}
Observe that if $\xi = r v_1 + s v_2$ then  the equation 
\[
\left[Df(\mathbf{x}_0) - \lambda \mbox{Id} \right] \xi = 
\left(
\begin{array}{cccc}
-\lambda & 1 & 0 & 0 \\
g_{11} & - \lambda  & g_{12} & 2 \\
0 & 0 & - \lambda & 1 \\
g_{12} & -2 & g_{22} & - \lambda
\end{array}
\right)
\left(
\begin{array}{c}
r \\
\lambda r \\
s \\
\lambda s
\end{array}
\right),
\]
is equivalent to the system of equations 
\begin{eqnarray*}
-s \lambda + s \lambda & = 0 \\
s g_{11} - \lambda^2 s + r g_{12} + 2r \lambda & = 0 \\
-r \lambda + r \lambda &= 0 \\
s g_{12} - 2 s \lambda + r g_{22} - r \lambda^2 &=0.
\end{eqnarray*}
Since the first and third equations are trivially satisfied by any $r, s \in \mathbb{C}$ 
the system reduces to $2 \times 2$ the homogeneous system  
\[
\left[
\begin{array}{cc}
g_{11} - \lambda^2 & g_{12} + 2 \lambda \\
g_{12} - 2 \lambda & g_{22} - \lambda^2 
\end{array}
\right]
\left(
\begin{array}{c}
r \\
s
\end{array}
\right)
= 
\left(
\begin{array}{c}
0 \\
0
\end{array}
\right).
\]
The equation has nontrivial solutions if and only if one of the the rows is a multiple of the other, 
which happens if and only if the determinant of the $2 \times 2$ coefficient matrix 
has zero determinant.   But we now have that  
\[
\mbox{det}\left(
\left[
\begin{array}{cc}
g_{11} - \lambda^2 & g_{12} + 2 \lambda \\
g_{12} - 2 \lambda & g_{22} - \lambda^2 
\end{array}
\right]
\right) = \lambda^4 + (4 - g_{11} - g_{22}) \lambda^2 + g_{11} g_{22} - g_{12}^2, 
\]
which is the characteristic equation for $Df(\mathbf{x}_0)$.  This determinant
is zero as $\lambda$ is an eigenvalue by hypothesis.  
Now it follows that second row of the $2 \times 2$ homogeneous
systems depends on the first.
We obtain a solution $(r, s)$ by taking $s$ free and choosing 
\[
r = - s \frac{g_{12} + 2 \lambda}{g_{11} - \lambda^2}.
\]
Adjusting the choice of $s$ determines the length of $\xi$.  
\end{proof}

\section{Recovering the CRFBP from the polynomial problem} \label{sec:polyFacts}
The formal series computations discussed so far focus on the
polynomial vector field $F \colon \mathbb{C}^7 \to \mathbb{C}^7$  
defined in Equation \eqref{eq:bigPoly}.  In this section we justify the use of 
the polynomial problem, and show how to recover the dynamics for the 
CRFBP.

Let $U \subset \mathbb{C}^4$ be as defined in Equation \eqref{eq:defDomain_U}, and
$f \colon U \to \mathbb{C}^4$ be the analytic vector field for the CRFBP
defined in Equation \eqref{eq:SCRFBP}.  
Define the nonlinear map $R \colon U  \to \mathbb{C}^7$
by 
\[
R(x, \dot x, y, \dot y) := 
\left(
\begin{array}{c}
x \\
\dot x \\
y \\
\dot y \\
\frac{1}{\sqrt{(x-x_1)^2 + (y - y_1)^2}} \\
\frac{1}{\sqrt{(x-x_2)^2 + (y - y_2)^2}} \\
\frac{1}{\sqrt{(x-x_3)^2 + (y - y_3)^2}}
\end{array}
\right).
\]

For $\mathbf{u} = (u_1, u_2, u_3, u_4, u_5, u_6, u_7) \in \mathbf{C}^7$ define 
the projections $\pi \colon \mathbb{C}^7 \to \mathbb{C}^4$ and $\pi^{\perp} \colon \mathbb{C}^7 \to \mathbb{C}^3$
by 
\[
\pi (\mathbf{u}) = 
\left(
\begin{array}{c}
u_1 \\
u_2 \\
u_3 \\
u_4
\end{array}
\right), 
\quad \quad \quad \mbox{and} \quad \quad \quad 
\pi^\perp (\mathbf{u}) = 
\left(
\begin{array}{c}
u_4 \\
u_5 \\
u_7
\end{array}
\right).
\]
For any $\mathbf{u} \in \mathbb{C}^7$ we have the decomposition
$\mathbf{u} = (\pi(\mathbf{u}), \pi^\perp(\mathbf{u}))$.

Observe that $R$ and $\pi$ satisfy the following identity
\[
\pi (R(\mathbf{x})) = \mathbf{x},
\]
for any $\mathbf{x} \in \mathbb{C}^4$.
Now define the set
\[
\mathcal{S} :=
\left\{ \mathbf{u} \in \mathbf{C}^7 \, | \,
\mathbf{u} = R(\mathbf{x}) \mbox{ for some }
\mathbf{x} \in U \subset \mathbb{C}^4 \right\}.
\]
That is, $\mathcal{S} = \mbox{image}(R) = R(U)$.
Observe that 
$\mathbf{u} \in \mathcal{S}$ if and only if 
\begin{equation} \label{eq:u_is_Rpi_u}
\mathbf{u} = R(\pi \mathbf{u}),
\end{equation}
an identity which shows that $R$ is one to one.
It is worth noting that $R$ is differentiable on $U$ and that 
\begin{equation} \label{eq:DiffR}
DR(\mathbf{x}) = \left(
\begin{array}{c}
\mbox{Id} \\
\nabla  r_1^{-1}(\mathbf{x}) \\
\nabla  r_2^{-1}(\mathbf{x}) \\
\nabla  r_3^{-1}(\mathbf{x})
\end{array}
\right), 
\quad \quad \quad \quad \mathbf{x} \in U.
\end{equation}
The importance of the function $R$ and the 
set $\mathcal{S}$ are illustrated by the following Lemma, 
which one proves by direct calculation.

\begin{lemma} \label{lem:theVectorFields}{\em
For all $\mathbf{x} \in U$ we have that 
\[
\pi F(R(\mathbf{x})) = f(\mathbf{x}).
\]
That is, the composition of $F$ with $R$ recovers the 
CRFBP field as its first four components.  }
\end{lemma}

\subsection{Conjugacy} \label{sec:autoDiffConj}
We are interested in the relationship between the dynamics generated by the 
vector fields $f$ and $F$.  The following identity will be useful.

\begin{lemma} \label{lem:autoDiffinfConj} {\em
For all $\mathbf{x} \in U$, 
\begin{equation} \label{eq:autoDiffInfConj}
DR(\mathbf{x}) f(\mathbf{x}) = F(R(\mathbf{x})).
\end{equation}}
\end{lemma}
\begin{proof}
First note that for any $\mathbf{x} \in U$ we have that  
\begin{align*}
 \pi F(R(\mathbf{x})) &= f(\mathbf{x}) \\
&= \mbox{Id} f(\mathbf{x}) \\
& = \pi DR(\mathbf{x})  f(\mathbf{x}), \\ 
\end{align*}
by the identity of Lemma \ref{lem:theVectorFields}, and the fact that 
$DR$ is the identity matrix in its first four components.

For the remaining components we have, by direct calculation, that 
\begin{align*}
F_{5,6,7}(R(\mathbf{x})) &= - (u_1 - x_{1,2,3}) u_2 u_{5,6,7}^3 - (u_3 - y_{1,2,3}) u_4 u_{5,6,7}^3 \left|_{\mathbf{u} = R(\mathbf{x})}  \right.\\  
&= - (x - x_{1,2,3}) \dot x \left( r_{1,2,3}^{-1} \right)^3 - (y - y_{1,2,3}) \dot y    \left( r_{1,2,3}^{-1} \right)^3 \\ 
&=   \left< - \left(r_{1,2,3}^{-1} \right)^3 ( x - x_{1,2,3}, 0, y - y_{1,2,3}, 0), f(\mathbf{x}) \right>  \\
&= \left< \nabla r_{1,2,3}^{-1}, f(\mathbf{x}) \right>,
\end{align*}
and, recalling the formula for the derivative of $R$ in Equation \eqref{eq:DiffR}, this gives 
\[
\pi^\perp F(R(\mathbf{x})) = \pi^{\perp} DR(\mathbf{x}) f(\mathbf{x}), 
\]
as desired.
\end{proof}

\begin{remark}[A general remark on automatic differentiation] \label{rem:whatIsAutoDiff}
We note that, in a more general treatment of automatic differentiation, Equation \eqref{eq:autoDiffInfConj}
would be a crucial part of the high level explanation of the method.  Geometrically Equation \eqref{eq:autoDiffInfConj}
says that the push forward of $f$ by 
$DR$ is equal to $F$ on $\mathcal{S}$, which explains why the dynamics of $F$ on $\mathcal{S}$ recover the dynamics of $f$.
In other words, the ``infinitesimal conjugacy'' of Equation \eqref{eq:autoDiffInfConj} is exactly what makes
the automatic differentiation work.  

So in general, the strategy of automatic differentiation is this: given a non-polynomial vector field $f$ defined 
on an open set $U \subset \mathbb{R}^d$ one looks for an embedding $R \colon U \to \mathbb{R}^D$
and a polynomial vector field $F \colon \mathbb{R}^D \to \mathbb{R}^D$ 
having that $\pi R = \mbox{Id}$ and that $f, R, F$ satisfy Equation \eqref{eq:autoDiffInfConj}.
Note that if $F$ and $R$ have the properties given in the last sentence then one immediately recovers 
the critical identity given in Lemma \ref{lem:theVectorFields}.  
If $f$ is composed of finite sums, products, and compositions of elementary functions then 
the informal procedure illustrated in Section \ref{sec:poly} always leads to appropriate $F$, and $R$.
This claim is best illustrated by considering a number of examples. For more discussion 
see \cite{mamotreto}.
\end{remark}

\bigskip

The following lemma makes the preceding geometric remarks more precise.

\begin{lemma}[Solution curves of $f$ lift to solution curves of $F$] \label{lem:lift_of_a_solution} {\em
Let $\mathbf{x} \in U$ and suppose that $\gamma \colon (-T, T) \to \mathbb{C}^4$ 
is a solution of the initial value problem $\gamma' = f(\gamma)$ with 
$\gamma(0) = \mathbf{x}$.
Then the curve $\Gamma \colon (-T,T) \to \mathbb{C}^7$ with 
\[
\Gamma(t) := R(\gamma(t)), 
\]
solves the initial value problem for $\Gamma' = F(\Gamma)$ 
with $\Gamma(0) = R(\mathbf{x})$. }
\end{lemma}
\begin{proof}
Assume that $\gamma \colon (-T, T) \to \mathbb{C}^4$ is a solution of the 
initial value problem $\gamma' = f(\gamma)$, with $\gamma(0) = \mathbf{x}$ where $\mathbf{x} \in U$.
It follows that $\gamma((-T, T)) \subset U$, so that $\gamma(t) \in \mbox{dom}(R)$
for all $t \in (-T,T)$.
Then the curve $\Gamma \colon (-T, T) \to \mathbb{C}^7$ given by the expression 
\[
\Gamma(t) = R(\gamma(t)),
\]
is well defined.  Note that 
\[
\Gamma(0) = R(\gamma(0)) =  R(\mathbf{x}).
\]
Taking the time derivative of $\Gamma$ gives 
\begin{align*}
\frac{d}{dt} \Gamma(t) &= \frac{d}{dt} R(\gamma(t)) \\
&= DR(\gamma(t)) \gamma'(t) \\
&= DR(\gamma(t)) f(\gamma(t)) \\
&=  F(R(\gamma(t))) \\
&= F(\Gamma(t)),
\end{align*}
where we used the infinitesimal conjugacy of Lemma \ref{lem:autoDiffinfConj} 
to pass from the third to the fourth line.
This calculation shows that $\Gamma(t)$ solves the initial value problem $\Gamma' = F(\Gamma)$
with $\Gamma(0) = R(\mathbf{x})$, and since $\mathbf{x} \in U$ was 
arbitrary we have the claim.
\end{proof}

The following Lemma gives us some information about \textit{any} solution curve of the vector 
field $F$.  Not just solutions on the image of $R$.

\begin{lemma} \label{lem:formOfAutoDiffSol}  {\em
Suppose that $\Gamma \colon (-T, T) \to \mathbb{C}^7$ is a
smooth solution of the differential equation $\Gamma' = F(\Gamma)$.
Assume that $\Gamma_1(t) \neq x_{1,2,3}$ and $\Gamma_3(t) \neq y_{1,2,3}$
for all $t \in (-T, T)$.
Then there are $C_1, C_2, C_3 \in \mathbb{C}$ so that 
\[
\Gamma_{5,6,7}(t)  = \frac{1}{\sqrt{(\Gamma_1(t) - x_{1,2,3})^2 + (\Gamma_3(t) - y_{1,2,3})^2 + C_{1,2,3}}}.
\]
}
\end{lemma}
\begin{proof}
Let $a, b \in \mathbb{R}$ and 
suppose that $x, y \colon (-T, T) \to \mathbb{C}$ are smooth functions.  
Consider the differential equation 
\begin{equation} \label{eq:gODEform}
g'(t) = \left( \frac{-(x(t) - a)^2 - (y(t)- b)^2}{2} \right)' g(t)^3.
\end{equation}
If $x(t) \neq a$ and $y(t) \neq b$ for all $t \in (-T, T)$ then the equation has  the one parameter 
family of solutions 
\[
g(t) = \frac{1}{\sqrt{(x(t) - a)^2 + (y(t) - b)^2 + C}},
\]
as can be checked by direct substitution.  Now, if $\Gamma(t)$ is a 
solution curve for the vector field $F$ then we have that 
\[
\frac{d}{dt} \Gamma_{5,6,7}(t)  
\]
\begin{align*}
  &= -\Gamma_1(t) \Gamma_2(t) \Gamma^3_{5,6,7}(t) + x_{1,2,3} \Gamma_2(t) \Gamma^3_{5,6,7}(t)
   -\Gamma_3(t) \Gamma_4(t) \Gamma^3_{5,6,7}(t) + y_{1,2,3} \Gamma_4(t) \Gamma^3_{5,6,7}(t) \\
   &= -\left(\Gamma_1(t) - x_{1,2,3}\right) \Gamma_2(t) \Gamma^3_{5,6,7}(t)
   - \left(\Gamma_3(t) - y_{1,2,3}\right) \Gamma_4(t) \Gamma^3_{5,6,7}(t) \\
   &= \left[ -\left(\Gamma_1(t) - x_{1,2,3}\right) \Gamma_1'(t) 
   - \left(\Gamma_3(t) - y_{1,2,3}\right) \Gamma_3'(t) \right] \Gamma^3_{5,6,7}(t) \\
   &=  \left[  \frac{-\left(\Gamma_1(t) - x_{1,2,3}\right)^2  
   - \left(\Gamma_3(t) - y_{1,2,3}\right)^2}{2} \right]' \Gamma^3_{5,6,7}(t).
\end{align*}
In other words, $\Gamma_{5,6,7}$ satisfy Equation \eqref{eq:gODEform} with $a = x_{1,2,3}$, $b = y_{1,2,3}$,
$x(t) = \Gamma_1(t)$, and $y(t) = \Gamma_3(t)$.  The result follows by uniqueness.
\end{proof}

It follows from the previous Lemma that, as long as there are no collisions, $\mathcal{S}$ is an invariant set.

\begin{cor}[$F$-invariance of $\mathcal{S}$] \label{lem:IVPpoly} {\em
Let $T > 0$ and $\Gamma \colon (-T, T) \to \mathbb{C}^7$ be a smooth 
solution of the differential equation $\Gamma' = F(\Gamma)$.
Assume that  
$\Gamma_1(t) \neq x_{j}$ and $\Gamma_3(t) \neq y_{j}$ for $j = 1,2,3$ and all $t \in (-T,T)$,
and that 
\[
\Gamma(0) \in \mathcal{S}.
\]
Then 
\[
\Gamma(t) \in \mathcal{S}, 
\]
for all $t \in (-T, T)$.}
\end{cor}

\begin{proof}
By Lemma \ref{lem:formOfAutoDiffSol} we have that 
\begin{equation} \label{eq:someProofEq1}
\Gamma_{5,6,7}(t)  = \frac{1}{\sqrt{(\Gamma_1(t) - x_{1,2,3})^2 + (\Gamma_3(t) - y_{1,2,3})^2 + C_{1,2,3}}}.
\end{equation}
for some $C_{1,2,3} \in \mathbb{C}$. But $\Gamma(0) \in \mathcal{S}$, so that $\Gamma(0) \in \mbox{image}(R)$, 
or  
\[
\Gamma_{5,6,7}(0) = \frac{1}{\sqrt{(\Gamma_1(0) - x_{1,2,3})^2 + (\Gamma_3(0) - y_{1,2,3})^2 }}.
\]
Then
\[
C_{1,2,3} = 0.
\]
Substituting these values back into Equation \eqref{eq:someProofEq1} now shows that 
$\Gamma(t) \in \mbox{image}(R) = \mathcal{S}$ for all $t \in (-T,T)$. 
\end{proof}

Now we have the following partial converse of Lemma \ref{lem:lift_of_a_solution}, which says
that -- when restricted to the invariant set $\mathcal{S}$ -- 
the $\pi$ projection of an orbit of $F$ 
recovers and orbit of the CRFBP.

\begin{lemma}[Projections of solution curves of $F$ on $\mathcal{S}$ are solution curves for $f$] \label{prop:autoDiff_IVP} {\em
Suppose that $\Gamma \colon (-T, T) \to \mathbb{C}^7$
is a solution of the initial value problem $\Gamma' = F(\Gamma)$ having 
$\Gamma(0) = \mathbf{u} \in \mathcal{S}$.  Suppose in addition that   
$\Gamma_1(t) \neq x_{j}$ and $\Gamma_3(t) \neq y_j$ for $j = 1,2,3$ and 
 for all $t \in (-T, T)$.  Then 
\[
\gamma(t) = \pi \Gamma(t), 
\]
is the solution of the CRFBP with initial conditions $\gamma(0) = \pi \mathbf{u}$.  }
\end{lemma}

\begin{proof}
Since there are no collisions on $(-T, T)$ we have that $\gamma(t) = \pi \Gamma(t) \in U$ for 
all $t \in (-T, T)$.  
Moreover, since $\Gamma(0) = \mathbf{u} \in \mathcal{S}$ (and there are no collisions) 
we have that $\Gamma(t) \in \mathcal{S}$ for all $t \in (-T, T)$
by Lemma \ref{lem:IVPpoly}.  
Now, recalling that Equation \eqref{eq:u_is_Rpi_u} holds on 
$\mathcal{S}$, we have that for all $t \in (-T, T)$ 
\[
\Gamma(t) = R(\pi \Gamma(t)).
\]
Exploiting the identity above in conjunction with Lemma \ref{lem:theVectorFields} 
gives that
\begin{equation} \label{eq:IVP_eq1}
f(\gamma(t)) = \pi F(R(\gamma(t))) = \pi F(R(\pi \Gamma(t))) = \pi F(\Gamma(t)).
\end{equation}

Now, since the derivative commutes with the projection operator, and since 
$\Gamma$ is a trajectory for $F$ by hypothesis, we have that   
\begin{align}
\frac{d}{dt} \gamma(t) &= \frac{d}{dt} \pi \Gamma(t)  \nonumber \\
&= \pi \frac{d}{dt} \Gamma(t) \nonumber \\
&= \pi F(\Gamma(t)).    \label{eq:IVP_eq2}
\end{align}
Combining Equation \eqref{eq:IVP_eq1} with Equation \eqref{eq:IVP_eq2} gives that 
\[
\frac{d}{dt} \gamma(t) = \pi F(\Gamma(t)) =  f(\gamma(t)).
\]
Combining the equation above with the fact that
\[
\gamma(0) = \pi \Gamma(0) = \pi \mathbf{u},
\]
we have that $\gamma$ solves the desired initial value problem.
\end{proof}

The preceding results combine to give the conjugacy result for our automatic differentiation scheme.   
Let $\Phi$ denote the flow generated by $F \colon \mathbb{C}^7 \to \mathbb{C}^7$ and $\Psi$
be the flow generated by $f \colon U \to \mathbb{C}^4$.  Note that there are some initial 
conditions for which the flows are not defined for all times, but only when there 
are collisions.

\begin{prop}[Flow conjugacy for the automatic differentiation] \label{prop:autoDiffConj}  {\em
Let $\mathbf{x} \in U$.  Then   
\begin{equation} \label{eq:autoDiffConj}
\Phi(R(\mathbf{x}), t) = R \left(\Psi(\mathbf{x}, t) \right),
\end{equation}
for all $t$ such that there is no collision.}
\end{prop}
\begin{proof}
Let $\mathbf{x} \in U$.  Then $\mathbf{u} := R(\mathbf{x}) \in \mathcal{S}$ by definition.
By Lemma \ref{lem:IVPpoly}, we have that $\Phi(\mathbf{u}, t) \in \mathcal{S}$ for all $t$ such that 
there is no collision, and by Lemma \ref{prop:autoDiff_IVP} it follows that 
$\pi \Phi(\mathbf{u}, t)$ is the solution of the IVP for $f$ associated 
with the initial condition $\mathbf{x}$.  
This says that 
\begin{equation} \label{eq:anEquationInAProof}
\Psi(\mathbf{x}, t) = \pi \Phi(R(\mathbf{x}), t).
\end{equation}
Now since $\mathcal{S}$ is invariant it follows from the identity 
of Equation \eqref{eq:u_is_Rpi_u} that 
\[
R(\pi \Phi(R(\mathbf{x}), t)) = \Phi(R(\mathbf{x}), t).
\]
Combining the expression above with Equation \eqref{eq:anEquationInAProof}
gives
\[
\Phi(R(\mathbf{x}), t) = R(\pi \Phi(R(\mathbf{x}), t)) = R(\Psi(\mathbf{x}, t)),
\]
as desired.

Suppose on the other hand we begin with the same $\mathbf{x} \in U$ and consider the orbit 
$\Psi(\mathbf{x}, t)$, which exists for some open interval of time.  Then by Lemma \ref{lem:lift_of_a_solution} we 
have that $R(\Psi(\mathbf{x}, t))$ is a solution of the IVP for $F$ associated with the initial 
condition $R(\mathbf{x})$.
Combining these observations leads back to the identity
\[
R(\Psi(\mathbf{x}, t)) = \Phi(R(\mathbf{x}), t).
\]  
So: the identity holds as long as one side or the other is defined.
\end{proof}

\begin{cor} \label{cor:equilibriaRelation} {\em
Suppose that $\mathbf{x}_0 \in U$ is an equilibrium point of $f$.  Then $\mathbf{u}_0 = R(\mathbf{x}_0)$
is an equilibrium point for $F$.}
\end{cor}
\begin{proof}
Recall that a point is a zero or equilibrium point for a locally Lipschitz vector field  if an only if 
it is fixed by the flow for all time.  Then,  
since $\mathbf{x}_0$ is an equilibrium point of $f$ we have that 
\[
\Psi(\mathbf{x}_0, t) = \mathbf{x}_0, 
\]
for all $t \in \mathbb{R}$.  The conjugacy relation of Equation \eqref{eq:autoDiffConj},
now gives that
\[
\Phi(\mathbf{u}_0, t) = \Phi(R(\mathbf{x}_0), t) = R \left(\Psi(\mathbf{x}_0, t) \right) = R(\mathbf{x}_0) = \mathbf{u}_0,
\]
for all $t \in \mathbb{R}$, hence $\mathbf{u}_0$ is an equilibrium as claimed.  
\end{proof}

\subsection{Linear stability} \label{sec:linearStability}
Suppose that $\mathbf{x}_0$ is an equilibrium for $f$ and let 
$\mathbf{u}_0 = R(\mathbf{x}_0)$ denote the associated equilibrium for 
$F$.  
We begin by deriving a useful relationship between the differentials of $f$ and $F$ 
at an equilibrium point.  First differentiate Equation \eqref{eq:autoDiffInfConj} with respect to $\mathbf{x}$ to get
\[
DF(R(\mathbf{x})) DR(\mathbf{x}) = D^2 R(\mathbf{x}) f(\mathbf{x}) + DR(\mathbf{x}) Df(\mathbf{x}).
\]
Here $D^2 R(\mathbf{x})$ is a bilinear mapping, and the 
notation $D^2 R(\mathbf{x}) f(\mathbf{x})$ means that one of the 
two arguments has  $f(\mathbf{x})$ fixed.
The result  $D^2 R(\mathbf{x}) f(\mathbf{x})$ a linear 
mapping and hence a matrix like everything else in the expression.  
Now, using the fact that 
$\mathbf{x}_0 \in \mathbb{C}^4$ is an equilibrium for $f$ -- so that $\mathbf{u}_0 = R(\mathbf{x}_0)$
is an equilibrium for $F$ -- the identity above reduces to 
\begin{equation} \label{eq:helpful_linear_identity}
DF(R(\mathbf{x}_0)) DR(\mathbf{x}_0) = DR(\mathbf{x}_0) Df(\mathbf{x}_0).
\end{equation}

We now consider the eigenvalues and eigenvectors of $DF(\mathbf{u}_0)$ and $Df(\mathbf{x}_0)$.

\begin{lemma}[Zero eigenvalues of $DF(\mathbf{u}_0)$] \label{lem:zeroEigs} {\em
Zero is an eigenvalue of $DF(\mathbf{u}_0)$ with multiplicity at least three.  }
\end{lemma}
\begin{proof}
Note that $\lambda = 0$ is an eigenvalue for $DF(\mathbf{u}_0)$ if and only if $DF(\mathbf{u}_0)$ is 
not invertible, and the multiplicity of the eigenvalue $\lambda = 0$ is the dimension of the kernel 
of $DF(\mathbf{u}_0)$.  
Now observe for example that 
\[
DF_5(\mathbf{u}) = 
\]
\[
\left[
- u_2 u_5^3, -u_1 u_5^3 + x_1 u_5^3,  - u_4 u_5^3, -u_3 u_5^3 + y_1 u_5^3, 3 u_5^2 \left(- u_1 u_2  +  x_1 u_2 -  u_3 u_4 + y_1 u_4 \right), 0, 0
\right],  
\]
in general, but that at the equilibrium this reduces to 
\[
DF_5(\mathbf{u}_0) = 
 \left[
 0, -u_1 u_5^3 + x_1 u_5^3,   0, -u_3 u_5^3 + y_1 u_5^3, 0, 0, 0
\right],  
\]
as $u_2 = u_4 = 0$ in equilibrium.  
Noting that 
\[
DF_1(\mathbf{u}) = \left[
0, 1, 0, 0, 0, 0, 0
\right]
\]
and that 
\[
DF_3(\mathbf{u}) = \left[
0, 0, 0, 1, 0, 0, 0
\right], 
\]
we see that 
\[
DF_5(\mathbf{u}_0) = -(u_1 - x_1)u_5^3 DF_1(\mathbf{u}_0) - (u_3 - y_1) u_5^3 DF_3(\mathbf{u}_0),
\]
where $u_1 - x_1, u_3 - y_1$, and $u_5^3$ are all non-zero.
That is, the fifth row of $DF(\mathbf{u}_0)$ is a non-zero linear combination of rows one and three.
Then performing Gaussian elimination on $DF(\mathbf{u}_0)$ will result in a row of zeros.

A nearly identical argument shows that the sixth and seventh rows of $DF(\mathbf{u}_0)$ 
are non-zero linear combinations of rows one and three.  
It follows that the reduced row echelon form of $DF(\mathbf{u}_0)$ has at least three zero rows.
This implies that the rank of $DF(\mathbf{u}_0)$ is at most four, hence the kernel is at least three 
as claimed.  
\end{proof}

\begin{remark}[Zero surfaces of $F$] 
Viewed in light of Lemma \ref{lem:formOfAutoDiffSol}, the zero eigenvalue count above makes perfect sense.
Again, solutions of the equations
\[
u_{5,6,7}' = F_{5,6,7}(\mathbf{u}),
\]
come in one parameter families.  Only by fixing all three initial conditions do we obtain a unique 
solution.  Restricting to the set $\mathcal{S} = \mbox{image}(R)$ has the effect of 
selecting the constants $C_{1,2,3} = 0$ in Lemma \ref{lem:formOfAutoDiffSol}.  Moving normal 
to the surface $\mathcal{S}$ yields other zeros of $F$ corresponding to other values 
of the constants.  These solutions however have nothing to do with the dynamics of the CRFBP.
\end{remark}

\begin{lemma}[Non-zero eigenvalues of $Df(\mathbf{x}_0)$] \label{lem:nonZeroEigs} {\em
Suppose that $\lambda \in \mathbb{C}$ is a non-zero eigenvalue for $Df(\mathbf{x}_0)$ 
and let $\mathbf{\xi} \in \mathbb{C}^4$ denote an associated eigenvector.  
Define 
\[
\textbf{v} = DR(\mathbf{x}_0) \mathbf{\xi}.
\] 
Then $\lambda$ is an eigenvalue for $DF(\mathbf{u}_0)$ and $\mathbf{v} \in \mathbb{C}^7$
is an associated eigenvector.  }
\end{lemma}
\begin{proof}
We have 
\begin{align*}
DF(\mathbf{u}_0) \mathbf{v} &= DF(\mathbf{u}_0)  DR(\mathbf{x}_0) \mathbf{\xi} \\
&= DF(R(\mathbf{x}_0))  DR(\mathbf{x}_0) \mathbf{\xi} \\
&= DR(\mathbf{x}_0) Df(\mathbf{x}_0) \mathbf{\xi} \\
&= DR(\mathbf{x}_0) \left(\lambda \mathbf{\xi}\right) \\
& = \lambda  DR(\mathbf{x}_0)  \mathbf{\xi} \\
&= \lambda \mathbf{v},
\end{align*}
where we use that $\mathbf{x}_0$ is an equilibrium and the identity of Equation \eqref{eq:helpful_linear_identity}
to pass from linear two to line three in the calculation.
\end{proof}

\begin{lemma}[Eigenvectors associated with the zero eigenvalues and diagonalizibility] {\em
Let $\mathbf{u}_0 = (u_1, u_2, u_3, u_4, u_5, u_6, u_7) \in \mathbb{R}^7$ be an 
equilibrium solution of $F$ and assume that $DF(\mathbf{u}_0)$ has
four distinct, non-zero eigenvalues.  
If 
\[
1 - m_1 u_5^3 - m_2 u_6^3 - m_3 u_7^3 \neq 0.
\]
then $DF(\mathbf{u}_0)$ is diagonalizable.  
}
\end{lemma}
\begin{proof}
The Jacobian of $F$ decomposes into 
\begin{equation}\label{eq:polyJacobi}
DF(\mathbf{u}) = 
\left(
\begin{array}{cc}
D_1(\mathbf{u}) & D_2(\mathbf{u}) \\
D_3(\mathbf{u}) &  D_4(\mathbf{u})
\end{array}
\right)
\end{equation}
where 
\[
D_1(u_1, u_2, u_3, u_4, u_5, u_6, u_7) := 
\]
\[
\left(
\begin{array}{ccccccc}
0 & 1 & 0& 0  \\
    1 - m_1 u_5^3  - m_2 u_6^3  - m_3 u_7^3 & 0& 0& 2 \\
   0& 0& 0&  1 \\
    0& -2 &    1 - m_1 u_5^3  - m_2 u_6^3  - m_3 u_7^3   &  0      
 \end{array}
\right),
\]
\[
D_2(u_1, u_2, u_3, u_4, u_5, u_6, u_7) := \left(
\begin{array}{ccccccc}
 0& 0& 0\\
  - 3 m_1( u_1 - x_1) u_5^2 &  - 3 m_2( u_1 -  x_2) u_6^2&  - 3 m_3(u_1 -  x_3) u_7^2\\
  0& 0& 0\\
 - 3 m_1( u_3 - y_1) u_5^2   &  - 3 m_2( u_3 -y_2) u_6^2  & - 3 m_3( u_3 -y_3) u_7^2    
\end{array}
\right),
\]
\[
D_3(u_1, u_2, u_3, u_4, u_5, u_6, u_7) := 
\left(
\begin{array}{ccccccc}
  - u_2 u_5^3   & - (u_1  - x_1)  u_5^3   &    - u_4 u_5^3    &   -  (u_3  - y_1) u_5^3  \\
    - u_2 u_6^3   & - (u_1  - x_2)  u_6^3   &    - u_4 u_6^3    &   -  (u_3  - y_2) u_6^3  \\
     - u_2 u_7^3   & - (u_1  - x_3)  u_7^3   &    - u_4 u_7^3    &   -  (u_3  - y_3) u_7^3  
\end{array}
\right)
\] 
and 
\[
D_4(u_1, u_2, u_3, u_4, u_5, u_6, u_7) := 
\]
{\tiny
\[
\left(
\begin{array}{ccccccc}
   -3 (u_1 u_2 - x_1 u_2  +  u_3 u_4 - y_1 u_4) u_5^2  & 0& 0\\
  0   &  -3 (u_1 u_2 - x_2 u_2  +  u_3 u_4 - y_2 u_4) u_6^2  & 0\\
  0   & 0 &  -3 (u_1 u_2 - x_3 u_2  +  u_3 u_4 - y_3 u_4) u_7^2  
\end{array}
\right)
\]
}
Moreover, if $\mathbf{u}_0$ is an equilibrium point for $F$ then 
\[
D_3(\mathbf{u}_0)
= \left(
\begin{array}{ccccccc}
  0  & - (u_1  - x_1)  u_5^3   &    0   &   -  (u_3  - y_1) u_5^3  \\
    0   & - (u_1  - x_2)  u_6^3   &    0   &   -  (u_3  - y_2) u_6^3  \\
    0   & - (u_1  - x_3)  u_7^3   &   0    &   -  (u_3  - y_3) u_7^3  
\end{array}
\right), 
\]
and 
\[
D_4(\mathbf{u}_0) = 0,
\]
due to the fact that $u_2 = u_4 = 0$.

The proof of Lemma \ref{lem:zeroEigs} shows that the the last three rows 
are completely eliminated after three steps of Gaussian elimination.
Then $DF(\mathbf{u}_0)$ is row equivalent to 
\[
B = \left(
\begin{array}{cc}
D_1 & D_2 \\
0 & 0 
\end{array}
\right).
\]
By completing the Gaussian elimination, we see that 
$DF(\textbf{u}_0)$ has RREF  
\[
R = \left(
\begin{array}{cc}
\mbox{Id} & R_2 \\
0 & 0 
\end{array}
\right),
\] 
where 
\[
R_2 = 
\frac{1}{a}
\left(
\begin{array}{cccc}
-3m_1 (u_1 - x_1) u_5^2 & -3m_1 (u_1 - x_2) u_6^2 & -3m_1 (u_1 - x_3) u_7^2 \\
0  &  0  &  0 \\
-3m_1 (u_3 - y_1) u_5^2 & -3m_1 (u_3 - y_2) u_6^2 & -3m_1 (u_3 - y_3) u_7^2 \\
0 & 0 & 0 
\end{array}
\right)
\]
and
\[
a := 1 - m_1 u_5^3 - m_2 u_6^3 - m_3 u_7^3.
\]
From the RREF we see that the kernel of $R$, and hence the kernel of $DF(\mathbf{u}_0)$, 
is spanned by the three linearly independent vectors 
\begin{equation} \label{eq:basisKernel}
v_5 = \left(
\begin{array}{c}
\frac{3 m_1 (u_1 - x_1) u_5^2}{a} \\
0 \\
\frac{3 m_1 (u_3 - y_1) u_5^2}{a} \\
0 \\
1 \\
0 \\
0
\end{array}
\right),  \quad \quad 
v_6 = \left(
\begin{array}{c}
\frac{3 m_2 (u_1 - x_2) u_6^2}{a} \\
0 \\
\frac{3 m_2 (u_3 - y_2) u_6^2}{a} \\
0 \\
0 \\
1 \\
0
\end{array}
\right), \mbox{ and } 
v_7 = \left(
\begin{array}{c}
\frac{3 m_3 (u_1 - x_3) u_7^2}{a} \\
0 \\
\frac{3 m_3 (u_3 - y_3) u_7^2}{a} \\
0 \\
0 \\
0 \\
1
\end{array}
\right).
\end{equation}
Since we assumed that the 
remaining eigenvalues are non-zero and distinct,  
$DF(\mathbf{u}_0)$ has seven linearly independent eigenvectors -- 
hence is diagonalizable.  
\end{proof}

\subsection{The Parameterized Manifolds} \label{sec:parmPolyToCRFBP}
Take $\mathbf{x}_0 \in U \cap \mathbb{R}^4$ a real equilibrium of $f$ 
and $\mathbf{u}_0 = R(\mathbf{x}_0)$ the corresponding real equilibrium of $F$.
Suppose that $Df(\mathbf{x}_0)$ and 
hence $DF(\mathbf{u}_0)$ have four non-zero eigenvalues which we denote 
by $\lambda_1, \lambda_2, \lambda_3 \lambda_4 \in \mathbb{C}$.   
By Lemma \ref{lem:zeroEigs} the remaining
three eigenvalues of $DF(\mathbf{u}_0)$ are zero.  
In this paper we are interested in the symplectic saddle-focus case, so that 
\[
\lambda_{1,2} = -\alpha \pm i \beta,
\quad \quad \quad \quad \lambda_{3,4} = \alpha \pm i \beta,  
\]
for some $\alpha, \beta > 0$.

Let $\mathbf{\xi}_1, \mathbf{\xi}_2, \mathbf{\xi}_3, \mathbf{\xi}_4 \in \mathbb{C}^4$
be the corresponding eigenvalues of $Df(\mathbf{x}_0)$, so that by Lemma \ref{lem:nonZeroEigs}
\[
\mathbf{v}_j := R(\mathbf{x}_0) \mathbf{\xi}_j, 
\quad \quad \quad \quad j = 1,2,3,4
\]
are eigenvectors of $DF(\mathbf{u}_0)$ associated with $\lambda_1, \lambda_2, \lambda_3 \lambda_4$ 
respectively. Note that $\pi \mathbf{v}_j = \mathbf{\xi}_j$ for $j = 1,2,3,4$.
We have the following.  

\begin{lemma} \label{lem:imagePinS}{\em
Suppose that $P \colon D \to \mathbb{C}^7$ solves the invariance equation \eqref{eq:invEqnFlows1}, subject to the 
constraints $P(0) = \mathbf{u}_0$, $\partial_1 P(0) = \mathbf{v}_1$, and $\partial_2 P(0) = \mathbf{v}_2$.  Assume that 
$P_1(z_1, z_2) \neq x_{1,2,3}$ and $P_3(z_1, z_2) \neq y_{1,2,3}$ for $(z_1, z_2) \in D$.
Then 
\[
P(D) \subset \mathcal{S}.
\]
}
\end{lemma}

\begin{proof}
Choose $(z_1, z_2) \in D$ and define the curve $\Gamma \colon [0, \infty) \to \mathbf{C}^7$ 
\[
\Gamma(t) := P(e^{\lambda_1 t} z_1, e^{\lambda_2 t} z_2).
\]
By Lemma \ref{lem:flowConj} we have that $\Gamma$ is a solution 
of the initial value problem $\Gamma' = F(\Gamma)$ with $\Gamma(0) = P(z_1, z_2)$, and that 
\[
\lim_{t \to \infty} \Gamma(t) = \mathbf{u}_0.
\]
Since $P$ and hence $\Gamma$ have no collisions
by hypothesis,  it follows from Lemma \ref{lem:formOfAutoDiffSol} that  
\begin{equation} \label{eq:someEquation2}
\Gamma_{5,6,7}(t)  = \frac{1}{\sqrt{(\Gamma_1(t) - x_{1,2,3})^2 + (\Gamma_3(t) - y_{1,2,3})^2 + C_{1,2,3}}},
\end{equation}
for some $C_{1,2,3} \in \mathbb{C}$.  
But 
\begin{align*}
\lim_{t \to \infty} \Gamma_{5,6,7}(t) &= \lim_{t \to \infty}  \frac{1}{\sqrt{(\Gamma_1(t) - x_{1,2,3})^2 + (\Gamma_3(t) - y_{1,2,3})^2 + C_{1,2,3}}} \\
&=  \frac{1}{\sqrt{(  \lim_{t \to \infty} \Gamma_1(t) - x_{1,2,3})^2 + (  \lim_{t \to \infty} \Gamma_3(t) - y_{1,2,3})^2 + C_{1,2,3}}} \\
&=  \frac{1}{\sqrt{( \mathbf{u}_{1} - x_{1,2,3})^2 + (  \mathbf{u}_3 - y_{1,2,3})^2 + C_{1,2,3}}},
\end{align*}
and since $\Gamma(t) \to \mathbf{u}_0 \in \mathcal{S}$ as $t \to \infty$ it must be the case that $C_{1,2,3} = 0$.  
Plugging this back into Equation \eqref{eq:someEquation2} we see that 
\[
\Gamma(t) \in \mathcal{S}, 
\]
for all $t \geq 0$.  In particular, the inclusion holds at $t = 0$ and we have that 
\[
P(z_1, z_2) = \Gamma(0) \in \mathcal{S}.
\]
But $(z_1, z_2) \in D$ was arbitrary, giving the result.  
\end{proof}

The next lemma gives that the parameterization method applied to $F$ recovers the parameterization 
for $f$, and hence the local stable/unstable manifolds for the CRFBP.  

\begin{lemma}\label{lem:polyGivesCRFBP}
{\em
Suppose that $P \colon D \to \mathbb{C}^7$ solves the invariance equation \eqref{eq:invEqnFlows1}, subject to the 
constraints $P(0) = \mathbf{u}_0$, $\partial_1 P(0) = \mathbf{v}_1$, and $\partial_2 P(0) = \mathbf{v}_2$,
and define $p \colon D \to \mathbb{C}^4$ by 
\[
p(z_1, z_2) = \pi P(z_1, z_2),  \quad \quad \quad \quad (z_1, z_2) \in D.
\] 
Then $p$ parameterizes a local stable manifold attached to 
$\mathbf{x}_0$ for the CRFBP.  }
\end{lemma}

\begin{proof}
The idea is to use the parameterization method, but applied to the 
vector field $f$.  More precisely, let us show that $p$
is a solution of the invariance equation \eqref{eq:invEqnFlows1} for the CRFBP
subject to the appropriate first order constraints.  

Let $(z_1, z_2) \in D$, so that by Lemma \ref{lem:imagePinS} we have  
\[
P(z_1, z_2) \in  \mathcal{S} := \mbox{image}(R).
\]
Using this inclusion, and recalling the identity of Equation \eqref{eq:u_is_Rpi_u},
we have that 
\[
P(z_1, z_2) = R(\pi P(z_1, z_2)) = R(p(z_1, z_2)).
\]
Combining with Lemma \ref{lem:theVectorFields} gives
\begin{align} 
\pi F(P(z_1, z_2)) & =    \pi F(R(p(z_1, z_2)))   \nonumber \\
&= f(p(z_1, z_2)). \label{eq:parmProof_id}
\end{align}

Now, since $P$ satisfies Equation \eqref{eq:invEqnFlows1} for $F$, and since the 
projection operators are linear and commute with derivatives, we have that 
\begin{align*}
\lambda_1 z_1 \frac{\partial}{\partial z_1} p(z_1, z_2) + \lambda_2 z_2 \frac{\partial}{\partial z_2} p(z_1, z_2) &= 
\lambda_1 z_1 \frac{\partial}{\partial z_1} \pi P(z_1, z_2) +
\lambda_2 z_2 \frac{\partial}{\partial z_2} \pi P(z_1, z_2) \\
&= \lambda_1 z_1\pi \left( \frac{\partial}{\partial z_1}  P(z_1, z_2) \right) +
\lambda_2 z_2 \pi \left(\frac{\partial}{\partial z_2}  P(z_1, z_2) \right) \\
&= \pi \left(\lambda_1 z_1 \frac{\partial}{\partial z_1} P(z_1, z_2) +
\lambda_2 z_2 \frac{\partial}{\partial z_2}  P(z_1, z_2) \right) \\
&= \pi F[P(z_1, z_2)] \\
&= f(p(z_1, z_2)),
 \end{align*}
where we used Equation \eqref{eq:parmProof_id} to pass to the last line. 
Since $(z_1, z_2) \in D$ were arbitrary, this 
shows that $p$ satisfies Equation \eqref{eq:invEqnFlows1} on $D$.
 
Now we check that 
\[
p(0,0) = \pi P(0,0)  = \pi \mathbf{u}_0 = \pi R(\mathbf{x}_0) = \mathbf{x}_0,
\]
that 
\[
\frac{\partial}{\partial z_1} p(0,0) = \frac{\partial}{\partial z_1} \pi P(0,0) 
=  \pi \frac{\partial}{\partial z_1} P (0,0) =  \pi \mathbf{v}_1 = \mathbf{\xi}_1, 
\]
and that 
\[
\frac{\partial}{\partial z_2} p(0,0) = \frac{\partial}{\partial z_2} \pi P (0,0) 
= \pi \frac{\partial}{\partial z_2} P (0,0) =  \pi \mathbf{v}_2 = \mathbf{\xi}_2, 
\]
so that -- as claimed --$p$ satisfies the 
necessary first order constraints for the parameterization method
applied to $f$, and the proof is complete.  
\end{proof}

The proof of Lemma \ref{lem:polyGivesCRFBP} gives that $p = \pi P  \colon D \to \mathbb{C}^4$ 
satisfies the parameterization method for $f$, and it
follows that $p$ conjugates the CRFBP flow to the linear saddle-focus flow near the equilibrium.  
Moreover if $\lambda_1$ and $\lambda_2$ are complex conjugate eigenvalues, 
recall that by choosing $v_1$ and $v_2$ complex conjugate eigenvectors we
have that the image of $P(s_1 + i s_2, s_1 - i s_2)$ is real -- that is 
$P$ parameterizes the real stable manifold of $\mathbf{u}_0$ for $F$.
It follows that the image of $p(s_1 + i s_2, s_1 - i s_2) = \pi P(s_1 + i s_2, s_1 - i s_2)$
is also real, so that this is the parameterization of the real stable manifold of the CRFBP.
Of course analogous comments apply to the parameterization of the unstable manifold.

\bibliographystyle{unsrt}
\bibliography{papers}

\end{document}